\documentclass[11pt]{article}

\usepackage{amsmath}
\usepackage{amssymb}
\usepackage{amsthm}
\usepackage[dvipsnames]{xcolor}
\usepackage{tikz}
\usepackage{bm}
\usepackage{bbm}
\usepackage{color}
\usepackage{comment}
\usepackage{dsfont}
\usepackage{etoolbox}
\usepackage{fullpage}
\usepackage[utf8]{inputenc}
\usepackage{mathtools}
\usepackage{multirow} 
\usepackage{pgf}
\usepackage{pgfplots}
\usepackage{placeins} 
\usepackage{tcolorbox}
\usepackage[final]{pdfpages}
\usepackage{subcaption}

\usetikzlibrary{patterns}
\usepgfplotslibrary{groupplots}
\theoremstyle{plain}
\usepackage{amsthm}
\newtheorem{theorem}{Theorem}
\newtheorem{proposition}[theorem]{Proposition}
\newtheorem{corollary}[theorem]{Corollary}
\newtheorem{lemma}[theorem]{Lemma}

\theoremstyle{definition}

\theoremstyle{remark}

\numberwithin{equation}{section}
\numberwithin{theorem}{section}

\theoremstyle{plain}

\newtheorem*{assumptionS*}{Assumption S}

\newtheorem*{assumptionT*}{Assumption T}

\newtheorem*{assumptionU*}{Assumption U}

\newtheorem*{assumptionR*}{Assumption R}

\newtheorem*{assumptionP*}{Assumption P}

\newtheorem*{assumptionC1*}{Assumption C1}

\newtheorem*{assumptionC2*}{Assumption C2}

\newtheorem*{modelA*}{Model A}

\newtheorem*{modelB*}{Model B}

\def\bs{\boldsymbol}

\def\mcl{\mathcal}
\def\mbb{\mathbb}

\def\E{\mathop{}\!\mbb{E}}
\def\P{\mathop{}\!\mbb{P}}

\def\f{\frac}
\def\e{\mathrm{e}}
\def\es{\emptyset}
\def\su{\subseteq}
\def\co{\colon}

\def\N{\mathbb N}

\def\ti{\times}
\def\PP{\mcl P}

\renewcommand\le{\leqslant}
\renewcommand\ge{\geqslant}

\def\M{\mbb M}
\def\Q{\mbb Q}
\def\R{\mbb R}
\def\S{\mbb S}

\def\V{\mbb V}

\def\W{\mbb W}

\def\00{\bs 0}

\def\eps{\varepsilon}

\def\r{\rho}

\def\k{\kappa}

\def\bet{\begin{theorem}}
\def\ent{\end{theorem}}
\def\bepr{\begin{proposition}}
\def\enpr{\end{proposition}}
\def\bel{\begin{lemma}}
\def\enl{\end{lemma}}
\def\bec{\begin{corollary}}
\def\enc{\end{corollary}}
\def\bep{\begin{proof}}
\def\enp{\end{proof}}
\def\been{\begin{enumerate}}
\def\enen{\end{enumerate}}

\def\({\left(}
\def\){\right)}

\def\mc{\mathcal}

\def\one{\mathbbmss{1}}

\usepackage{enumitem}
\usepackage{hyperref}
\usepackage{fullpage}
\usepackage{setspace}
\usepackage[a4paper, width=15cm, height=24cm]{geometry}

\usepackage{xcolor}
\usepackage{titlesec}
\usepackage[toc,page]{appendix}
\usetikzlibrary{3d}

\definecolor{myblue}{RGB}{65, 105, 225}    
\definecolor{mygreen}{RGB}{34, 139, 34}

\usepackage{enumitem}

\newlist{assumptionAenum}{enumerate}{1}
\setlist[assumptionAenum,1]{
  label=\normalfont(\roman*),
  ref=A\arabic*,
  leftmargin=3.5em,
  align=left
}

\newcounter{assA}

\newenvironment{assumptionA}[1][]{
  \refstepcounter{assA}
  \ifx#1\empty\else\label{#1}\fi
  \par\noindent\textbf{Assumption~A.}\,
  \emph{Let $( W_n)_{n\ge1}$ be of the form in (\ref{eq:rectangle_form}) and
  let $\PP_n$ be a Poisson point process on $ \W_n$
  with intensity measure $|\cdot|\otimes\mu$. Assume the score $f$ is symmetric,
  non-negative, measurable, vanishes on the diagonal, and assume further that}
  \begin{assumptionAenum}
    \itshape
}{
  \end{assumptionAenum}
}

\newlist{assumptionBenum}{enumerate}{1}
\setlist[assumptionBenum,1]{
  label=\normalfont(\roman*),
  ref=B\arabic*,
  leftmargin=3.5em,
  align=left
}

\newcounter{assB}

\newenvironment{assumptionB}[1][]{
  \refstepcounter{assB}
  \ifx#1\empty\else\label{#1}\fi
  \par\noindent\textbf{Assumption~B.}\,
  \emph{Let $\PP_n$ be as in Assumption \ref{ass:A} and let $k < d$. Assume $g$ is symmetric, non-negative, measurable, vanishes on the diagonal, $\mathbb{N}_0$-valued and satisfies Assumption \ref{ass:A}(i) - \ref{ass:A}(iii). Assume further that}
  \begin{assumptionBenum}
    \itshape
}{
  \end{assumptionBenum}
}

\newcommand{\sectioncolor}{black}

\titleformat{\section}
  {\color{\sectioncolor}\normalfont\Large\bfseries}
  {\thesection}{1em}{}

\begin{document}

\title{Normal approximation of stabilizing Poisson pair functionals
with column-type dependence}

\author{Hanna Döring, Ad\'elie Garin, Christian Hirsch, and Nikolaj N. Lundbye}

\maketitle
\thispagestyle{empty}
\begin{spacing}{1.05}
\begin{abstract}
In this paper, we study two specific types of $d$-dimensional Poisson functionals:
 a \emph{double-sum} type and a \emph{sum-log-sum} type, both over pairs of Poisson points.
On these functionals, we impose column-type dependence, i.e., local behavior in the first $k$ directions and allow non-local, yet stabilizing behavior in the remaining $d-k$ directions.

The main contribution of the paper is to establish sufficient conditions for Normal approximation for sequences of such functionals over growing regions. 
Specifically, for any fixed region, we provide an upper bound on the Wasserstein distance between each functional and the standard Normal distribution.

We then apply these results to several examples.
 Inspired by problems in computer science, we prove a Normal approximation for the \emph{rectilinear crossing number}, 
 arising from projections of certain random graphs onto a 2-dimensional plane. From the field of topological data analysis, 
 we examine two types of barcode summaries, the \emph{inversion count} and the \emph{tree realization number}, and establish
  Normal approximations for both summaries under suitable models of the topological lifetimes. \newline \newline

\end{abstract}
\clearpage

\thispagestyle{empty}
\tableofcontents
\clearpage

\pagebreak 
	\setcounter{page}{1}
\section{Introduction}
\label{sec:intro}


\subsection{Background}
Central limit theorems are a cornerstone of probability theory, providing a rigorous explanation as to why Gaussian fluctuations appear in many large random systems.
Broadly speaking, they state that centered and normalized statistics of complex random structures converge in distribution to a Normal law as the system size grows.
While these results describe the limiting behavior, they do not say how quickly the distribution approaches its limit. 
This motivates the study of \emph{Normal approximation}, which aims to give explicit, finite-sample bounds on the distance between a statistic and a Normal law.

A first step in the direction of Normal approximation and a refinement of central limit theorems are given by Berry--Esseen inequalities. 
These provide finite-sample error bounds in metrics such as the Wasserstein or Kolmogorov distance. 
In its most basic form \cite{Esseen1942_LiapunoffError}, the Berry--Esseen theorem states that for independent and identically distributed real random variables $X_1, X_2, \ldots$ with finite third moment, 
and $S_n$ the $n$'th partial sum of $X_1, X_2, \ldots$, there exists a constant $C > 0$ depending on the metric $\rho$ such that 
\begin{equation}
    \label{eq:berry_esseen}
    \rho \big(\tfrac{S_n - \E[S_n]}{\sqrt{\V[S_n]}}, \mathcal{N}(0,1) \big) \le \tfrac{C \E[\vert X_1 \vert^3]}{\V[X_1]^{3/2} \sqrt{n}}.
\end{equation}
In many cases, this distance actually decays as $n^{-1/2}$, and hence in general, one cannot hope to improve the
bound in \eqref{eq:berry_esseen}. This bound of $n^{-1/2}$ is therefore sometimes called \emph{the optimal rate}. In spatial models or systems of dimension $d$, this optimal rate can appear as $n^{-d/2}$,
 where $n^d$ represents the effective volume or number of degrees of freedom driving the fluctuations. Beyond sums of independent variables, Berry--Esseen type results also hold for strongly dependent statistics, including functionals of Poisson point processes. 
Normal approximation of functionals of a Poisson point process has become a highly active research area in theoretical probability theory, especially stochastic geometry (e.g.~\cite{mehler}). 

%
%
The success of Normal approximation for functionals of Poisson point processes is largely due to a two-step approach. The first step typically consists of applying Malliavin-Stein 
theory to derive general upper bounds on Wasserstein- and Kolmogorov distances of the considered test statistics and a standard Normal random variable \cite{mehler,y3,trauth,trauth2}. 
These upper bounds are typically expressed in terms of iterated integrals involving mixed moments of both first- and second-order difference operators describing the effect on 
the functional of adding one or two points. After that substantial effort needs to be invested into deriving useful bounds on these iterated integrals. This has successfully been
carried out in several cases of interest in probability theory, including sums of region-stabilizing scores, certain hyperbolic functionals and random connection models \cite{mal_stab,chinmoy,trauth3,hug,nestmann}.

Loosely speaking, the stabilizing functionals from \cite{mal_stab, yukCLT} are based on score functions exhibiting local dependence.
 This means that when changing the Poisson point process at a specific location, the scores outside a constant-order neighborhood are unaffected.
However in some settings, only a subset of coordinates drives the local dependence, 
 while the remaining directions contribute in a non-local manner. This phenomenon leads to \emph{column-type interactions}, 
 where locality is restricted to the first $k$ coordinates.  In contrast, for functionals of column-type interactions, changing the Poisson point process at a specific location can induce changes in the score function far away from that 
 point. Such column-type interactions have been studied intensively,
  especially in the context of percolation theory, where they are notoriously difficult to treat \cite{hilario,AD,hoffmann,brochette,vares}.

\subsection{Motivational examples}
\label{sec:motivational_examples}
 To illustrate column-type dependence of order $k \leq d$, we highlight three examples.
\begin{enumerate}[label=(\Alph*)]

\item Consider snowflakes falling through the air and land on a two-dimensional surface, and consider as a statistic the number of times that two hexagonal arms of the snowflakes land on top of each other (Figure \ref{fig:introduction}).
If two snowflakes are far apart along the two directions of the Earth's surface, they do not contribute to the statistic. However, two snowflakes on the same location on Earth but starting at different
altitudes could hit one-another on the ground. Hence, this statistic exhibits column-type interactions of order $k=2$. This dependence can mathematically be generalized to \emph{the crossing number} of edges in a random connection model in $\R^d$
projected onto a two-dimensional plane, which in the fixed radius setting is studied in \cite{doring}.

\item In topological data analysis, a barcode summarizes the persistence, or lifetime, of topological features across different scales. 
A barcode is a set of pairs $\{(b_i,d_i)\}_{i \in I}$ commonly represented by bars, each bar representing the lifetime of a topological feature (Figure \ref{fig:introduction}).
A basic statistic is the inversion count \cite{adelie}, which records how often one bar is nested within another,
 while the tree realization number \cite{stanley1997enumerative} encodes how many ways such barcodes can be realized as tree structures.
  The tree realization number is hence the product of all inversion counts. Both of these statistics are only concerned with the time-overlap of bars,
   and hence bars far apart in the time direction do not contribute to the statistic, 
   whereas bars close in the time direction but far apart in other directions may contribute to the statistic.
    Hence, the inversion count and tree realization number exhibit column-type dependence of order $k = 1$. 


\item In telecommunication and telemarketing networks, one may ask whether there are clients or devices entirely disconnected from the system. 
Detecting such \emph{isolated vertices} is critical for ensuring coverage or designing robust marketing strategies. 
Modelling these vertices as Poisson points and connections as vertices within a fixed distance (Figure \ref{fig:introduction}),
the event of an isolated vertex can be computed from a Poisson functional, which exhibits full locality, i.e., column-type interactions of order $k=d$. \vspace{1ex}
\end{enumerate}

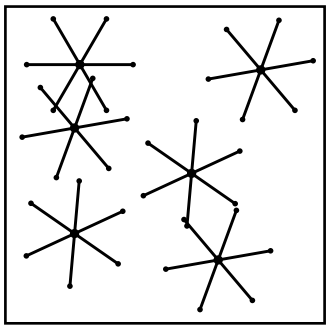
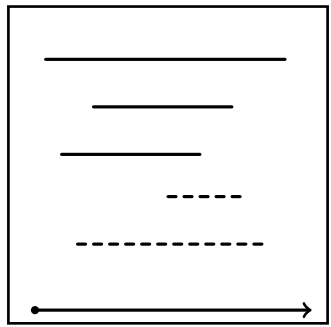
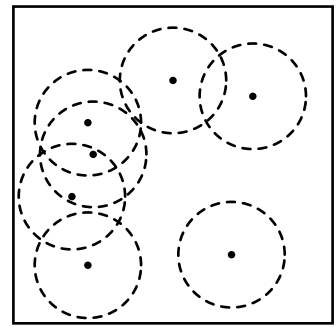
\begin{figure}[ht]
\centering

\begin{subfigure}[t]{0.3\textwidth}
\centering
\begin{tikzpicture}[scale=0.7,line cap=round]
  \draw[line width=1pt] (0,0) rectangle (6,6);

  \newcommand{\snowflake}[3]{
    \begin{scope}[shift={(#1,#2)}, line width=1.1pt]
      \fill (0,0) circle (2.5pt);
      \foreach \ang in {0,60,...,300}{
        \draw (0,0) -- +(\ang:#3);
        \fill (\ang:#3) circle (1.5pt);
      }
    \end{scope}
  }

   \newcommand{\snowflakeR}[3]{
    \begin{scope}[shift={(#1,#2)}, line width=1.1pt]
      \fill (0,0) circle (2.5pt);
      \foreach \ang in {10,70,...,310}{
        \draw (0,0) -- +(\ang:#3);
        \fill (\ang:#3) circle (1.5pt);
      }
    \end{scope}
  }

    \newcommand{\snowflakeRR}[3]{
    \begin{scope}[shift={(#1,#2)}, line width=1.1pt]
      \fill (0,0) circle (2.5pt);
      \foreach \ang in {25,85,...,325}{
        \draw (0,0) -- +(\ang:#3);
        \fill (\ang:#3) circle (1.5pt);
      }
    \end{scope}
  }

  \snowflake{1.4}{4.9}{1.0}
  \snowflakeR{1.3}{3.7}{1.0}

  \snowflakeRR{1.3}{1.7}{1.0}
  \snowflakeR{4.0}{1.2}{1.0}

  \snowflakeRR{3.5}{2.84}{1.0}
  \snowflakeR{4.8}{4.8}{1.0}
\end{tikzpicture}
\caption{A two-dimensional surface viewed from above. In total 3 edges of the fallen snowflakes are overlapping, so the crossing number here is 3.}
\end{subfigure}
\hfill
\begin{subfigure}[t]{0.3\textwidth}
\centering
\begin{tikzpicture}[scale=0.7,line cap=round]
  \draw[line width=1pt] (0,0) rectangle (6,6);

  \draw[line width=1.2pt] (0.7,5.0) -- (5.2,5.0);
  \draw[line width=1.2pt] (1.6,4.1) -- (4.2,4.1);
  \draw[line width=1.2pt] (1.0,3.2) -- (3.6,3.2);
  \draw[line width=1.2pt, dashed] (3.0,2.4) -- (4.5,2.4);
  \draw[line width=1.2pt, dashed] (1.3,1.5) -- (4.9,1.5);

  \fill (0.5,0.25) circle (2.2pt);
  \draw[line width=1.2pt,->] (0.5,0.25) -- (5.7,0.25);
\end{tikzpicture}
\caption{A barcode plot and an arrow that represents time. The top dashed bar is nested in the bottom dashed bar.
The inversion count is 6.}
\end{subfigure}
\hfill
\begin{subfigure}[t]{0.3\textwidth}
\centering
\begin{tikzpicture}[scale=0.7,line cap=round]
  \draw[line width=1pt] (0,0) rectangle (6,6);

  \def\r{1.0}

  \foreach \P in {(1.4,3.8),(3.0,4.6),(4.5,4.3),(1.5,3.2),(1.1,2.4),(1.4,1.1)}{
    \fill \P circle (2pt);
    \draw[dashed,line width=1pt] \P circle (\r);
  }

  \fill (4.1,1.3) circle (2pt);
  \draw[dashed,line width=1pt] (4.1,1.3) circle (\r);
\end{tikzpicture}
\caption{A network of 7 devices with a fixed range of connection (dashed circles). The device to the bottom-right is an isolated vertex.}
\end{subfigure}
\caption{Illustrations of the three examples highlighted in Section \ref{sec:motivational_examples}.}
\label{fig:introduction}
\end{figure}

\subsection{Main contributions}

The main contributions of this paper can now be presented in light of the first two examples above. The third example was only included to illustrate the full spectrum of locality, 
and will not be treated in detail in this paper.

\begin{enumerate}[label=(\Roman*)]
  \item \textbf{Crossing number:} In \cite{doring}, a first Normal approximation result for the crossing number of projected random geometric graphs was established in a fixed window with increasing intensity of points. However, we can extend the Normal approximation
  to also include connections models where the connection radius is random and even spatially may depend on the other Poisson points in a localizing manner.

  \item \textbf{Barcode statistics:} We establish Normal approximation of the inversion count for different models for the barcode lengths. In particular, we consider barcodes generated by Poisson trees,
  which is a geometrically dependent model, and where Normal approximation previously had been inaccessible with the current theory.
    Additionally, we consider a log-transformed version of the tree realization number, and establish Normal approximation using the same barcode models as before. 
    Applying the Delta-Method to this statistic also allows us to obtain asymptotic normality of the tree realization number itself. 

  \item \textbf{Unified framework:} By embedding these examples into a general class of \emph{pairwise score functions} of
   double-sum or sum-log-sum form, we provide a systematic approach to Normal approximation under column-type dependence. 
   This situates our results at the interface between stochastic geometry and topological data analysis, offering new tools for problems where partial locality and stabilization coexist.
\end{enumerate}

The main results of the present work are Normal approximation in the Wasserstein distance for what we call the \emph{double-sum}  and \emph{sum-log-sum} Poisson functionals, respectively.
These functionals capture the crossing number, inversion count, and tree realization number as special cases.
Additionally, in view of the spatially optimal rate of $n^{-d/2}$,
as discussed above, we obtain a \emph{near-optimal rate} for our bounds in the sense that the bound is of order $n^{-d/2 + \eps}$ for any $\eps > 0$. \vspace{0.5ex}

The rest of the paper is organized as follows.

\emph{Section \ref{sec:main_results}:} We introduce the double-sum and sum-log-sum functionals in detail and state the Normal approximation results as Theorems \ref{thm:sum_clt} and \ref{thm:log_clt}. 
We prove Theorems \ref{thm:sum_clt} and \ref{thm:log_clt} using the Malliavin--Stein Normal approximation in \cite{mehler}, which involves controlling three error terms, and also sketch the overall strategy for bounding these terms.
We end this section by discussing extensions and limitations of our results.

\emph{Section \ref{sec:3}:} We check that we can apply the Normal approximation results to the crossing number, inversion count, and tree realization number. To ensure that the bounds vanish, we particularly need to verify that the variance grows sufficiently fast.

\emph{Section \ref{sec:4}:} We prove the error bounds related to the double-sum Poisson functionals.

\emph{Section \ref{sec:5}:} We prove the error bounds related to the sum-log-sum Poisson functionals.

\emph{Appendix \ref{sec:A}:} We include proofs of the technical tools that are used in Section \ref{sec:3} to verify the assumptions in the Normal approximation theorems for the examples we consider. In particular, this includes a lower bound on the variance for both the double-sum functional and the sum-log-sum functional.

\pagebreak
\section{Main results}
\label{sec:main_results}
This section is dedicated to stating and proving the main results of the paper. 
Specifically, we consider a marked stationary Poisson point process observed on expanding regions of space and investigate the asymptotic behavior of two types of Poisson functionals,
\begin{equation}
    \label{eq:double_sum_form}
\sum_{Z}\sum_{V} f(Z,V),
\end{equation}
\begin{equation}
    \label{eq:sum_log_sum_form}
\sum_{Z} \log \sum_{V} f(Z,V),
\end{equation}
where $(Z,V)$ denotes a distinct pair of marked Poisson points. We
call \eqref{eq:double_sum_form} the \emph{double-sum} functional
and \eqref{eq:sum_log_sum_form} the \emph{sum-log-sum} functional. 
Additionally, when transforming \eqref{eq:sum_log_sum_form} by the exponential function, we obtain the \emph{product-sum} functional.
Key examples of the double-sum functional include the \emph{crossing number} and the \emph{inversion count}.
On the other hand, the product-sum functional (or sum-log-sum functional) is exemplified by the \emph{tree realization number} (or the \emph{log-transformed tree realization number}).

The rest of the section is organized as follows:
In Section \ref{sec:setup}, we introduce the concepts needed to rigorously state the main results.
In Section \ref{sec:double_sum}, we state the Normal approximation for the double-sum functional (Theorem \ref{thm:sum_clt}).
In Section \ref{sec:sum_log_sum}, we state the Normal approximation for the sum-log-sum functional (Theorem \ref{thm:log_clt}).
Moreover, as a consequence of Theorem \ref{thm:log_clt}, we obtain asymptotic normality of the product-sum functional.
In Section \ref{sec:main_proofs}, we prove Theorems \ref{thm:sum_clt} and \ref{thm:log_clt}.
Both proofs rely on the Malliavin--Stein Normal approximation from \cite{mehler}, which hinges on controlling three error terms.
In Section \ref{sec:proof_strategy}, we sketch the main ideas for bounding these error terms.
Finally, in Section \ref{sec:limitations}, we discuss possible extensions and limitations of our results.

\subsection{Setup, terminology and notation}
\label{sec:setup}
Let $(\Omega, \mathcal{F}, P)$ denote a probability space which is large enough to contain all random objects in the present work. 
Let $\mathbb M$ denote a Polish metric space and let $\mathbb X = \R^d \times \mathbb M$. Let $\mathfrak N(\mathbb{X})$ denote the
space of locally finite counting measures on the Borel $\sigma$-algebra of $\mathbb{X}$. Let $\PP$ denote a Poisson point process in $\mathbb{X}$
with intensity measure $\lambda = \vert \cdot \vert \otimes \mu$, where $\vert \cdot \vert$ is the Lebesgue measure on $\R^d$ and $\mu$ is 
an atom-free probability measure on $\mathbb M$. In other words, we think of $\PP$ as a unit-intensity Poisson point process in the \emph{spatial component} $\R^d$,
where each point is equipped with an \textit{independent mark} from $\mathbb M$ according to the measure $\mu$. We write $\dot x$
for the spatial component in $\R^d$ of a point $x \in \mathbb{X}$.
As an example, we could take $\mathbb M = [0,\infty)$ as the mark space and equip this space with the Exponential distribution $\mu$.

Next, for every $n \ge 1$, let $W_n \subseteq \R^d$ denote a $d$-dimensional rectangle of the form
\begin{equation}
    \label{eq:rectangle_form}
 W_n = [0,n] \times [0,a_2 n^{\alpha_2}] \times \dots \times [0,a_d n^{\alpha_d}],
\end{equation}
for some $a_j, \alpha_j > 0$ and $j=2,\ldots,d$. In particular, as $n$ increases, the volume of $W_n$
grows at most polynomially with $n$. Furthermore, let $\W_n = W_n \times \mathbb M$ and let
$
\PP_n = \PP \cap \W_n
$
denote the spatial restriction of $\PP$ to the rectangle $W_n$, which
leads to a sequence of Poisson point processes with intensity increasing in $n \ge 1$. Let
$
f \co \mathbb{X} \ti \mathbb{X} \ti \mathfrak N(\mathbb{X}) \to [0, \infty)
$
denote a Borel-measurable map 
that is symmetric in the first two entries and where $f(x,x,\cdot)=0$
for any $x \in \R^d$, i.e., $f$ vanishes on the diagonal. Henceforth, we refer to $f$ as the \emph{score function},
and refer to $f(Z,V,\PP_n)$ as the \textit{score} between the points $Z,V \in \PP_n$.

Finally, in view of the main results in the upcoming sections, recall that the Wasserstein distance $d_W$ between the random variables $X$ and $Y$ is defined as
\begin{equation}
d_W(X,Y) = \sup_{h \in \text{Lip}(1)} \vert \E[h(X)] - \E[h(Y)] \vert,
\end{equation}
where $\text{Lip}(1)$ denotes the set of all Lipschitz functions $h : \R \to \R$
with Lipschitz constant at most 1. Note that if the Wasserstein distance vanishes, i.e., $d_W(X_n,Y) \to 0$, then \cite[Theorem 7.12]{Villani2003_TopicsInOptimalTransportation} implies weak convergence of $X_n$ to $Y$.

\subsection{Normal approximation of the double-sum functional}
\label{sec:double_sum}

We now formalize the functional in (\ref{eq:double_sum_form}) and define the double-sum functional $\Sigma(\PP_n)$ as
\begin{equation}
    \label{eq:double_sum_formal}
\Sigma(\PP_n) = \sum_{Z \in \PP_n}\sum_{V \in \PP_n} f(Z,V,\PP_n).
\end{equation}
As a simple example, if $f(Z,V,\PP_n)$ is the indicator of the event that
the spatial distance between $Z$ and $V$ is less than 1,
then (\ref{eq:double_sum_formal}) is the total number of distinct Poisson pairs within distance 1 of each other.

Naturally, in order to establish Normal approximation, we need to impose some conditions on the score function $f$.
We now introduce these conditions one by one.

First, for any $1 \le j \le d$ and $x \in \mathbb X$, let $\dot x_j$ denote the $j$th coordinate of the spatial component $\dot x$.
Then, for $1 \le k \le d$, we say that $f$ is \emph{$k$-local} if
\begin{equation}
    \label{eq:k_local}
f(x, y, \PP_n) = 0, \quad \text{whenever } |\dot x_j - \dot y_j| > 1 \text{ for some } 1 \le j \le k,
\end{equation}
i.e., the score function $f$ vanishes whenever any of the first $k$ coordinates of the spatial components of $x$ and $y$ are at a distance larger than 1.
Note that if $k_1 \le k_2$ and $f$ is $k_2$-local, then $f$ is also $k_1$-local.
Hence the weakest such assumption is that $f$ is 1-local.
In conjunction with $k$-locality, let
\begin{equation}
S_n^k(x,s) = \{ y \in W_n \co |\dot x_j - y_j| \le s \text{ for all } 1 \le j \le k \}
\end{equation}
denote the vertical \emph{column} or \emph{slab} consisting of all points in $W_n$
whose first $k$ coordinates are at distance at most $s$ from the first $k$ coordinates of the spatial component of $x$.
Additionally, we let $\mathbb{S}_n^k(x,s) = S_n^k(x,s) \times \mathbb M$ and note that $\lambda(\mathbb{S}_n^k(x,s)) = \vert S_n^k(x,s) \vert$.
The consequence of $k$-locality is that $Z$ can only have non-zero scores with Poisson points $V$ that lie in the slab $\mathbb{S}_n^k(Z,1)$,
see Figure \ref{fig:partition}.

Additionally, we will be inspired by the concept of stabilization radii from \cite{yukCLT}.
To be precise, let $x \in \mathbb{X}$ and $\PP_n^x = \PP_n \cup \{x\}$. Furthermore, for $m \ge 1$, let
\begin{equation}
Q(x,m) = x + [-m,m]^d
\end{equation}
denote the $d$-dimensional cube centered at $x$ with side length $2m$. As before, we let $\mathbb{Q}(x,m) = Q(x,m) \times \mathbb M$.
With this, we say that $f$ \emph{stabilizes at $x$ at radius $m$} if
\begin{equation}
    \label{eq:stabilization}
f(Z,V,\PP_n^x) = f(Z,V,\PP_n) \quad \text{whenever } Z,V \notin \PP_n \cap \mathbb{Q}(x,m).
\end{equation}
Then, we introduce the \emph{radius of stabilization} as 
\begin{equation}
    \label{eq:stabilization_radius_original}
R_{n}(x) = R_{n}(x,f,\PP_n) = \min \{ m \ge 1 \co f \text{ and } m \text{ satisfy } \eqref{eq:stabilization} \},
\end{equation}
i.e., the smallest integer radius at which $f$ stabilizes at $x$.
We refer to both $\mathbb{Q}(x,R_n(x))$ and $Q(x,R_n(x))$ as \emph{the non-stable cubes} around the point $x$, see Figure \ref{fig:partition}.

\vspace{1.5ex}
\begin{figure}[h!]
    \centering

\begin{tikzpicture}[scale=2]
  
  \draw (-2,-1.5) rectangle (2.75,1.5);

  \fill[fill=cyan!20, opacity=0.3] (-0.5,-1.5) rectangle (0.5,1.5);
  \draw (-0.5,-1.5) rectangle (0.5,1.5);
   Optional label for strip width
  \draw[<->] (-0.48,-1.2) -- (0.48,-1.2) node[midway,below] {$2s$};

  \fill (0,0) circle (0.75pt) node[below] {$x$};
  \fill (-0.1,0.4) circle (0.75pt) node[below] {$Z$};
  \fill (0.3,1) circle (0.75pt);
  \fill (-1.4,-1) circle (0.75pt);
    \fill (-1.5,1.1) circle (0.75pt);
  \fill (1.5,-0.9) circle (0.75pt);
  \fill (1.75,1.3) circle (0.75pt);
    \fill (2.25,-0.1) circle (0.75pt);

  \draw[dashed,red] (-0.75,-0.75) rectangle (0.75,0.75);
  
  \draw[<->] (0.85,-0.75) -- (0.85,0.75) node[midway,right] {$2R_n(x)$};
\end{tikzpicture}

    \caption{The black square is the rectangle $ W_n$, the light blue vertical strip is the slab $S_n^k(x,s)$,
     and the red dashed square is the non-stable cube $Q(x,R_n(x))$. Since $Z$ is the only Poisson point in this cube,
    the insertion of $x$ can only affect scores between $Z$ and another Poisson point, and not between any two Poisson points
    both outside this cube.}
    \label{fig:partition}
\end{figure}
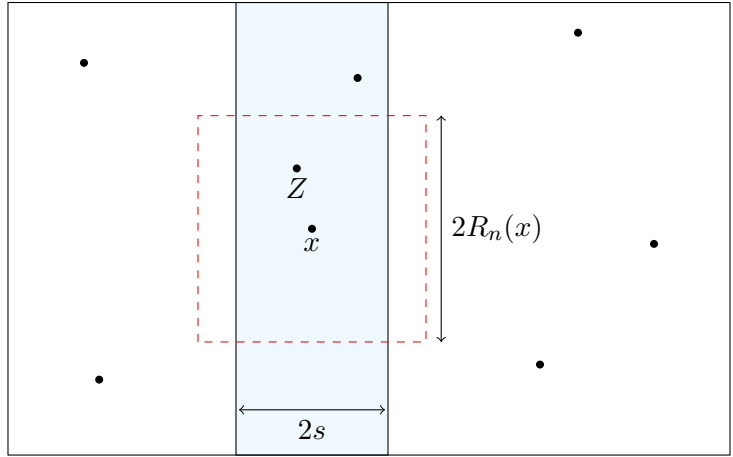
\vspace{1.5ex}

What's more, we say that $f$ \emph{stabilizes exponentially} 
(inside $\W_n$) if there exists $\beta_1 > 0$ such that for any $\eps \in (0,1)$ and all sufficiently 
large $n \ge 1$,
\begin{equation}
    \label{eq:stabilization_exponential}
\sup_{x \in W_n} P(R_{n}(x) > n^{\eps}) \le \e^{-\beta_1 n^{\eps}},
\end{equation}
i.e., the random side length of the non-stable cube has exponentially decaying tails around every
point. As this will not be the last assumption of this type,
we henceforth write $n \gg 1$ to denote that an expression is true for
``sufficiently large $n \ge 1$''. Note that the size of $n$ may depend on the chosen $\eps > 0$.

Next, we introduce $\mc R$ as the set of all pairs of points in $\mathbb X$ alongside the empty set, i.e.,
\begin{equation}
    \label{eq:all_pairs}
\mc R = \{\emptyset\} \cup \big\{\{x,y\} \colon x,y \in \mathbb X\big\},
\end{equation}
as well as the quantity
\begin{equation}
\label{eq:maximal_score}
 f_{\sup}(\PP_n) = \sup_{\mathrm x, \mathrm y \in \mc R} \sup_{Z,V \in \PP_n^\mathrm{x}} f(Z,V,\PP_n^\mathrm{y}).
\end{equation}
With this, we say that $f$ has \emph{sub-polynomial moments} if there is a random variable $\overline f_{\sup}(\PP_n) \ge f_{\sup}(\PP_n)$
such that for all $m \in \N$, 
$\eps > 0$, and $n \gg 1$,
\begin{equation}
    \label{eq:exp_decay}
\E[\overline f_{\sup}(\PP_n)^m] \le n^{m \eps},
\end{equation}
i.e., the score function $f$ itself has sub-polynomial moments. Note that this condition is automatically satisfied if $f$ is bounded, e.g., by 1.

For convenience, we now compactly recap the above considerations as Assumption \ref{ass:A}.

\vspace{1ex}
\begin{assumptionA}[ass:A]
  \item \label{ass:A1} $f$ is $k$-local for some $k \in \N$, cf.\ \eqref{eq:k_local}.
  \item \label{ass:A2} $f$ stabilizes exponentially, cf.\ \eqref{eq:stabilization_radius_original} and \eqref{eq:stabilization_exponential}.
  \item \label{ass:A3} $f$ has sub-polynomial moments, cf.\ \eqref{eq:maximal_score} and \eqref{eq:exp_decay}.
\end{assumptionA}

We are now ready to state the main result of this section, which 
provides an upper bound on the Wasserstein distance between a centered and scaled version of $\Sigma(\PP_n)$ and a standard Normal random variable $\mc N(0,1)$. Moreover, we 
also obtain a sufficient lower bound on the variance of $\Sigma(\PP_n)$ which ensures weak convergence of $\Sigma(\PP_n)$.
For simplicity, we write $\vert S_n^k\vert$ instead of $\vert S_n^k(0,1)\vert$, and we use $\V[\Sigma(\PP_n)]$ to denote the variance of $\Sigma(\PP_n)$.

\bet[Normal approximation of $\Sigma$]
\label{thm:sum_clt}
Suppose $\PP_n$ and $f$ are as in Assumption \ref{ass:A}. Then, for every $\delta > 0$ and $n \gg 1$,
\begin{equation}
    \label{eq:main_sum_bound}
d_W\Big(
\frac{\Sigma(\PP_n) - \E[\Sigma(\PP_n)]}{\sqrt{\V[\Sigma(\PP_n)]}}, 
\mc N(0,1) \Big)
\le
 \frac{n^{\delta } \sqrt{\vert W_n \vert} \vert S_n^k\vert^2}{\V[\Sigma(\PP_n)]} 
 + \frac{n^{\delta } \vert W_n \vert \vert S_n^k\vert^3}{\V[\Sigma(\PP_n)]^{3/2}}.
\end{equation}
In particular, if 
$\V[\Sigma(\PP_n)] \ge C \vert W_n \vert \vert S_n^k\vert^2$ for some $C>0$, then for all $\delta > 0$ and $n \gg 1$,
\begin{equation}
    \label{eq:main_sum_bound_optimal}
d_W\Big(
\frac{\Sigma(\PP_n)- \E[\Sigma(\PP_n)]}{\sqrt{\V[\Sigma(\PP_n)]}}, 
\mc N(0,1) \Big)
\le
 \frac{n^{\delta }}{\sqrt{\vert W_n \vert}}.
\end{equation}
\end{theorem}

Let us list a few observations from Theorem \ref{thm:sum_clt}. Note that a larger value of $k$ in $k$-locality implies $\vert S_n^k \vert$ is smaller, and hence the upper bound in Theorem \ref{thm:sum_clt} is smaller as well.
Also, note that under the convention that $d_W(\infty,N(0,1))=\infty$, the bound in (\ref{eq:main_sum_bound}) still holds even if $\V[\Sigma(\PP_n)] = 0$.
Finally, we record how \eqref{eq:main_sum_bound_optimal} takes form when the sides of $W_n$ are equal
in length, i.e., $W_n = [0,n]^d$:
If $\V[\Sigma(\PP_n)] \ge Cn^{3d-2k}$ for some $C>0$, then
\begin{equation}
    \label{eq:main_sum_bound_equal_sides}
d_W\Big(
\frac{\Sigma(\PP_n)- \E[\Sigma(\PP_n)]}{\sqrt{\V[\Sigma(\PP_n)]}}, 
\mc N(0,1) \Big)
\le n^{- d/2 + \delta}.
\end{equation}
Thus, \eqref{eq:main_sum_bound_equal_sides} shows
that we obtain a near-optimal bound as discussed in Section \ref{sec:intro}.

\subsection{Normal approximation of the sum-log-sum functional}
\label{sec:sum_log_sum}

Before we can rigorously define the sum-log-sum functional, we need to restrict our setting further. First, we impose that the score function is of the following form
\begin{equation}
    \label{eq:f_form}
\one\{Z \in \mc A_n(\PP_n)\} g(Z,V,\PP_n),
\end{equation}
where $\mc A_n(\PP_n)$ is a random set and $g$ is an integer-valued function.
Here we think of $\mc A_n(\PP_n)$ as an admissibility condition that $Z$ must satisfy, which can be tailored to fit the application.
As an example, in the case of the tree realization number, $\mc A_n(\PP_n)$ could be the condition that $Z$ has a non-zero barcode length associated with it
and that $Z$ lies not too close to the boundary of $W_n$.
Note that we still require that \eqref{eq:f_form} is symmetric in the pair of Poisson points.
In the same beat, for $\mathrm{x} \in \mc R$, we introduce the abbreviation
\begin{equation}
G(Z,\PP_n^\mathrm{x}) = \sum_{V \in \PP_n^\mathrm{x}} g(Z,V,\PP_n^\mathrm{x}).
\end{equation}
For convenience, we then introduce the extended condition $\mc A_n^+(\PP_n)$ defined as
\begin{equation}
\mc A_n^+(\PP_n) = \mc A_n(\PP_n) \cap \{ Z \in \PP_n \colon G(Z,\PP_n) > 0\},
\end{equation}
where the $+$ in $\mc A_n^+$ indicates that the compound score $G$ is positive. Then, we can formalize the sum-log-sum functional in (\ref{eq:sum_log_sum_form}) as
\begin{equation}
    \label{eq:sum_log_sum_formal}
\Sigma^{\log}_n(\PP_n) = \sum_{Z \in \mc A_n^+(\PP_n)} \log G(Z,\PP_n),
\end{equation}
where we note that the condition $G(Z,\PP_n) > 0$ inside $\mc A_n^+(\PP_n)$ ensures that the sum-log-sum functional is well-defined.
The additional $n$ in the notation $\Sigma^{\log}_n$ is to emphasize that the functional may depend on $n$ through the admissibility condition $\mc A_n^+(\PP_n)$.
Additionally, we also define the product-sum functional as
\begin{equation}
    \label{eq:product_sum_formal}
\Pi_n(\PP_n) = \exp\big(\Sigma^{\log}_n(\PP_n)\big) = \prod_{Z \in \mc A_n^+(\PP_n)} G(Z,\PP_n).
\end{equation}
While the concepts of $k$-locality, exponential stabilization and exponential decay of the score function will be sufficient for the Normal approximation of the double-sum functional, we need to impose some additional conditions
for the sum-log-sum functional.

The first additional assumption we impose is a second type of stabilization. Similar to Assumption \ref{ass:A}(ii), we say that $\mc A_n^+(\PP_n)$
\emph{stabilizes at $x$ at radius $m$} if 
\begin{equation}
    \label{eq:additional_stabilization}
Z \in \mc A_n^+(\PP_n^x) \iff Z \in \mc A_n^+(\PP_n) \quad \text{whenever } Z \notin Q(x,m),
\end{equation}
and then we let $R_n(x) > 0$ denote the smallest integer radius
such that the stabilization in \eqref{eq:stabilization} and \eqref{eq:additional_stabilization} 
holds simultaneously, i.e.,
\begin{equation}
    \label{eq:stabilization_radius}
     R_n(x) = \min \{ m \ge 1 \co m, \ g  \text{ and } \mc A_n^+(\PP_n) \text{ satisfy } \eqref{eq:stabilization}  \text{ and } \eqref{eq:additional_stabilization} \}.
\end{equation}
We say $g$ \emph{stabilizes exponentially with respect to $\mc A_n^+(\PP_n)$} if there exists a $\beta_2 > 0$ such that
\begin{equation}
    \label{eq:exponential_stabilization}
    \sup_{x \in W_n} P(R_n(x) > n^\eps) \le \e^{-\beta_2 n^\eps} \quad \text{for all } \eps \in (0,1) \text{ and } n \gg 1.
\end{equation}
Next, for $\beta_3 > 0$, introduce the event $F_n = F_n(\beta_3)$ as
\begin{equation}
    \label{eq:exponential_concentration}
F_n = \bigcup_{x \in W_n} \bigcup_{Z \in \mc A_n^+(\PP_n^x)} \bigcup_{v \in \mc R} \big\{ G(Z,\PP_n^\mathrm{x}, \PP_n^\mathrm{x}) < \beta_3 \vert S_n^k \vert \big\}.
\end{equation}
We then say that $G$ \emph{concentrates exponentially} if there exist a measurable set $\widetilde F_n \in \mc F$ with $F_n \subseteq \widetilde F_n$ and a constant $\beta_4 > 0$ such that for all $n \gg 1$,
\begin{equation}
    \label{eq:exponential_concentration_2}
    P(\widetilde F_n) \le \e^{-\beta_4 \vert S_n^k \vert},
\end{equation}
i.e., loosely speaking, the probability that any admissible Poisson point only has a few scores is exponentially small. To ensure that quantities such 
as $\e^{-\beta_4 \vert S_n^k \vert}$ and $1/\vert S_n^k \vert$ vanish as $n \to \infty$, we also assume that $k < d$. 
For convenience, once again, we now compactly recap the additional assumptions above as Assumption \ref{ass:B}.
\vspace{1.5ex}
\begin{assumptionB}[ass:B]
    \item \label{ass:B1} $g$ stabilizes exponentially with respect to $\mc A_n^+(\PP_n)$, cf.\ \eqref{eq:stabilization_radius} and \eqref{eq:exponential_stabilization}.
    \item \label{ass:B2} $G$ concentrates exponentially, cf.\ \eqref{eq:exponential_concentration} and \eqref{eq:exponential_concentration_2}.
\end{assumptionB}
We can now state the main result of this section, which is an analogue of Theorem \ref{thm:sum_clt} for the sum-log-sum functional $\Sigma^{\log}_n(\PP_n)$ rather than for the double-sum functional $\Sigma(\PP_n)$.
\bet[Normal approximation for $\Sigma^{\log}_n$]
\label{thm:log_clt}
Suppose $\PP_n$ and $g$ are as in Assumption \ref{ass:B}. Then, for every $\delta > 0$ and $n \gg 1$,
\begin{equation}
    \label{eq:main_log_bound}
d_W\Big(
\frac{\Sigma^{\log}_n(\PP_n)- \E[\Sigma^{\log}_n(\PP_n)]}{\sqrt{\V[\Sigma^{\log}_n(\PP_n)]}}, 
\mc N(0,1) \Big)
\le
\frac{n^{\delta}\sqrt{\vert W_n \vert}}{\V[\Sigma^{\log}_n(\PP_n)]}
+
\frac{n^{\delta}\vert W_n \vert}{\V[\Sigma^{\log}_n(\PP_n)]^{3/2}}.
\end{equation}
In particular, if $\V[\Sigma^{\log}_n(\PP_n)] \ge C \vert W_n \vert$ for some $C >0$, then for every $\delta > 0$ and $n \gg 1$,
\begin{equation}
    \label{eq:main_log_bound_optimal}
d_W\Big(
\frac{\Sigma^{\log}_n(\PP_n)- \E[\Sigma^{\log}_n(\PP_n)]}{\sqrt{\V[\Sigma^{\log}_n(\PP_n)]}}, 
\mc N(0,1) \Big)
\le
 \frac{n^{\delta}}{\sqrt{\vert W_n \vert}}.
\end{equation}
\ent

In Section \ref{sec:limitations}, we discuss what we can say about convergence of $\Sigma^{\log}_n(\PP_n)$ (as well as $\Sigma(\PP_n)$) in the Kolmogorov metric and whether any of the 
conditions in Assumptions \ref{ass:A} and \ref{ass:B} can be relaxed. Additionally, if $W_n$ is the cube $[0,n]^d$ and $\V[\Sigma^{\log}_n(\PP_n)] \ge Cn^{d}$ for some $C>0$, then \eqref{eq:main_log_bound_optimal} becomes
\begin{equation}
d_W\Big(
\f{\Sigma^{\log}_n(\PP_n)- \E[\Sigma^{\log}_n(\PP_n)]}{\sqrt{\V[\Sigma^{\log}_n(\PP_n)]}}, 
\mc N(0,1) \Big)
\le n^{- d/2 + \delta}.
\end{equation}
Finally, we can use Theorem \ref{thm:log_clt} and the Delta method \cite{Wasserman2004_AllOfStatistics} with the exponential function $h(x)=\e^{x}$ (where we note that $h'(x) > 0$ for all $x$) to obtain asymptotic normality for the product-sum functional as claimed.

\begin{corollary}[Asymptotic normality of $\Pi$]
\label{thm:prod_clt}
Suppose $\PP_n$ and $f$ are as in Assumption \ref{ass:B}. If $\V[\Sigma^{\log}_n(\PP_n)] \ge C \vert W_n \vert$ for some $C > 0$, then as $n \to \infty$,
\begin{equation}
\frac{\Pi_n(\PP_n) - \e^{\E[\Sigma^{\log}_n(\PP_n)]}}{\e^{\E[\Sigma^{\log}_n(\PP_n)]}\sqrt{\V[\Sigma^{\log}_n(\PP_n)]}} \overset{d}{\longrightarrow} \mc N(0,1).
\end{equation}
\end{corollary}
Naturally, the downside of the Delta method is that we cannot directly obtain rates for the product-sum functional as in Theorems \ref{thm:sum_clt} and \ref{thm:log_clt}.

\subsection{Proof of Theorems \ref{thm:sum_clt} and \ref{thm:log_clt}}
\label{sec:main_proofs}

As already mentioned, the main tool in proving both Theorems \ref{thm:sum_clt} and \ref{thm:log_clt} is the Malliavin--Stein Normal approximation from \cite[Theorem 1.1]{mehler}. The approach is very
similar for both functionals, and we start by proving the Normal approximation for $\Sigma(\PP_n)$.

In view of this task, consider now for $x,y \in \mathbb X$ the first- and second-order difference operators $D_x$
and $D_{xy}^2$ defined by 
\begin{equation}
\begin{aligned}
	D_x \Sigma(\PP_n)  = & \Sigma(\PP_n^x) - \Sigma(\PP_n),\\
	D_{xy}^2\Sigma(\PP_n)  = &  \Sigma(\PP_n^{xy})  - \Sigma(\PP_n^{x}) - \Sigma(\PP_n^{y})  +  \Sigma(\PP_n).
\end{aligned}
\end{equation}
Applying \cite[Theorem 1.1]{mehler}
to $\Sigma(\PP_n)$---which we argue is possible under Assumption \ref{ass:A}---yields three error terms $4 \sqrt{I_{n,1}}$, $\sqrt{I_{n,2}}$ and $I_{n,3}$, where
\begin{equation}
    \label{eq:error_terms_sum}
\begin{aligned}
	I_{n,1}  = & \int
     \E\big[D_x \Sigma(\PP_n)^2
    D_y \Sigma(\PP_n)^2\big]^{1/2}
    \E\big[D_{xz}^2 \Sigma(\PP_n)^2
    D_{yz}^2 \Sigma(\PP_n)^2\big]^{1/2} \ \text d (x,y,z),\\
	I_{n, 2}  = & \int \E\big[D_{xz}^2 \Sigma(\PP_n)^2
   D_{yz}^2 \Sigma(\PP_n)^2\big] \ \text d(x,y,z) ,\\
	I_{n, 3}  = & \int \E\big[|D_x\Sigma(\PP_n)|^3\big] \text dx,
\end{aligned}
\end{equation}
where the integration domain is $\W_n^3$ and $\W_n$, respectively. We now state upper bounds on each of these error terms, which we subsequently use to prove Theorem \ref{thm:sum_clt}.
\begin{lemma}[Error bounds: Double-Sum]
\label{lem:en1}
Let $\eps > 0$ and $n \gg 1$. Under Assumption \ref{ass:A},
\begin{enumerate}[label=\textup{(\roman*)}]
\item  $ \displaystyle I_{n,1} \le   n^{\eps} \vert  W_n \vert \vert S_n^k \vert^4, $
\item  $ \displaystyle  I_{n, 2} \le  n^{\eps} \vert  W_n \vert \vert S_n^k \vert^4, $
\item  $ \displaystyle  I_{n, 3} \le  n^{\eps} \vert  W_n \vert \vert S_n^k \vert^3. $
\end{enumerate}
\end{lemma}
The proof of Lemma \ref{lem:en1} is postponed
until Section 5, but we sketch the overall ideas in Section \ref{sec:proof_strategy}. Note that it is not a coincidence that there are no constants in any of the bounds in
Lemma \ref{lem:en1}. Essentially, the factor $n^{\gamma \eps}$ for any $\gamma > 0$ dominates any constant provided $n$ is large enough.
We record this observation (without proof) since it will come into play in nearly every single result in this paper. 
\begin{lemma}[Normalization]
    \label{lem:normalization}
    Let $\gamma > 0$ and $(a_n)_{n \ge 1}$ denote a sequence in $[0,\infty)$. Assume that for any $\eps > 0$, there exist $C(\eps) > 0$ 
    and $N_1(\eps) \in \N$ such that whenever $n \ge N_1(\eps)$,
    $$
    a_n \le C(\eps) n^{\gamma \eps}.
    $$
    Then, there exists $N_2(\eps) \ge N_1(\eps)$ such that whenever $n \ge N_2(\eps)$,
    $$
    a_n \le n^{\gamma \eps}.
    $$
\end{lemma}
We can now combine the error term bounds with Lemma \ref{lem:normalization} to prove Theorem \ref{thm:sum_clt}.
\bep[Proof of Theorem \ref{thm:sum_clt}]
First, by Lemma \ref{lem:en1}(iii), we see that 
\begin{equation}
    \label{eq:proof_bound}
    \E \Big[\int_{ \W_n} D_x \Sigma(\PP_n)^2\,  \text d x \Big] \le 
    \vert  W_n \vert \mu(\M)+ \E \Big[\int_{ \W_n} \vert D_x \Sigma(\PP_n)\vert^3\,  \text d x \Big] < \infty
\end{equation}
and hence, by the Poincaré inequality, we also have that $\Sigma(\PP_n)$ is square-integrable. 
Thus, we can define the centered and standardized version of $\Sigma(\PP_n)$ as
$$
\zeta(\PP_n) = \f{\Sigma(\PP_n)- \E[\Sigma(\PP_n)]}{\sqrt{\V[\Sigma(\PP_n)]}},
$$
and note by (\ref{eq:proof_bound}) that $\zeta(\PP_n)$ also satisfies
$$
    \E \Big[\int_{ \W_n} D_x \zeta(\PP_n)^2\,  d x \Big] < \infty.
$$
Thus, since we may now invoke \cite[Theorem 1.1]{mehler}, it follows that
$$
d_W(\zeta(\PP_n), \mc N) \le \frac{\sqrt{I_{n,1}}}{\V[\Sigma(\PP_n)]} + \frac{\sqrt{I_{n,2}}}{\V[\Sigma(\PP_n)]} + \frac{I_{n,3}}{\V[\Sigma(\PP_n)]^{3/2}}.
$$
Hence, applying Lemma \ref{lem:en1} and using that the bound on $I_{n,2}$ is smaller than the bound on $I_{n,1}$, for any $\eps > 0$,
\begin{align*}
d_W(\zeta(\PP_n), \mc N) \le 
 \frac{n^{2(d+1)\eps} \vert  W_n \vert^{1/2}  \vert S_n^k(0,n^\eps) \vert^2}{\V[\Sigma(\PP_n)]}+ 
\frac{n^{3(d+1)\eps/2} \vert  W_n \vert  \vert S_n^k(0,n^\eps)\vert^3 }{\V[\Sigma(\PP_n)]^{3/2}}.
\end{align*}
Choosing $\eps$ so small that $ 2(d+1+k)\eps \le \delta$ and invoking Lemma \ref{lem:normalization} completes the proof of (\ref{eq:main_sum_bound}).
Plugging in the inequality $\V[\Sigma(\PP_n)] \ge C \vert  W_n \vert\vert S_n^k(0,n^\eps)\vert^2$ in both denominators and simplifying again using Lemma \ref{lem:normalization},
completes the proof of the second bound in (\ref{eq:main_sum_bound_optimal}).
\enp

Next, we mirror the approach above to prove the Normal approximation for the functional $\Sigma^{\log}_n(\PP_n)$. Hence, consider the first- and second-order difference operators
\begin{equation}
\begin{aligned}
	D_x \Sigma^{\log}_n(\PP_n)  = & \Sigma^{\log}_n(\PP_n^x) - \Sigma^{\log}_n(\PP_n),\\
	D_{xy}^2\Sigma^{\log}_n(\PP_n)  = &  \Sigma^{\log}_n(\PP_n^{xy})  - \Sigma^{\log}_n(\PP_n^{x}) - \Sigma^{\log}_n(\PP_n^{y})  + \Sigma^{\log}_n(\PP_n),
\end{aligned}
\end{equation}
and the corresponding three error terms,
\begin{equation}
    \label{eq:error_terms_log}
\begin{aligned}
	\widetilde I_{n,1}  = & \int 
     \E\big[D_x \Sigma^{\log}_n(\PP_n)^2
    D_y \Sigma^{\log}_n(\PP_n)^2\big]^{1/2}
 \E\big[D_{xz}^2 \Sigma^{\log}_n(\PP_n)^2
    D_{yz}^2 \Sigma^{\log}_n(\PP_n)^2\big]^{1/2} \ \text d (x,y,z),\\
	\widetilde I_{n, 2}  = & \int_{\W_n^3} \E\big[D_{xz}^2 \Sigma^{\log}_n(\PP_n)^2
   D_{yz}^2 \Sigma^{\log}_n(\PP_n)^2\big] \ \text d(x,y,z) ,\\
	\widetilde I_{n, 3}  = & \int_{ \W_n}\E\big[|D_x\Sigma^{\log}_n(\PP_n)|^3\big] \ \text dx.
\end{aligned}
\end{equation}
As before, we now state upper bounds on each of these terms, while postponing their proofs until Section 5.
Instead, we immediately proceed to proving Theorem \ref{thm:log_clt}.
\begin{lemma}[Error bounds: Sum-log-sum]
\label{lem:en1_sim}
Let $\eps > 0$ and $n \gg 1$. Under Assumption \ref{ass:B},
\begin{enumerate}[label=\textup{(\roman*)}]
\item  $ \displaystyle \widetilde I_{n,1} \le  n^{\eps} \vert  W_n \vert, $
\item  $ \displaystyle  \widetilde I_{n, 2} \le  n^{\eps} \vert  W_n \vert, $
\item  $ \displaystyle  \widetilde I_{n, 3} \le  n^{\eps} \vert  W_n \vert. $
\end{enumerate}
\end{lemma}

\begin{proof}[Proof of Theorem \ref{thm:log_clt}]
First, similar to the proof of Theorem \ref{thm:sum_clt}, 
it follows by Lemma \ref{lem:en1_sim}(iii) and the Poincaré inequality that
 $\Sigma^{\log}_n(\PP_n)$ is square-integrable, and we can define 
$$
\zeta^{\log}_n(\PP_n) = \f{\Sigma^{\log}_n(\PP_n)- \E[\Sigma^{\log}_n(\PP_n)]}{\sqrt{\V[\Sigma^{\log}_n(\PP_n)]}},
$$
and conclude that 
$
    \E \Big[\int_{ \W_n} D_x \zeta^{\log}_n(\PP_n)^2\,  d x \Big] < \infty.
$
Thus, since we may now invoke \cite[Theorem 1.1]{mehler}, it follows that
$$
d_W(\zeta^{\log}_n(\PP_n), \mc N) \le \frac{\sqrt{\widetilde I_{n,1}}}{\V[\Sigma^{\log}_n(\PP_n)]} + \frac{\sqrt{\widetilde I_{n,2}}}{\V[\Sigma^{\log}_n(\PP_n)]} + \frac{\widetilde I_{n,3}}{\V[\Sigma^{\log}_n(\PP_n)]^{3/2}}.
$$
Hence, applying Lemma \ref{lem:en1_sim}, for any $\eps > 0$,
\begin{align*}
d_W(\zeta^{\log}_n(\PP_n), \mc N) \le 
 \frac{n^{\eps/2}\vert  W_n \vert^{1/2}}{\V[\Sigma^{\log}_n(\PP_n)]}+ 
\frac{n^{\eps}\vert  W_n \vert}{\V[\Sigma^{\log}_n(\PP_n)]^{3/2}},
\end{align*}
Plugging in that $\V[\Sigma(\PP_n)] \ge C \vert  W_n \vert$ and using Lemma \ref{lem:normalization} completes the proof.
\end{proof}

\subsection{Proof strategy for the error terms}
\label{sec:proof_strategy}

We now outline the overall strategy for obtaining the bounds in Lemmas \ref{lem:en1} and \ref{lem:en1_sim}.
To make the exposition clearer, we first focus on the double-sum case, and then explain the additional challenges that arise in the sum-log-sum case.

\noindent \textbf{Step 1:} We apply the Cauchy--Schwarz inequality to the expectation inside $I_{n,1}$ and $I_{n,2}$, 
\begin{align*}
	I_{n,1}  \le & \int \Big(\int
    \E\big[(D_y \Sigma(\PP_n))^4\big]^{\tfrac 1 4}\E\big[(D_{xy}^2 \Sigma(\PP_n))^4\big]^{\tfrac 1 4}\text dy\Big)^2 \text dx,\\
	I_{n, 2}  \le  & \int\Big(\int \E\big[(D_{xy}^2 \Sigma(\PP_n))^4\big]^{1/2} \text dy\Big)^2 \text dx,\\
	I_{n, 3}  = & \int \E\big[|D_x\Sigma(\PP_n)|^3\big] \text dx.
\end{align*}
Thus, we 
see that it suffices to obtain sufficiently strong bounds on the third and fourth moments of $D_x\Sigma(\PP_n)$ and
the fourth moment of $D_{xy}^2\Sigma(\PP_n)$.

\noindent \textbf{Step 2:} First,
we split $D_x\Sigma(\PP_n)$ into (I) the total score in $x$, i.e., $ \sum_{V}f(x,V,\PP_n^x)$, and (II) the total change when adding $x$, i.e.,
$\sum_{Z,V} [f(Z,V,\PP_n^x) - f(Z,V,\PP_n)]$.
By $k$-locality and sub-polynomial moments, the $m$th moment of (I) is less than  $n^{\eps} \vert S_n^k \vert^m$,
and by exponential stabilization and sub-polynomial moments,
the $m$th moment of (II) is less than $n^{\eps'} \vert S_n^k \vert^m$, where $\eps, \eps' > 0$ are arbitrarily small. 

\noindent \textbf{Step 3:} To handle the fourth moment of $D_{xy}^2\Sigma(\PP_n)$, we need to use different approaches based on the spatial location of $x$ and $y$ in relation to each other in order to avoid
error term bounds of size $\vert  W_n \vert^2$ or larger. Hence, we now introduce three cases,
\begin{equation}
    \label{eq:three_cases}
\begin{aligned}
\textbf{Case I:} \qquad & y \in  \W_n \setminus \S_n^k(x,n^\eps),\\
\textbf{Case II:} \qquad & y \in \S_n^k(x,n^\eps) \setminus \Q(x,n^\eps),\\
\textbf{Case III:} \qquad & y \in \Q(x,n^\eps).
\end{aligned}
\end{equation}
where we bound the fourth moment in different ways. Loosely speaking, in Case I, $x$ and $y$ are far apart in at least one local direction; in Case II, $x$ and $y$ are close in local directions but far apart in one of the non-local,
yet stabilization directions; and finally in Case III, $x$ and $y$ are close in every direction. We have depicted these cases in Figure \ref{fig:cases} below.

\noindent \textbf{Step 4:}  In Case I, there is a large number of possible $y$-values, but we can prove, using $k$-locality and exponential stabilization---since any scores
must come from the case when $R_n(x) + R_n(y) > n^\eps$---that the fourth moment of $D_{xy}^2\Sigma(\PP_n)$ can
be bounded by $\e^{-\beta_1 n^{\eps} / 8}$, and hence this dominates the order $n^{d-k}$ of possible $y$-values.

\noindent \textbf{Step 5:} In Case II, the idea is once more to split into whether $R_n(x)+R_n(y) > n^\eps$ or $R_n(x)+R_n(y) \le n^\eps$. In the first case, we obtain an exponential bound
from exponential stabilization, and hence it suffices to consider the second case. Here the stabilization cubes are sufficiently small, which, since $x$ and $y$ are far apart, implies these cubes are disjoint.
Thus, in broad strokes, the only score contributions come from Poisson points in $Q(y,R_n)$ and the expected number of these is of constant order, and hence the fourth moment bound contribution is $n^{\eps}$,
which comes from the sub-polynomial moment assumption.

\noindent \textbf{Step 6:} In Case III, there are only a constant number of possible $y$-values, but since the stabilization cubes of $x$ and $y$ are no longer disjoint, the addition of $x$ and $y$
can cause changes inside the entire column $\mathbb S_n^k(x,1)$ around $x$. Thus, we obtain a bound on the fourth moment of $D_{xy}^2\Sigma(\PP_n)$ of the form $n^{\eps}\vert S_n^k \vert$.

\noindent \textbf{Step 7:}  Lastly, we split the $y$-integral in $I_{n,1}$ and $I_{n,2}$ into Cases I--III, i.e.,
$$
\int \big(\int \cdots \text dy\big)^2 \ \text dx = \int \big(\int_{ \W_n \setminus \S_n^k(x,n^\eps)} \cdots \text dy + \int_{ \S_n^k(x,n^\eps) \setminus \Q(x,n^\eps)} \cdots \text dy + \int_{ \Q(x,n^\eps)} \cdots \text dy\big)^2 \ \text dx,
$$
and inserting the $4^{\text{th}}$ moment bounds on $D_{xy}^2$ from Steps 4–6 as well as the $4^{\text{th}}$ moment bound on $D_{x}$ from Step 2, this yields the
final bounds in Lemma \ref{lem:en1}(i)–(ii).

\ \vspace{0.5ex}

\begin{figure}[h!]
    \centering
    \begin{subfigure}[b]{0.48\textwidth}
        \centering
\begin{tikzpicture}[scale=0.85,
    cgreen/.style={fill=green!20, opacity=0.3},
    cblue/.style={fill=cyan!20, opacity=0.3},
    cyellow/.style={fill=blue!50, opacity=0.2},
    roman/.style={text=gray!60, font=\large\bfseries},
    thickline/.style={draw=black, thick}
]
\draw[thickline] (1,0) rectangle (3,6);
\node[roman] at (2.5, 5.5) {I};
\draw[thickline] (6,0) rectangle (9,6);
\node[roman] at (8.5, 5.5) {I};
\draw[thickline, cblue] (3,0) rectangle (6,1);
\node[roman] at (5.5, 0.5) {II};
\draw[thickline, cyellow] (3,1) rectangle (6,4);
\node[roman] at (5.5, 3.5) {III};
\node at (4, 3.5) {$Q(x,n^\eps)$};
\fill[black] (4.5, 2.5) circle (2pt); 
\node at (4.5, 2.2) {$x$}; 
\draw[thickline, cblue] (3,4) rectangle (6,6);
\node[roman] at (5.5, 5.5) {II};
\node at (4, 5.5) {$S(x,n^\eps)$};
\draw[thickline] (1,0) rectangle (9,6);
\draw[<->, thick] (3.05, 1.6) -- (5.95, 1.6) node[midway, below] {\footnotesize $2n^\varepsilon$};
\end{tikzpicture}
\end{subfigure}  
    \begin{subfigure}[b]{0.48\textwidth}
        \centering
\begin{tikzpicture}[scale=0.85,
    cgreen/.style={fill=green!20, opacity=0.3},
    cblue/.style={fill=cyan!20, opacity=0.3},
    cyellow/.style={fill=blue!50, opacity=0.2},
    roman/.style={text=gray!60, font=\large\bfseries},
    thickline/.style={draw=black, thick}
]
\draw[thickline] (1,0) rectangle (3,6);
\draw[thickline] (6,0) rectangle (9,6);
\draw[thickline, cblue] (3,0) rectangle (6,1);
\draw[thickline, cyellow] (3,1) rectangle (6,4);
\fill[black] (4.5, 2.5) circle (2pt); 
  \draw[dashed,red] (3.75,1.75) rectangle (5.25,3.25);
\node at (4.5, 2.2) {$x$}; 
\draw[thickline, cblue] (3,4) rectangle (6,6);
\draw[thickline] (1,0) rectangle (9,6);
\fill[black] (5.75, 4.25) circle (2pt); 
  \draw[dashed,red] (5,3.5) rectangle (6.5,5);
\node at (5.75, 4.55) {$y$}; 
\fill[black] (6.5, 2.25) circle (2pt); 
  \draw[dashed,red] (5.75,1.5) rectangle (7.25,3);
\node at (6.5, 2.55) {$y$}; 
\end{tikzpicture}
\end{subfigure}

\caption{Left: The white region is $ W_n \setminus S(x,n^\eps)$,
the light blue $S(x,n^\eps) \setminus Q(x,n^\eps)$, and the darker blue $Q(x,n^\eps)$, i.e., Cases I-III as defined in \eqref{eq:three_cases}. 
Right: Illustration of the event that $R_n(x)+R_n(y) \le n^\eps$. When $y$ lies in either Case I or II, this implies that the two non-stable cubes around $x$ and $y$ (red dashed squares verbatim to Figure \ref{fig:partition}) are disjoint.}
\label{fig:cases}
\end{figure}
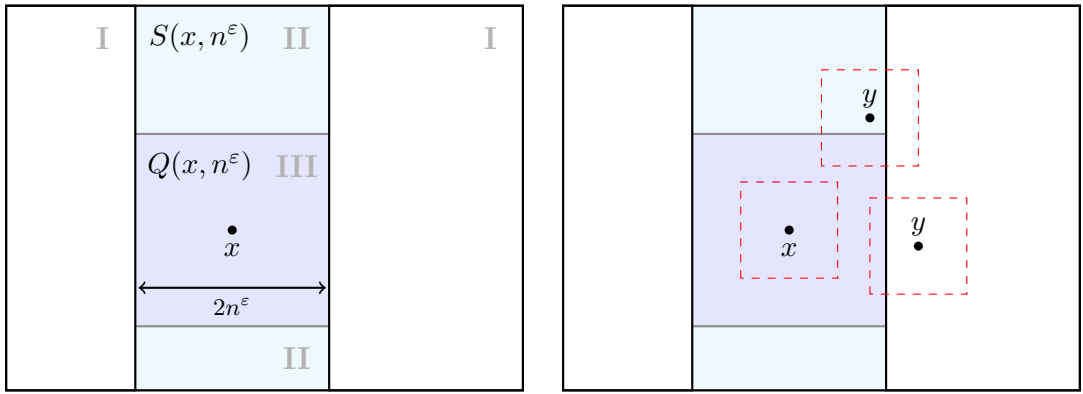

Finally, we outline how the approach differs in the sum-log-sum case. Steps 1, 3, and 7 are the same (with Lemma \ref{lem:en1_sim} instead). However, 
when computing differences in Steps 2, 4, 5, and 6, we run into several new issues. First, since $\mc A_n^+(\PP_n)$ and $\mc A_n^+(\PP_n^x)$ might not contain the same Poisson points, we cannot compare the compound scores (i.e., total scores) 
Poisson point for Poisson point, as before. This is where we need Assumption \ref{ass:B}(i), i.e., stabilization, which ensures that the points where these sets differ lie only far away from $x$. Once we identify the Poisson points in common,
we have a main sum in $D_x$ and $D_{xy}$ with terms, respectively, of the form
$$
\log\Big(\frac{G(Z,\PP_n^x)}{G(Z,\PP_n)}\Big) 
\qquad \text{and} \qquad 
\log\Big(\frac{G(Z,\PP_n^{xy})G(Z,\PP_n)}{G(Z,\PP_n^x)G(Z,\PP_n^y)}\Big).
$$
The main obstacle is that the variance of the sum-log-sum functional is typically much smaller than the variance of the double-sum functional, and thus we need to obtain
tighter bounds on the error terms. Hence, we want to use the inequality $\log(t) \le t^\eps$ rather than $\log(t) \le t$, 
and this is where we need Assumption \ref{ass:B}(ii), i.e., concentration of the compound scores. Additionally, Assumption \ref{ass:B}(ii)
also ensures that, on an event of high probability, the two denominators above are of order $\vert S_n^k \vert$ and $\vert S_n^k \vert^2$, respectively.
In broad strokes, this unlocks the ability to use the inequality $\vert \log(1+t) \vert \le 2 \vert t \vert$, where $\vert t \vert$ not being too large, allows us to follow the same overall approach as in the double-sum case.

\subsection{Extensions and limitations}
\label{sec:limitations}
In this section, we discuss some relevant remarks regarding the results in Theorems \ref{thm:sum_clt} and \ref{thm:log_clt}, including whether
some assumptions can be relaxed, and whether the Normal approximation can be extended to the Kolmogorov distance and/or similar functionals.

\subsubsection{On Assumptions \ref{ass:A} and \ref{ass:B}}

We first discuss the conditions in Assumption \ref{ass:A} one by one.

\noindent \emph{\ref{ass:A}(i); $k$-locality:} One could relax the condition of cut-off locality at distance 1 to any finite distance without changing anything in the overall approach. 
It would also be possible to allow these ‘local’ directions to have non-zero scores among any points in the window, as long as the probability of a non-zero score for two points far apart has exponentially
decreasing tails in the distance between the points.

\noindent \emph{\ref{ass:A}(ii); Exponential stabilization of $f$:} This is a key feature of our approach, but one could possibly relax the assumption that the stabilization radius has
exponential tails to sub-polynomial tails, at the cost of more delicate arguments in bounding the error terms.

\noindent \emph{\ref{ass:A}(iii); Sub-polynomial moments:} It suffices to require that $\overline f_{\sup}$ has exponential tails, which implies sub-polynomial
moments, as also shown inside the proof of Lemma \ref{lem:poisson_ball_bound}. We choose the sub-polynomial moment condition to make the exposition clearer.

We now move on to the conditions in Assumption \ref{ass:B}.

\noindent \emph{\ref{ass:B}(i); Additional stabilization with respect to $\mc A_n^+(\PP_n)$:} Similar to \ref{ass:A}(ii), this could be relaxed to polynomial tails of sufficiently high negative order for the stabilization radius.

\noindent \emph{\ref{ass:B}(ii); Concentration of $G$:} Similarly to \ref{ass:A}(ii) and \ref{ass:B}(i), this could possibly be relaxed to requiring polynomial concentration bounds rather than exponential bounds.

Finally, we consider to what extent the underlying point process can be generalized.

\noindent \emph{Unit Poisson input $\PP_n$:} It would be straightforward to extend the results to stationary Poisson processes with intensity $\lambda > 0$ different from 1.
We further claim that it would also be possible to extend the results to inhomogeneous Poisson processes where the intensity function is bounded away from 0 and from infinity.
However, when the intensity function is allowed to approach zero, it affects the likelihood of the shield configurations that are utilized in Section \ref{sec:3} to control the variance.
Similarly, if the intensity function approaches infinity, the void regions in the stabilization arguments in Section \ref{sec:3} become too unlikely.

\noindent \emph{Rectangular window $W_n$:} Changing the shape of the window in
the local directions would not have a major effect on the approach, but the current arguments in Section \ref{sec:3} that verify both types of stabilization
as well as concentration of $G$ rely heavily on the box structure of the window in the non-local directions for discretization, construction of shields,
and the subsequent Bernoulli trials of equal probability. If the window $W_n$ was, e.g., a ball instead, the point process would behave differently near the boundary. 
However, if one were to replace the slabs $S_n^k$ by annulus regions of fixed width in the local directions, we conjecture the 
arguments in Section \ref{sec:3} would still hold with $W_n$ as a ball.

\subsubsection{Alternative approaches using existing literature}

In Section \ref{sec:intro}, we mentioned several Normal approximation results for stabilizing functionals of Poisson processes \cite{mal_stab,chinmoy,trauth3,hug,nestmann}.
We cannot apply these results directly in our setting, since our functionals have a lower-dimensional variance contribution. 
One could try to use the Normal approximation for region-stabilizing scores in \cite{chinmoy}. 
However, this would require at least as much effort as our approach using the Normal approximation in \cite{mehler}. 

\subsubsection{Convergence in the Kolmogorov distance}

Recall the \emph{Kolmogorov distance} $d_\text{K}$ between two real random variables $X$ and $Y$,
$$
d_\text{K}(X,Y) = \sup_{t \in \R} \big| P(X \le t) - P(Y \le t) \big|.
$$
From a general bound in \cite{Ross2011_FundamentalsOfSteinsMethod}, for some $C > 0$,
$$
d_K(\Sigma(\PP_n), \mc N(0,1)) \le C \sqrt{d_W(\Sigma(\PP_n), \mc N(0,1))}.
$$
Hence, by Theorem \ref{thm:sum_clt}, $\Sigma(\PP_n)$ also converges
to a normal distribution in Kolmogorov distance with a rate of at least $n^{\delta/2}/\sqrt[4]{\vert W_n \vert}$, provided the variance is large enough.
Similarly, we may use the same bound and Theorem \ref{thm:log_clt} to obtain the same rate for $\Sigma^{\log}_n(\PP_n)$.

But can we obtain near-optimal rates in the Kolmogorov distance as well? As a partial answer, we explore the additional error terms
$I_{n,4}$, $\sqrt{I_{n,5}}$, and $\sqrt{I_{n,6}}$ in \cite{mehler}, defined as:
\begin{align*}
	I_{n,4}  = & \frac{1}{2} \sqrt[4]{\E[\Sigma(\PP_n)^4]} \int
     \sqrt[4]{\E\big[D_x \Sigma(\PP_n)^4\big]^{3}} \text dx,\\
     I_{n, 5}  = & \int\E\big[D_x \Sigma(\PP_n)^4\big] \text dx,\\
	I_{n, 6}  = & \int 6 \sqrt{\E\big[D_x \Sigma(\PP_n)^4\big]} \sqrt{\E\big[D_{xy}^2 \Sigma(\PP_n)^4\big]}
    + 3\E\big[D_{xy}^2 \Sigma(\PP_n)^4\big] \ \text d(x,y) .
\end{align*}
Here, when applying \cite[Theorem 1.1]{mehler}, the slowest of these terms determines the rate of convergence in the Kolmogorov distance. 
With our current bounds on moments of $D_x$ and $D_{xy}^2$ in Section \ref{sec:4}, and the imposed lower bound $\vert W_n \vert \vert S_n^k \vert^2$ on $\V[\Sigma(\PP_n)]$,
a straightforward calculation shows that $\sqrt{I_{n,5}}$ always gives the near-optimal contribution $n^{\delta}/\sqrt{\vert W_n \vert}$, and
likewise for $\sqrt{I_{n,6}}$. For $I_{n,4}$, we believe that an application of the Poincare inequality
yields that $\sqrt[4]{\E[\Sigma(\PP_n)^4]}$ will be of order at most $n^\eps$, and if this assertion is true, then $I_{n,4}$
also yields a near-optimal contribution. We omit a formal treatment of $\sqrt[4]{\E[\Sigma(\PP_n)^4]}$, but we speculate that it is possible to obtain
near-optimal rates in the Kolmogorov distance as well for our examples in Section \ref{sec:3}.

\subsubsection{The sum-product, double-product and triple-sum functionals}

One can also ask whether the approach in this paper can be extended to similar functionals. For instance, consider the \emph{sum-product functional}
and \emph{double-product functional},
$$
\sum_{Z \in \PP_n} \prod_{V \in \PP_n} f(Z,V,\PP_n),
$$
$$
\prod_{Z \in \PP_n} \prod_{V \in \PP_n} f(Z,V,\PP_n).
$$
If we log-transform the double-product functional, we end up with a double-sum functional
with score function $\widetilde f(Z,V,\PP_n) = \log f(Z,V,\PP_n)$, and, if
$\widetilde f$ satisfies Assumption \ref{ass:A}, we can apply Theorem \ref{thm:sum_clt} directly.
To that end, we see that $\widetilde f$ satisfies Assumption \ref{ass:A} if the original score function $f$ takes values in $[1,\infty)$, takes the value $1$ on the diagonal, and satisfies $k$-locality and stabilization,
where the score is $1$ rather than $0$ outside the vertical slab and non-stable cube, respectively, as well as
sub-polynomial moments, just as before. 
Subsequently, we would also obtain asymptotics for the double-product functional itself via the Delta Method. As for the sum-product functional, it is
not clear how one could apply Theorems \ref{thm:sum_clt} or \ref{thm:log_clt} directly, and hence this would require further ideas. Also, since the authors are 
not currently aware of any particular application of these functionals in random geometric structures, we leave this for future work. Another possible
extension is to higher-order functionals involving sums and products over triples or more Poisson points. Consider the \emph{triple-sum functional} with
scores $f(Z,V,W,\PP_n)$. Then, e.g., the first-order difference operator is of the form
$$
\sum_{V \in \PP_n} \sum_{W \in \PP_n} f(x,V,W,\PP_n^x)
\;+\;
\sum_{Z \in \PP_n} \sum_{V \in \PP_n} \sum_{W \in \PP_n}
\big( f(Z,V,W,\PP_n^x) - f(Z,V,W,\PP_n) \big),
$$
provided that $f$ is also symmetric and vanishes when any two of the three points are the same.
Thus, with modified definitions of $k$-locality and exponential stabilization for $f$ in three points, the authors conjecture that it should be possible to control
the difference operators with the same overall approach and use \cite{mehler} to obtain a Normal approximation result. Again, without specific applications in mind, we
leave the details for future work.

\section{Examples}
\label{sec:3}

In this section, we apply Theorems \ref{thm:sum_clt} and \ref{thm:log_clt} to several examples involving random geometric structures found in the literature, all of which were already briefly introduced in Section \ref{sec:intro}.
First, in Section \ref{sec:3.1}, we record two useful tools that helps us establish sub-polynomial moments and control the variance.
In Section \ref{sec:3.2}, we establish Normal approximation of the \emph{crossing number} which arises from projecting random graphs in $\R^d$ to a 2-dimensional plane.
In Section \ref{sec:3.3}, we study the \emph{inversion count} when each Poisson point is assigned a barcode length, i.e., a lifetime. We study two types of lifetime models, namely
that they are independent uniform random variables as well as stemming from the Poisson tree model. For both models, we verify Assumption \ref{ass:A}, and hence obtain Normal approximation of the inversion count.
In Section \ref{sec:3.4}, we study the \emph{tree realization number}, which involves a product that we log-transform,
where we then have to verify Assumption \ref{ass:B} to obtain Normal approximation according to Theorem \ref{thm:log_clt}.
We also discuss if we can also extend the methodology to directed- and radial spanning forests as lifetime models.

\subsection{Preliminary tools}
\label{sec:3.1}

In this section, we introduce some preliminary tools that will be useful for our Normal-approximation verification.
First, in the case of crossing numbers, where the score function is \textit{not} bounded, the following
lemma is useful for a common approach to determining whether Assumption \ref{ass:A}(iii), i.e., sub-polynomial moments, is satisfied. Recall that $\mc R$ is the family of all one-point and two-point sets in $\R^d$ along with the empty set, see \eqref{eq:all_pairs}.
\begin{lemma}[Sub-polynomial moments criteria]
    \label{lem:poisson_ball_bound}
    Assume there exists $\ell \in \mathbb{N}$ such that
      \begin{equation}
        \label{eq:f_bound}
    \overline f_{\sup}(\PP_n) \le \sup_{\mathrm{x}\in \mc R } 
    \sup_{Z,V \in \PP_n^\mathrm{x}} 
    \PP^{\mathrm x}\big(B(\dot Z,\ell R_Z)\big)\,
    \PP^{\mathrm x}\big(B(\dot V,\ell R_V)\big),
    \end{equation}
    where $R_Z, R_V$ denote random variables satisfying that for any $\eps \in (0,1)$,
    $$
    P(R_Z > n^\eps) \vee P(R_V > n^\eps) \le \e^{- \gamma(\eps)\, n^\eps},
    $$
    for some $\gamma(\eps) > 0$ and $n \gg 1$.
    Then, Assumption \ref{ass:A}(iii) is fulfilled.
\end{lemma}

The proof of Lemma \ref{lem:poisson_ball_bound} is found in Appendix A. Henceforth, consider the case where 
\begin{equation}
    \label{eq:W_n_simple}
 W_n = [0,n] \times [0,a_2 n] \times \cdots \times [0,a_d n],
\end{equation}
i.e., all side lengths are proportional to $n$. Then, for $r > 0$ and $a_1 = 1$, let
$$
\alpha_r = \prod_{i=1}^d \Big\lceil \frac{a_i}{r} \Big\rceil,
$$
and let $\{Q_{n,j,r} \colon 1 \le j \le \alpha_r n^d\}$ denote a lexicographic ordering of the smallest partition covering $W_n$
with $\alpha_r n^d$ equal boxes of constant volume (Figure \ref{fig:cubes}). To be precise, let
\begin{equation}
    \label{eq:box_Q}
Q_{n,j,r} = \bigtimes_{i=1}^d 
\Big[
\Big\lceil \frac{j - 1}{n^{i-1}} \Big\rceil_n r a_i,\,
\Big(\Big\lceil \frac{j - 1}{n^{i-1}} \Big\rceil_n + 1 \Big) r a_i
\Big],
\end{equation}
where $\lceil \cdot \rceil_n$ denotes rounding down to the nearest integer and then taking the remainder modulo $n$.  
For example, when $d=2$, $n=3$, and $j=6$, then
$
Q_{3,6,r} = [2r,3r] \times [r a_2, 2 r a_2].
$
\begin{figure}[h!]
    \centering
\begin{tikzpicture}[
    scale=0.75,
    >=stealth, 
    thickline/.style={draw=black, thick},
    dashedline/.style={draw=black, dashed},
    dimarrow/.style={<->, thin}
]

\foreach \y [count=\row] in {0, 2, 4} {
    \foreach \x [count=\col] in {0, 4, 8} {
        
        \pgfmathtruncatemacro{\j}{(\row-1)*3 + \col}

        \node[font=\large] at (\x + 2, \y + 1) {$Q_{n,\j,r}$};
    }
}

\draw[dashedline] (4,0) -- (4,6);
\draw[dashedline] (8,0) -- (8,6);

\draw[dashedline] (0,2) -- (12,2);
\draw[dashedline] (0,4) -- (12,4);

\draw[dashedline] (0,0) rectangle (12,6);
\draw[thickline] (0,0) rectangle (11.5,5.5);

\draw[dimarrow] (0, -0.25) -- (11.5, -0.25) node[midway, below] {$n$};

\draw[dimarrow] (-0.25, 0) -- (-0.25, 5.5) node[midway, left] {$a_2 n$};

\draw[dimarrow] (4.15, 2.1) -- (7.9, 2.1) node[midway, above] {\small $r$};

\draw[dimarrow] (4.1, 2.15) -- (4.1, 3.9) node[midway, right] {\small $r a_2$};

\end{tikzpicture}
    \caption{Illustration of the ordering of the cubes $Q_{n,j,r}$ covering $ W_n$ in the case $d=2$. All the cubes have the same volume, which doesn't depend on $n$.}
    \label{fig:cubes}
\end{figure}
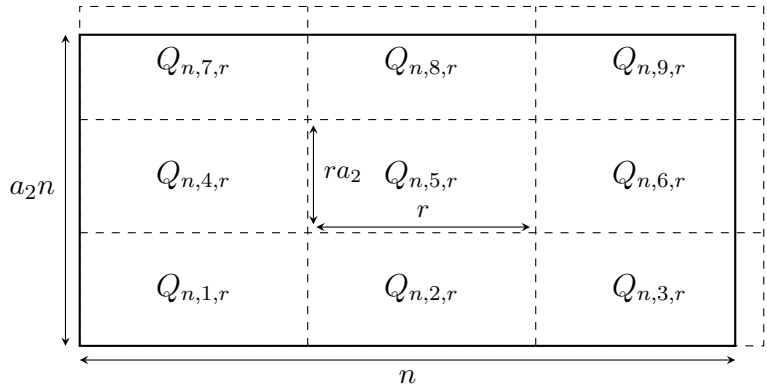

For bounding the variance of $\Sigma(\PP_n)$ and $\Sigma^{\log}_n(\PP_n)$, it will be useful to consider a martingale decomposition of these functionals in terms of the information contained in each of these boxes.
To that end, for a Borel set $A \subseteq \R^d$, let $\mathfrak{N}(A)$ denote the locally finite counting measures on $A \times \M$, and define for $1 \le j \le \alpha_r n^d$ the quantities
\begin{align*}
\mathfrak N_{n,j}^- & = \mathfrak{N}\Big(\bigcup_{i=1}^{j-1} Q_{n,i,r}\Big), \\
\mathfrak N_{n,j}^+ & = \mathfrak{N}\Big(\bigcup_{i=j+1}^{\alpha_r n^d} Q_{n,i,r}\Big),
\end{align*}
where $\mathfrak N_{n,j}^-$ will serve as the possible configurations in the \emph{explored space}
and $\mathfrak N_{n,j}^+$ as the configurations in the \emph{unexplored space}. We can then prove the following variance bound.

\begin{lemma}[Variance lower bound]
\label{lem:var_bound}
For $j \le \alpha_r n^d$, let $\mc B_{n,j}^{(1)}, \mc B_{n,j}^{(2)}, \mc B_{n,j}^{(3)}$ denote measurable subsets of $\mathfrak{N}(Q_{n,j,r})$
and let
\begin{equation}
    \label{eq:index_set}
I_{n,j}^k = \{  j + 1 \le i \le \alpha_r n^d \colon \pi_{1:k}(Q_{n,i,r}) = \pi_{1:k}(Q_{n,j,r}), \ Q_{n,i,r} \subseteq W_n \}.
\end{equation}
Assume that
\begin{enumerate}[label=(V\arabic*)]
    \item 
    $
\min_{i \in \{1, 2, 3\}}    \inf_{n \ge 1} \inf_{1 \le j \le \alpha_r n^d} 
    P\big(\PP \cap (Q_{n,j,r} \times \M) \in \mc B_{n,j}^{(i)}\big) > 0,
    $
    \item for any $j < \alpha_r n^d/2$ and any $(\omega_0,\omega_1,\omega_2,\omega_3) \in \mathfrak N_{n,j}^- \times \mc B_{n,j}^{(1)} \times \mc B_{n,j}^{(2)} \times \mathfrak N_{n,j}^+$,
    $$
    \big\vert \Sigma(\omega_0 \cup \omega_2 \cup \omega_3) - \Sigma(\omega_0 \cup \omega_1 \cup \omega_3) \big\vert
    \ge 
    \# \big\{ i \in I_{n,j}^k \colon \omega_3 \cap (Q_{n,i,r} \times \M) \in \mc B_{n,i}^{(3)} \big\}.
    $$
\end{enumerate}
Then, $\V[\Sigma(\PP_n)] \ge C n^{3d-2k}$ and $\V[\Sigma^{\log}_n(\PP_n)] \ge C n^{d}$ for some $C > 0$ and $n \gg 1$.
\end{lemma}

The proof of Lemma \ref{lem:var_bound} is postponed until Appendix A. Loosely speaking, the way to apply Lemma \ref{lem:var_bound} 
is to consider two types of configurations of the $j$th box that both occur with positive probability (e.g., the box being empty).
Then, consider a third type of configuration in all the boxes “above’’ the $j$th box (e.g., the presence of a certain edge), where each time this third configuration occurs, there
will be at least one score with respect to the $j$th box when configured the second way, but no scores when configured the first way. The discretization into disjoint boxes
and the corresponding Bernoulli trial in each box thus ensures that the variance will be of the desired order.

Finally, we can use the following Binomial concentration inequality from \cite[Lemma 1.1]{Penrose2003_RandomGeometricGraphs}
when we want to verify the concentration property in Assumption \ref{ass:B}(ii).
\begin{lemma}[Binomial concentration inequality]
    \label{lem:binom_conc}
For any Binomial random variable $X$ with parameters $m \in \N$ and $p \in (0,1)$,
$$
P(X < \tfrac{mp}{2}) \le \exp\!\Big( - mp\Big(\tfrac{1}{2}+\tfrac{1}{2}\log(\tfrac{1}{2})\Big)\Big).
$$
\end{lemma}
The concentration inequality in Lemma \ref{lem:binom_conc} comes from applying \cite[Lemma 1.1]{Penrose2003_RandomGeometricGraphs} with $k = \lceil mp/2 \rceil$.
As an example, when studying the tree realization number for independent lifetimes, we can discretize the window $\S_n^k$ into boxes, lower bound $G(Z,\PP_n)$ 
by the number of boxes satisfying a certain property (e.g., containing a bar with a long lifetime), and then use Lemma \ref{lem:binom_conc} 
to obtain the exponential concentration in \ref{ass:B}(ii). 
We now proceed to applying Lemmas \ref{lem:poisson_ball_bound}--\ref{lem:binom_conc} to the examples mentioned at the beginning of the section.

\subsection{The crossing number from planar projections}
\label{sec:3.2}
To motivate the study of crossing numbers, while also making the exposition more accessible, 
we rely on a similar introduction as in \cite{doring}.

The crossing number of a graph $G$ is the minimal number of intersecting edges among all drawings of the graph $G$ in a plane. 
The question is based on Tur\'an's brick problem asking for the least number of crossing tracks between kilns and storage sites,
 in other words the crossing number of a bipartite graph. The problem generalizes to arbitrary graphs (e.g., see Figure \ref{fig:projections}).
  Restricting to drawings with straight edges only is another interesting and complicated optimization problem.
   To distinguish the cases one often calls this solution the rectilinear crossing number. 
   Crossing numbers are relevant in computer science for chip design and graph drawing as well as in mathematics, too.
As shown in \cite{GareyJohnson1983_CrossingNumberNPComplete},
the problem to determine the crossing number of any given graph is NP-hard. 
Even in the particular case of a complete graphs $K_n$ on $n$ vertices,
 for high $n$, there exists only a conjecture on the crossing number and for the rectilinear crossing number even that is not known and there are only bounds available.
  This makes approximation algorithms all the more important.

  In  \cite{chimani2018crossing} it could be shown that for the random geometric graph,
   the projection to a fixed plane yields a constant factor 
  approximation for the \emph{rectilinear crossing number}. 
  We will from now on focus on the number of crossings in the projection 
  and abbreviate this by the name \emph{crossing number}. 
  Normal approximation of the crossing number in this setting is 
  shown in \cite{doring}.  
  
In this section, we study three types of random graphs: First, we consider a fixed cut-off radius of $1$ in all directions,
\textit{the random geometric graph}, which is already been studied
in \cite{doring} in the unit cube with intensity $t > 0$, and where
a Normal approximation for the crossing number is also found in \cite[Proposition 3.5]{doring}. Next, we will let the radius be random in non-projection directions
such that in expectation it is still $1$,
and let this radius have exponentially decreasing tails, where the crossing number is yet to be studied. If we want to maintain a undirected graph,
this essentially boils down to considering a \textit{max-kernel} \cite{GracarEtAl2021_percolation}, where we require both points to be within each other's radius to form an edge.
Finally, we move to a more general case and allow the radius to depend on other Poisson points as well as the additional randomness from before.

To make the above ideas concise, let $\PP_n$ denote
a Poisson point process in $[0,n]^d \times \M$ with intensity measure 
$
\vert \cdot \vert \otimes \mu,
$

where $\M$ and $\mu$ can be tailored to the random graph. As usual, for any $Z \in \PP_n$, we think of $\dot Z$ as the spatial location of the point in $\R^d$.

For points $p,p' \in \R^d$, let $[p,p']$ denote the line segment between these point, and recall that
$\S_n^2(x,1)$ denotes the vertical slab around $x$ of width 2 in the first 2 coordinates and unrestricted width
in the $d-2$ remaining coordinates and arbitrary marks. Define the score function,
$$
f(Z,V,\PP_n) = \frac 1 8 \one\{Z \neq V\} \sum_{ Z' \in \PP_n \cap \S_n^{2}(Z,1)} \sum_{V' \in \PP_n \cap \S_n^{2}(V,1)} h(Z,Z',V,V',\PP_n),
$$
where for some Borel set $\mc C_n \subseteq (\R^d \times \M)^4$ representing a connectivity condition,
$$
h(Z,Z',V,V',\PP_n) = \one\{ (Z,Z',V,V') \in \mc C_n \} \one\{ \pi_{1:2}([\dot Z,\dot Z']) \cap \pi_{1:2}([\dot V,\dot V']) \neq \emptyset \}.
$$
In other words, provided the connectivity condition $\mc C_n$ is met, the binary value of $h$ tells us whether
the edges $[\dot Z,\dot Z']$ and $[\dot V,\dot V']$ cross when projected down to the 2-dimensional plane.
As an example $\mc C_n$ could be the condition that the points $Z,Z'$ and $V,V'$ are, respectively, within distance 1 of one another. Note that $f$ exhibits 2-locality by construction. Subsequently,
$$
\Sigma (\PP_n) = \sum_{Z \in\PP_n}\sum_{V \in\PP_n}f(Z,V,\PP_n),
$$
counts the total number of crossings of edges formed by the
connection condition $ \mc C_n$ when projected down to the 2-dimensional plane. We now proceed to first consider the already studied case of the number of crossings of a fixed connectivity radius.

\begin{figure}[h!]
    \centering

\begin{tikzpicture}[
  x  = {(-1.0cm,-1.0cm)},
  y  = {(2cm,-0.3cm)},
  z  = {(0cm,2cm)},
  line cap=round, line join=round, scale = 0.88
]

\def\xmin{0}
\def\xmax{2}
\def\ymin{0}
\def\ymax{2}
\def\zmin{0}
\def\zmax{2}
\def\xL{2.85}

\begin{scope}[canvas is yx plane at z=\zmin]
  \path[fill=black!15, opacity=0.45] (\ymin,\xmin) rectangle (\ymax,\xmax);
\end{scope}

\begin{scope}[canvas is zy plane at x=\xmin]
  \path[fill=white, opacity=0.25] (\zmin,\ymin) rectangle (\zmax,\ymax);
  \draw[blue!40] (\zmin,\ymin) rectangle (\zmax,\ymax);
\end{scope}

\begin{scope}[canvas is zy plane at x=\xmax]
  \draw[blue!40] (\zmin,\ymin) rectangle (\zmax,\ymax);
\end{scope}

\begin{scope}[canvas is zx plane at y=\ymin]
  \draw[blue!40] (\zmin,\xmin) rectangle (\zmax,\xmax);
\end{scope}

\begin{scope}[canvas is zx plane at y=\ymax]
  \draw[blue!40] (\zmin,\xmin) rectangle (\zmax,\xmax);
\end{scope}

\begin{scope}[canvas is yx plane at z=-1.5]
  \path[fill=green!30, opacity=0.55]
    (\ymin-0.15,\xmin-0.15) rectangle (\ymax+0.25,\xmax+0.25);
\end{scope}

\coordinate (a) at (1.10, 0.2, 1.90); 
\coordinate (b) at (0.46, 0.60, 1.70); 
\coordinate (c) at (0.46, 1.72, 0.82); 
\coordinate (f) at (0.88, 1.38, 1.77); 
\coordinate (h) at (0.98, 1.17, 0.75); 
\coordinate (i) at (1.39, 1.8, 0.84); 
\coordinate (e) at (0.71, 0.53, 0.01); 
\coordinate (j) at (0.48, 0.11, 0.24); 
\coordinate (d) at (1.75, 0.4, 0.10);

\coordinate (aL) at (1.10,0.2 ,-1.5);
\coordinate (bL) at (0.46,0.6,-1.5);
\coordinate (cL) at (0.46, 1.72,-1.5);
\coordinate (dL) at (1.75,0.4,-1.5);
\coordinate (fL) at (0.88, 1.38,-1.5);
\coordinate (hL) at (0.98, 1.17,-1.5);

\coordinate (eL) at (0.71, 0.53,-1.5);
\coordinate (iL) at (1.39, 1.8,-1.5);
\coordinate (jL) at (0.48, 0.11,-1.5);

\draw[blue!70!black, line width=0.9pt] (a)--(b);

\draw[blue!70!black, line width=0.9pt] (b)--(f);
\draw[blue!70!black, line width=0.9pt] (c)--(i);
\draw[blue!70!black, line width=0.9pt] (c)--(h);

\draw[blue!70!black, line width=0.9pt] (e)--(j);

\draw[blue!70!black, line width=0.9pt] (h)--(i);

\foreach \p in {a,b,c,d,e,f,h,i,j}{
  \fill[blue!70!black] (\p) circle (1.6pt);
}

\foreach \P/\PL in {a/aL,b/bL,c/cL,d/dL,e/eL,f/fL,h/hL,i/iL,j/jL}{
  \draw[red!70, line width=0.6pt] (\P) -- (\PL);
  \fill[red!70] (\PL) circle (1.4pt);
}

\draw[green!60!black, line width=0.9pt] (aL)--(bL);

\draw[green!60!black, line width=0.9pt] (bL)--(fL);
\draw[green!60!black, line width=0.9pt] (cL)--(iL);
\draw[green!60!black, line width=0.9pt] (cL)--(hL);

\draw[green!60!black, line width=0.9pt] (eL)--(jL);

\draw[green!60!black, line width=0.9pt] (hL)--(iL);

\node[blue!70!black]  at (0.05,-1.25,1.35) {$W_n$};

\node[green!60!black] at (0.5,-0.8,-1.5) {$\pi_{1:2}(W_n)$};

\end{tikzpicture}
    \caption{Illustration of projections of a random 3-dimensional geometric graph onto a green 2-dimensional plane. 
    The crossing number counts the number of intersections between edges in the green projection plane, which is 2 for this particular graph.}
    \label{fig:projections}
\end{figure}
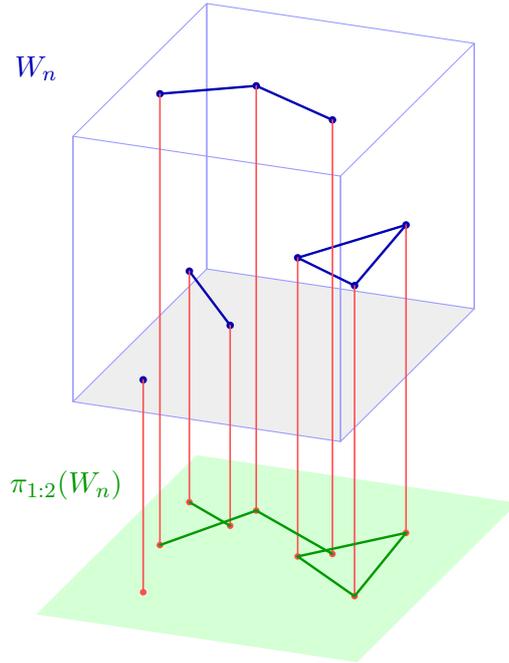
\subsubsection{Fixed radius}
In this section, we consider the random geometric graph in the \emph{thermodynamic regime}, i.e. where the 
typical degree of a vertex remains (or converges to a) constant when the number of points increases. More precisely,
let $\M = \emptyset$ and consider the condition $\mc C_n^{(1)}$ defined as
$$
\mc C_n^{(1)} = \big\{ (Z,Z',V,V') \colon \Vert Z - Z' \Vert \le 1,\ \Vert V - V' \Vert \le 1 \big\},
$$
where we use that notationally $Z$ coincides with $\dot Z$. Note that since
$$
\sup_{\mathrm{x}\in \mc R} f(Z,V,\PP_n^\mathrm{x}) \le \PP\big(B(Z, 1)) \PP\big(B( V, 1))
$$
for any $Z,V \in \PP_n$ and large $n \ge 1$, it follows by Lemma \ref{lem:poisson_ball_bound} that
$f$ satisfies Assumption \ref{ass:A}(iii), i.e. has sub-polynomial moments. Moreover,
note that by adding a point $x \in [0,n]^d$, it can only
add more edges, i.e., the edges $[x,V]$ for all 
$V \in \PP_n \cap B(x,1)$. However, for 
$Z,V \in \PP_n \cap B(x,1)^c$, the addition of $x$ adds no
new edges and hence the number of crossings remain the same,
i.e.,
$$
f(Z,V,\PP_n^x)=f(Z,V,\PP_n).
$$
Thus, for every $n \ge 1$, $R_n(x) \le 1$ 
and hence $f$ exhibits exponential stabilization, i.e. satisfies Assumption \ref{ass:A}(ii). Finally with $t=n^d$, it
follows by \cite[Lemma 3.6]{doring} that
\begin{equation}
    \label{eq:fixed_radius_variance}
\V[\Sigma(\PP_n)]  \ge  C n^{3d - 4}
\end{equation}
for some $C>0$ and all sufficiently large $n \ge 1$.
Thus, invoking Theorem \ref{thm:sum_clt},
$$
d_W\Big(\f{\Sigma(\PP_n)- \E[\Sigma(\PP_n)]}{\sqrt{\V[\Sigma(\PP_n)]}}, \mc N(0,1)\Big)
\le n^{-d/2 + \delta},
$$
for any $\delta > 0$, i.e. the crossing number asymptotically follows a Normal distribution, and we have a near-optimal bound on the rate of convergence.

\subsubsection{Random radius with exponential tails}
Consider now the case of $\M = [0,\infty)$ and let $\mu$ denote a probability measure hereon with exponential tails, i.e.,
$$
\mu((s,\infty)) \le \e^{-\gamma s}
$$
for some $\gamma >0$ and all $s \gg 1$. Assume also that $\mu([1,\infty)) > 0$. Examples of such distributions $\mu$ could be the uniform and exponential distribution. Then, we let the connectivity condition $\mc C_n$ be that that edges are formed between points $Z'$ and the 'core' $Z$
as long as the distance is less than mark $R_Z$ from $\mu$ associated to $\dot Z$, i.e., let
$$
\mc C_n^{(2)} = \big\{ (Z,Z',V,V') \colon \Vert \dot Z -  \dot Z' \Vert \le R_{Z},\ \Vert \dot V - \dot V' \Vert \le R_{V} \big\}.
$$
Note that in general, the graph generated by this condition is a \emph{directed} graph as
$R_Z$ and $R_{Z'}$ may be different. If for modelling reasons, we instead wanted an \emph{undirected} graph as in the fixed radius case,
we could just replace $\mc C_n$ by the connectivity condition,
 $$
\widetilde{ \mc{C}}_n^{(2)} = \big\{ (Z,Z',V,V') \colon \Vert \dot Z - \dot Z' \Vert \le R_{Z} \wedge R_{Z'},\ \Vert \dot V - \dot V' \Vert \le R_{V} \wedge R_{V'} \big\}.
$$
The condition $\widetilde{\mc{C}}^{(2)}_n$ is sometimes called the \emph{max-kernel} in random connection models \cite{GracarEtAl2021_percolation}. 
The reason it is called the max-kernel rather than the min-kernel is that usually the connection threshold is on the form $R = s^{-\gamma}$ for some $\gamma > 0$
where $s \in (0,1)$ and hence 
$$
R_1 \wedge R_2 = (s_1 \vee s_2)^{-\gamma}.
$$
Since we are still only considering
$Z' \in \S_n^2(Z,1)$, we are only considering the crossings of a subset of all edges in the connection model. However, if we instead were to change $\mu$ to a distribution supported on $[0,\ell]$
for some $\ell \ge 1$, then modifying $k$-locality slightly to consider slabs of width $2 \ell$ instead of width $2$, we would indeed be looking at all edges.  

We proceed to verify the remaining conditions in Assumption $A$ for $\mc C_n^{(2)}$. Since
$$
\sup_{\mathrm{x}\in \mc R} f(Z,V,\PP_n^\mathrm{x}) \le \PP\big(B(\dot Z,R_Z)\big) \PP\big(B(\dot V,R_V)\big)
$$
for any $Z,V \in \PP_n$, it follows by Lemma \ref{lem:poisson_ball_bound} that
$f$ satisfies Assumption \ref{ass:A}(iii), i.e., sub-polynomial moments. Next, we let
$$
T_n = \sup_{Z \in \PP_n} R_Z,
$$
and note as before that the insertion of $x$ can only add more edges
and that by construction no edges are formed with $x$ from points outside $Q(x,T_n)$.
Hence $f$ stabilizes inside $Q(x,T_n)$. Moreover,
by the union bound and Mecke's formula, it follows that
$$
P(T_n > n^\eps) \le \int_{[0,n]^d\times [0,\infty)} P(R_x > n^\eps) \ \text d x \le n^d \e^{-\gamma n^\eps},
$$
for any $\eps \in (0,1)$ and $n \gg 1$. Thus, 
as $Q(x,T_n)$ has exponentially decaying tails, it follows that this is also the
case for the non-stable cube.
Thus, $f$ satisfies Assumption \ref{ass:A}(ii).

Finally, we use the bound in Lemma \ref{lem:var_bound} to bound the variance from below. To that end, choose $r$ small enough such that each box $Q_{n,j,r}$
has side length atmost $1/\sqrt{d}$. For each $1 \le j \le \alpha_rn^d$, consider now a further subdivision of $Q_{n,j,r}$ into $4$ subboxes: When shifted
such that the center point of $Q_{n,j,r}$ is the origin, let $Q_{n,j,r}^{++}$ denote the part of the cube with positive first and second coordinates,
$Q_{n,j,r}^{+-}$ positive first and negative second coordinates,
 $Q_{n,j,r}^{-+}$ negative first and positive second coordinates, and $Q_{n,j,r}^{--}$ negative first and second coordinates. With these subboxes,
 we now define the sets of point configurations,
        \begin{equation}
    \label{eq:crossing_configurations}
        \begin{aligned}
        \mc B_{n,j}^{(1)}  = & \{ D \in \mathfrak{N}(Q_{n,j,r}) :   \vert D \vert = 0 \},\\[0.5em]
        \mc B_{n,j}^{(2)}  = & \{ D \in \mathfrak{N}(Q_{n,j,r}) :   \vert D \cap (Q_{n,j,r}^{++} \times [1,\infty)) \vert = \vert D \cap (Q_{n,j,r}^{--} \times [1,\infty)) \vert = 1 \}, \\
         \mc B_{n,j}^{(3)}  = & \{ D \in \mathfrak{N}(Q_{n,j,r}) :   \vert D \cap (Q_{n,j,r}^{+-} \times [1,\infty)) \vert = \vert D \cap (Q_{n,j,r}^{-+} \times [1,\infty)) \vert = 1 \},
        \end{aligned}
        \end{equation}
i.e., $ \mc B_{n,j}^{(2)} $ are the configurations where there is exactly one point with mark at least $1$ in the subbox $Q_{n,j,r}^{++}$ and exactly one point with mark at least $1$ in the subbox $Q_{n,j,r}^{--}$,
while $ \mc B_{n,j}^{(3)} $ are the configurations where the two points instead are in $Q_{n,j,r}^{+-}$ and $Q_{n,j,r}^{-+}$.
Since the boxes $Q_{n,j,r}$ have side length at most $1/\sqrt{d}$, it follows in both cases that these two points must be connected by an edge. Note that by the assumption that
$\mu([1,\infty)) > 0$ and that the volume of box $Q_{n,j,r}$ is constant, then by the Poisson void probabilities assumption (V1) in Lemma \ref{lem:var_bound} is satisfied.
Next, consider the set $I_{n,j}^2$ as defined in \eqref{eq:index_set}.

Note that for any $i \in I_{n,j}^2$, any configuration in $\mc B_{n,i}^{(3)}$ will add one crossing with the edge formed in any configuration of $Q_{n,j,r}$ in $\mc B_{n,j}^{(2)} $. 
Thus, it follows that assumption (V2) in Lemma \ref{lem:var_bound} is satisfied. Hence by Lemma \ref{lem:var_bound}, $\V[\Sigma(\PP_n)] \ge C n^{3d - 4}$,
which by Theorem \ref{thm:sum_clt} implies that for any $\delta > 0$,
$$
d_W\Big(\f{\Sigma(\PP_n)- \E[\Sigma(\PP_n)]}{\sqrt{\V[\Sigma(\PP_n)]}}, \mc N(0,1)\Big)
\le n^{-d/2 + \delta},
$$
i.e., we once more have asymptotic normality of the crossing number. Note that we could also have verified the conditions in Assumption \ref{ass:A} for $\widetilde{\mc{C}}_n^{(2)}$ and obtained the same results.

\subsubsection{Random radius with exponential tails and spatial dependence}
Finally, as a proof of concept, we also briefly study a connection model, where the radius may depend on all other Poisson points. Let
$\mu$ be as  in the previous section, and let $(R_x)_{x \in \R^d}$ denote independent random variables with distribution $\mu$. Then, let 
$\mc S(\PP_n)$ denote a measurable set, which satisfies the following 'localization' property,
$$
x \in \mc S(\PP_n) \iff x \in \mc S(\PP_n \cap B( \dot x,R_x)),
$$
and define the new radius $R_{x,\PP_n}$ as 
$$
R_{x,\PP_n} = R_x \one\{x \in \mc S(\PP_n)\}.
$$
Moreover, let $\mc H_{n,j}$ denote the event that $Z \in \mc S(\PP_n)$ for every $Z \in \PP_n \cap Q_{n,j,r}$, and assume that $\mc H_{n,j}$ is measurable and that

\begin{equation}
    \label{eq:spatial_dependence_assumption}
\inf_{n \ge 1} \inf_{1\le j \le \alpha_rn^d} P(\mc H_{n,j}) > 0,
\end{equation}
As an example, we could for
 $\ell \in \mathbb{N}$, consider 
$$
R_{x,\PP_n} = R_x \one\{\vert \PP_n \cap B(\dot x,R_x) \vert \le \ell \},
$$
which would represent a model, where too many possible connections results in no connections at all.
In this section, we only consider undirected graphs for simplicity,
and thus consider the connectivity condition $\mc C_n^{(3)}$ defined as
$$
\mc C_n^{(3)} = \big\{ (Z,Z',W,W') \colon \Vert \dot Z - \dot Z' \Vert \le R_{Z,\PP_n}  \wedge R_{Z',\PP_n},\ \Vert \dot W - \dot W' \Vert \le R_{W,\PP_n} \wedge R_{W',\PP_n} \big\}.
$$

We now proceed to verify the rest of Assumption \ref{ass:A} for $\mc C_n^{(3)}$.
First, since $R_{x,\PP_n} \le R_x$,
$$
\sup_{\mathrm{x}\in \mc R} f(Z,V,\PP_n^\mathrm{x}) 
\le \PP\big(B(\dot Z, R_Z)\big) \PP\big(B(\dot W,R_W)\big)
$$
for any $Z,V \in \PP_n$, so by Lemma \ref{lem:poisson_ball_bound},
$f$ satisfies Assumption \ref{ass:A}(iii), i.e., has sub-polynomial moments. Next, due to min-kernel structure of $\mc C_n$ and the localization property $\mc S(\PP_n)$, 
it follows that no edges are formed with points outside $Q(x,R_x)$ and hence we have stabilization.
By the exponential tails assumption on $\mu$, it thus follows that Assumption \ref{ass:A}(ii), i.e., exponential stabilization, is satisfied as well.

Finally, let $ \mc B_{n,j}^{(1)}, \mc B_{n,j}^{(2)}, \mc B_{n,j}^{(3)} $
be the same sets of configurations as in \eqref{eq:crossing_configurations} and let $I_{n,j}^2$ be the index set defined in \eqref{eq:index_set}.
Then, by the additional assumption in \eqref{eq:spatial_dependence_assumption},
it follows by the same arguments that both conditions (V1) and (V2) in Lemma \ref{lem:var_bound} are satisfied.
Thus, by Theorem \ref{thm:sum_clt}, we have Normal approximation of the crossing number $\Sigma(\PP_n)$ even when the connection radius exhibits this type of spatial dependence.

\subsection{Barcodes I: The inversion count}
\label{sec:3.3}
The inversion number counts the number of inversions occurring in a permutation and is thereby a classical permutation statistic that plays a central role at the intersection
 of combinatorics and probability. From a statistical perspective, the inversion number and its distributional properties has been investigated most extensively 
 for uniformly drawn permutations.  Stanley (e.g.\ \cite{stanley1997enumerative}) develops the classical combinatorial theory of inversion statistics,
  including generating functions and q‐analogs, while Fulman (e.g.\ \cite{fulman2004stein}) applies Stein’s method to establish a Normal approximation result
  for the inversion number under the uniform permutation model.
Recent research has expanded beyond the classical regime. A central limit theorem for more general permutation statistics, 
including those exhibiting local dependency structures, was established by \cite{chatterjee2017cltnew}. Their framework extends the reach of probabilistic limit theorems to nontraditional or composite statistics on permutations, providing a general method for handling dependencies. 

Beyond classical and asymptotic analysis, the inversion number has appeared in new mathematical contexts.
 In \cite{bouquet2024barcode, jaramillo2023combinatorial}, combinatorial methods are developed for the analysis of persistence barcodes,
  which are one of the core tools of TDA. Here, loosely speaking, a collection of intervals is used to represent the birth and death times of topological features 
  over multiple scales (Figure \ref{fig:bars}). A precise description in the present setting is found below. 
  Permutation-type statistics are of high interest in TDA since the intervals in a barcode can be thought of inducing a permutation on the set of birth/death times.
Additionally, sometimes only parts of the barcode plot are considered. For instance \cite{BruckGarin2023_StratifyingSpaceOfBarcodes} removes the infinite bar. We go a step further and not only drop 
the infinite bar but also those that are longer than some threshold, which we simply set to 1 for presentation reasons.

In this section, we now examine several choices for the
lifetimes when modelling the underlying point cloud as a Poisson point process and prove that for each of these choices, we can prove
Normal approximation of the inversion count (Theorem \ref{thm:sum_clt}). To make the above considerations precise, we define the score function $f$ as 
\begin{equation}
    \label{eq:score_function_inversion}
    f(x, y, \PP_n) = \one\big\{\big(x_1^{(d)} - y_1^{(d)}\big)\big(x_1^{(d)} - y_1^{(d)} + \ell_{x,\PP_n}  - \ell_{y,\PP_n}  \big) <0, \ \ell_{x,\PP_n} , \ell_{y,\PP_n} \in (0,1) \big\},
\end{equation}
where $\ell_{x,\PP_n}, \ell_{y,\PP_n}$ are lifetimes---i.e., barcode lengths---which will either be independent uniform random variables or stem from Poisson trees.
In other words, $(\ref{eq:score_function_inversion})$ is equal to 1 if there is an inversion among $x$ and $y$ and 0 otherwise. 
Note that this score function is constructed to exhibit 1-locality, that is, to not have long-range dependence in the first coordinate as described in \ref{ass:A}(i), which in this section can be thought of as a time component. Since $f$ is bounded by 1,
the score function always automatically have sub-polynomial moments, i.e. satisfy Assumption \ref{ass:A}(iii). We now proceed to study the remaining conditions in Assumption \ref{ass:A} for independent lifetimes.

\begin{figure}[h!]
    \centering

\begin{tikzpicture}[x=1.25cm,y=1cm,>=latex,line cap=round,scale=1.25]

\def\xmin{0.0}
\def\xmax{8.6}

\def\yHa{2.35}
\def\yHb{1.15}
\def\yHc{-0.05}

\definecolor{barH0}{RGB}{130,25,25}   
\definecolor{barH1}{RGB}{18,92,55}    
\definecolor{barH2}{RGB}{120,60,140}  
\definecolor{axisblue}{RGB}{40,90,140}

\foreach \x in {0.5,1.75,3.0,4.25,5.5,6.75,8.0}{
  \draw[axisblue!45,dashed,line width=0.6pt] (\x,-0.55) -- (\x,2.75);
}

\draw[axisblue!60,line width=0.6pt] (\xmin,1.75) -- (\xmax,1.75);
\draw[axisblue!60,line width=0.6pt] (\xmin,0.55) -- (\xmax,0.55);

\foreach \y in {\yHa,\yHb,\yHc}{
  \draw[axisblue!70,->,line width=0.7pt] (\xmin,\y) -- (\xmax,\y);
}

\draw[barH0,line width=1.4pt] (0.15,2.55) -- (0.28,2.55);
\draw[barH0,line width=1.4pt] (1.10,2.45) -- (5.55,2.45);
\draw[barH0,line width=1.4pt] (2.30,2.35) -- (3.10,2.35);
\draw[barH0,line width=1.4pt] (1.80,2.25) -- (1.90,2.25);
\draw[barH0,line width=1.4pt] (3.15,2.15) -- (3.85,2.15);
\draw[barH0,line width=1.4pt] (2.10,2.05) -- (5.90,2.05);
\draw[barH0,line width=1.4pt] (1.10,1.95) -- (1.55,1.95);
\draw[barH0,line width=1.4pt] (0.10,1.85) -- (8.30,1.85);

\draw[barH1,line width=1.6pt] (1.05,1.45) -- (1.55,1.45);
\draw[barH1,line width=1.6pt] (1.25,1.32) -- (2.40,1.32);
\draw[barH1,line width=1.6pt] (1.55,1.20) -- (3.35,1.20);
\draw[barH1,line width=1.6pt] (1.75,1.08) -- (4.10,1.08);
\draw[barH1,line width=1.6pt] (2.15,1.52) -- (3.05,1.52);
\draw[barH1,line width=1.6pt] (2.45,0.96) -- (2.95,0.96);
\draw[barH1,line width=1.6pt] (2.80,1.28) -- (4.60,1.28);
\draw[barH1,line width=1.6pt] (3.10,1.58) -- (3.90,1.58);
\draw[barH1,line width=1.6pt] (3.55,1.36) -- (7,1.36);

\draw[barH1,line width=1.6pt] (4.35,0.92) -- (4.85,0.92);

\draw[barH2,line width=1.8pt] (5.05,0.10) -- (5.55,0.10);
\draw[barH2,line width=1.8pt] (5.35,-0.05) -- (5.90,-0.05);

\end{tikzpicture}

    \caption{The barcode plot in topological data analysis, where each bar represents the lifetime of a topological feature (\cite{BoissonnatChazalYvinec2018_GeometricTopologicalInference}).
    As an example, the green bars might represent the lifetime of loops and the purple bars the lifetime of voids, or the colors could be irrelevant and all bars represent the lifetimes of connected components.
    In this particular barcode plot, e.g., there are at least two inversions between 
    each of the two purple bars and the longest green bar if all of these have length at most 1 as required in \eqref{eq:score_function_inversion}.}
    \label{fig:bars}
\end{figure}
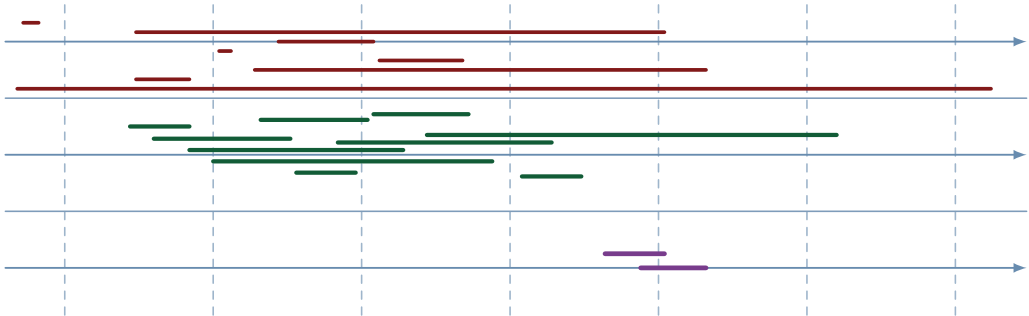

\subsubsection{Lifetimes from independent and uniform variables}
\label{sec:independent_uniform_lifetimes}

We first consider a model where the lifetimes are uniform and independent of eachother as well as of the spatial location of the corresponding point. 

To be precise, let $\PP_n$ denote
a Poisson point process in $ W_n \times [0,1]$ with intensity as the Lebesgue measure on $\W_n \times [0,1]$.
Hence, for each  $\dot x \in  W_n$, we assign a lifetime $\ell_{\dot x}$ as a uniform random variable on $[0,1]$.
Since the bars have length atmost
1 in this model and the uniform distribtion is atom-free, the bound on the lifetimes in the score function in (\ref{eq:score_function_inversion}) is automatically satisfied almost surely.
Moreover as all bars are independent, then 
$
f(Z,V,\PP_n^{x})= f(Z,V,\PP_n)
$
for any $Z,V \in \PP_n$ and $x \in  W_n \times [0,1]$, which
in particular also ensures exponential stabilization. Hence, the only significant piece of work is showing that the variance is sufficiently large.

To that end, we will once more use the bound in Lemma \ref{lem:var_bound}. For $s \in (0,1)$, let $Q_{n,j,r}(s)$ denote the box centered at the same location of $Q_{n,j,r}$ but with side length $2 s r$,
and let $S_{n,j,r}(s)$ denote the part $Q_{n,j,r}$ whose first coordinate lies in the bottom $s$-quantile of first coordinates for points in $Q_{n,j,r}$. That is,
\begin{align*}
Q_{n,j,r}(s) & = \big\{ x \in Q_{n,j,r} \colon Q(x,r(1-s)) \subseteq Q_{n,j,r} \big\}, \\
S_{n,j,r}(s) & = \Big\{ x \in Q_{n,j,r} \colon x_1 \in \Big[\Big(\Big\lceil \frac{j}{n^{i-1}} \Big\rceil_n - 1 \Big)r, \Big(\Big\lceil \frac{j}{n^{i-1}} \Big\rceil_n - 1 + s\Big) r \Big] \Big\}.
\end{align*}
We now choose $r=1/2$, so that $Q_{n,j,1/2}$ has width 1 in the time-direction. Next, define
\begin{equation}
    \label{eq:uniform_configurations}
        \begin{aligned}
        \mc B_{n,j}^{(1)}  = & \{ D \in \mathfrak{N}(Q_{n,j,1/2}) :   \vert D \vert = 0 \},\\[0.5em]
        \mc B_{n,j}^{(2)}  = & \{ D \in \mathfrak{N}(Q_{n,j,1/2}) :   \emptyset \not \subseteq D \subseteq Q_{n,j,1/2}(1/8) \times [0,1/4] \}, \\
         \mc B_{n,j}^{(3)}  = & \{ D \in \mathfrak{N}(Q_{n,j,1/2}) :   \emptyset \not \subseteq D \subseteq S_{n,j,1/2}(1/4) \times [7/8,1] \},
        \end{aligned}
        \end{equation}
i.e., $ \mc B_{n,j}^{(2)} $ are the configurations where there are only short bars in the middle of the cube, while 
$ \mc B_{n,j}^{(3)} $ where there are long bars close to the edge of the cube.
By the Poisson void probabilities, then assumption (V1) in Lemma \ref{lem:var_bound} is satisfied.
Next, we consider the index set $I_{n,j}^1$ as defined in \eqref{eq:index_set}, and note that for any $i \in I_{n,j}^1$,
any configuration in $\mc B_{n,i}^{(3)}$ will add at least one inversion with a short bar in any configuration of $Q_{n,j,r}$ in $\mc B_{n,j}^{(2)} $.
Thus, it follows that assumption (V2) in Lemma \ref{lem:var_bound} is satisfied. Hence by Lemma \ref{lem:var_bound}, $\V[\Sigma(\PP_n)] \ge C n^{3d - 2}$,
which by Theorem \ref{thm:sum_clt} implies that for any $\delta > 0$,
$$
d_W\Big(\f{\Sigma(\PP_n)- \E[\Sigma(\PP_n)]}{\sqrt{\V[\Sigma(\PP_n)]}}, \mc N(0,1)\Big)
\le n^{-d/2 + \delta},
$$
i.e., we have the desired Normal approximation of the inversion count.

\subsubsection{The Poisson tree model}

The second model we consider is where we dynamically construct a tree structure from the points in $\PP_n$ and assign lifetimes to the points based on the length
of these branches as done in \cite{FerrariLandimThorisson2004}. The construction of the tree structure, loosely speaking, consists of sweeping through
the point cloud in the time direction and connecting points to the nearest neighbor to the right (in time),
which simultaneously satisfies a certain threshold in the spatial component. When compared to the previous model, the lifetimes will now depend on other points in the window $ W_n$, yet only the other points, which means we do not need to introduce additionally randomness through a mark space.

We now make the above ideas concise; First, 
let $\M = \emptyset$ and note that $x$ notationally coincides with $\dot x$.
Then, for any $Z \in \PP_n$, we define the \emph{cylinder} $C_n(Z)$ as the set
$$
C_n(Z) = [0,n] \times \pi_{2:d}(B(Z,1)),
$$
where recall $\pi_{2:d}$ is the projection from $\R^d$ onto the $(d-1)$-dimensional subspace of the second coordinate to the $d$'th coordinate.
With this, we now define for $Z \in \PP_n$ the \emph{right ancestor} $Z^+$ as the (a.s.) unique element of the set
    \begin{align*}
    \{V \in C_n(Z) : V_1 \ge Z_1, \ V_1 \le \widetilde V_1 \text{ for every } \widetilde V \in \PP_n \text{ where } \widetilde V_1 \ge Z_1 \}
    \end{align*}
if the set is non-empty, and set $Z^{+}=Z$ otherwise.
Additionally, we define the \emph{left-most successor} $Z^-$ of $Z$ as the (a.s.) unique element of the set
    \begin{align*}
    \{V \in C_n(Z) : V^+ = Z, \ V_1 \le \widetilde V_1 \text{ for every } \widetilde V \in \PP_n \text{ where } \widetilde V^+ = Z \}.
    \end{align*}
if the set is non-empty, and set $Z^{-}=Z$ otherwise.
Now, we can construct the \emph{Poisson tree graph} as the (directed) graph with vertex set $\PP_n$ and edge set $\mc E_n$ 
    defined as
    $$
    \mc E_n = \{(Z, Z^+) : Z \in \PP_n\},
    $$
    where we also note that this graph contains no loops by construction. For vertex $Z$ to the left of $V$ (i.e., $Z_1 \le V_1$),
     we use the notation $ Z \to V$ to denote that there is a \emph{path} - that is a sequence of edges - from $Z$ to $V$.
     Note that the leaves - that is vertices of degree one - in this graph are exactly the points $Z \in \PP_n$ such that $Z^- = Z$, 
    i.e. the points with no succesors to the left, and we use the notation $\mc L(\PP_n)$ to denote the set of such leaves.
    Moreover, we denote the set of all vertices with degree three or larger $\mc M (\PP_n)$---which we refer to as \emph{merge points} since at least two distinct paths must 'merge' at such a vertex.
    The next step is to implement what is commonly known as
     \emph{the Elder rule} \cite{BoissonnatChazalYvinec2018_GeometricTopologicalInference} for the merging of certain paths.
     More precisely, we define among paths starting at leaves, the \emph{survivor} $Z^\to$ 
    at a merge point $Z \in \mc M(\PP_n)$ as the (a.s.) unique leaf in the set
    $$
    \{V \in \mc L(\PP_n) \colon V \to Z, 
    \ V_1 \le \widetilde V_1 \text{ for every } \widetilde V \in \mc L(\PP_n) 
    \text{ where } \widetilde V \to Z\}.
    $$
    In other words, the survivor is the leaf that is oldest, i.e., has the smallest time coordinate, 
    among the leaves whose paths meet at this merge point. Furthermore,
    we define the \emph{death point} $Z^\dagger$ of $Z \in \mc L(\PP_n)$
    as the (a.s.) unique merge point in the set
     $$
    \{V \colon Z \to V, \ V^{\to} \neq Z,
     \ V_1 \le \widetilde V_1 \text{ for every } \widetilde V \in \mc M(\PP_n) 
     \text{ with } Z \to \widetilde V,\ \widetilde V^{\to} \neq Z \}.
     $$
    In other words the death point of a leaf is the first merge point (in time) where the leaf is not the survivor.
    We refer to the path $Z \to Z^{\dagger}$ as the \emph{branch} of $Z$.
    Finally, we can now define for any leaf $Z \in \mc L(\PP_n)$ the \emph{lifetime of its branch} $\ell_{Z,\PP_n}$ as
    \begin{equation}
        \label{eq:lifetime_poisson_tree}
    \ell_{Z,\PP_n} = Z^\dagger_1 - Z_1,
    \end{equation}
    i.e., the difference in time coordinates between a leaf and the death point of this leaf, and set $\ell_{Z,\PP_n} = 0$ if $Z \in \mc L(\PP_n)^c$.

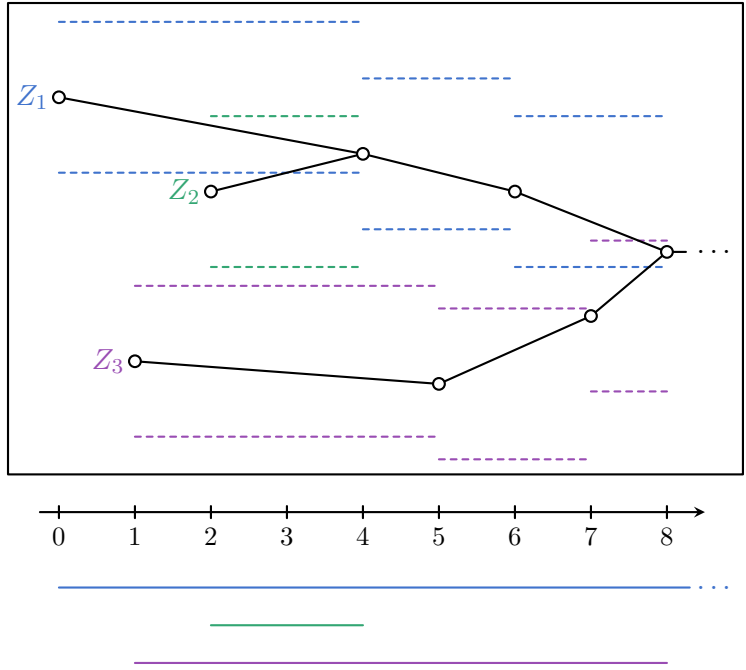
\begin{figure}[h!]
\centering
\begin{tikzpicture}[x=1cm,y=1cm,>=stealth,line cap=round,line join=round, scale=1]

\definecolor{ZoneBlue}{rgb}{0.25,0.45,0.80}
\definecolor{ZtwoGreen}{rgb}{0.20,0.65,0.45}
\definecolor{ZthreePurple}{rgb}{0.60,0.30,0.70}

\draw[black,thick] (0.33,0) rectangle (10,6.25);

\draw[black,thick,->] (0.75,-0.5) -- (9.5,-0.5);
\foreach \k in {0,...,8}{
  \draw[black,thick] ({1+\k},-0.58) -- ({1+\k},-0.42);
  \node[black,below] at ({1+\k},-0.58) {\small \k};
}

\draw[ZoneBlue,thick,dash pattern=on 2.2pt off 2.2pt] (1,6) -- (5,6);
\draw[ZoneBlue,thick,dash pattern=on 2.2pt off 2.2pt] (1,4) -- (5,4);

\draw[ZtwoGreen,thick,dash pattern=on 2.2pt off 2.2pt] (3,4.75) -- (5,4.75);
\draw[ZtwoGreen,thick,dash pattern=on 2.2pt off 2.2pt] (3,2.75) -- (5,2.75);

\draw[ZthreePurple,thick,dash pattern=on 2.2pt off 2.2pt] (2,2.5) -- (6,2.5);
\draw[ZthreePurple,thick,dash pattern=on 2.2pt off 2.2pt] (2,0.5) -- (6,0.5);

\draw[ZoneBlue,thick,dash pattern=on 2.2pt off 2.2pt] (5,5.25) -- (7,5.25);
\draw[ZoneBlue,thick,dash pattern=on 2.2pt off 2.2pt] (5,3.25) -- (7,3.25);

\draw[ZoneBlue,thick,dash pattern=on 2.2pt off 2.2pt] (7,4.75) -- (9,4.75);
\draw[ZoneBlue,thick,dash pattern=on 2.2pt off 2.2pt] (7,2.75) -- (9,2.75);

\draw[ZthreePurple,thick,dash pattern=on 2.2pt off 2.2pt] (6,2.2) -- (8,2.2);
\draw[ZthreePurple,thick,dash pattern=on 2.2pt off 2.2pt] (6,0.2) -- (8,0.2);

\draw[ZthreePurple,thick,dash pattern=on 2.2pt off 2.2pt] (8,3.1) -- (9,3.1);
\draw[ZthreePurple,thick,dash pattern=on 2.2pt off 2.2pt] (8,1.1) -- (9,1.1);

\draw[ZoneBlue,thick] (1,-1.5) -- (9.3,-1.5);
\node[right, ZoneBlue] at (9.0,-1.5) {$\ \ \ldots$};
\draw[ZtwoGreen,thick] (3,-2) -- (5,-2);
\draw[ZthreePurple,thick] (2,-2.5) -- (9,-2.5);

\coordinate (A) at (1.0,5);   
\coordinate (E) at (3,3.75);   
\coordinate (F) at (2,1.5);   

\coordinate (B) at (5,4.25);   
\coordinate (C) at (7, 3.75);   
\coordinate (D) at (9,2.95);   

\coordinate (G) at (6 , 1.2);   
\coordinate (H) at (8, 2.1);

\draw[black,thick] (A) -- (B) -- (C) -- (D);
\draw[black,thick] (E) -- (B);
\draw[black,thick] (F) -- (G) -- (H) -- (D);
\draw[black,thick] (D) -- (9.25,2.95);

\foreach \p in {A,E,F,B,C,D,G,H}{
  \filldraw[black,fill=white,thick] (\p) circle (2.2pt);
}

\node[ZoneBlue,left]   at (A) {$Z_1$};
\node[ZtwoGreen,left]  at (E) {$Z_2$};
\node[ZthreePurple,left] at (F) {$Z_3$};
\node[right] at (D) {$\ \ \ldots$};

\end{tikzpicture}
\label{fig:poisson_tree_model}
\caption{Illustration of the Poisson tree model. The points $Z_1$, $Z_2$, and $Z_3$ are leaves that appear at times $0$, $2$, and $1$, respectively. 
The dashed bars indicate a spatial constraint of radius $1$.
 When the first point within this radius appears, an edge is formed to that point, and the dashed bars shift to be centered at the new point. 
 This process continues until the branch merges with another branch. Below the horizontal arrow, the barcodes represent the lifetimes of each of the three branches
 starting at $Z_1$, $Z_2$, and $Z_3$. The branches from $Z_1$ and $Z_2$ merge at time $4$, and by the \emph{Elder rule} the first-born leaf $Z_1$ survives;
  hence the branch from $Z_2$ has a lifetime of $2$. The branches from $Z_1$ and $Z_3$ merge at time $8$, where again $Z_1$ survives, giving the branch from $Z_3$ a lifetime of $7$. 
  The branch from $Z_1$ and its barcode continues beyond the illustration.}
\end{figure}

\subsubsection{Lifetimes from Poisson trees}
\label{sec:poisson_tree_lifetimes}

    We now proceed to verify exponential stabilization, i.e., Assumption \ref{ass:A}(ii), and bound the variance of the inversion count using Lemma \ref{lem:var_bound} for when the lifetimes are as in \eqref{eq:lifetime_poisson_tree}.

    First, to verify Assumption \ref{ass:A}(ii), i.e., exponential stabilization; Let $x \in  W_n$ and define for $s > 0$ the cubic annulus $A(x,s)$ as 
    $$
    A(x,s) = [x_1^- -1, x_1^+ + 1] \times \big(\pi_{2:d}(Q(x,s+1)) \setminus \pi_{2:d}(Q(x,s))\big),
    $$
    and define the random critical value $\rho_x > 0$ as
    $$
    \rho_x = \min \{s \in \N_0 : \PP_n( A(x,s) ) = 0\}.
    $$
    Note that by construction, the box
    \begin{equation}
        \label{eq:stab_box}
    D(x,\rho_x) = [x_1^- -1, x_1^+ + 1] \times \pi_{2:d}(Q(x,\rho_x))
    \end{equation}
    has no Poisson points in an 1-thick band above and below itself, see Figure \ref{fig:stab_cases}.

    We now proceed to show that the score function $f$ stabilizes inside $D(x, \rho_x)$. Let $Z, V \in D(x, \rho_x)^c$ and consider the 
    following three regions,
    \begin{equation}
    \label{eq:stab_three_regions}
    \begin{aligned}
     W_n^{(1)} & = [0,x_1^- - 1] \times \pi_{2:d}( W_n), \\
     W_n^{(2)} & = \big([x_1^- - 1, x_1^+ + 1] \times \pi_{2:d}( W_n)\big) \cap D(x,\rho_x)^c, \\
     W_n^{(3)} & = [x_1^+ + 1,n] \times \pi_{2:d}( W_n),
    \end{aligned}
    \end{equation}
    as depicted in Figure \ref{fig:stab_cases}. Note first if both $Z$ and $V$ lie in either $ W_n^{(1)}$ or $ W_n^{(3)}$, it follows that the branches containing $Z$ and $V$ remains unchanged
    and hence $f(Z,V,\PP_n^x) = f(Z,V,\PP_n)$, i.e., the score function is unchanged. Suppose that $Z$ lies in $ W_n^{(2)}$ and $\ell_{Z,\PP_n} = 0$. Then, the insertion of $x$ does not
    make $Z$ a leaf and hence $f(Z,V,\PP_n^x) = f(Z,V,\PP_n) = 0$ for any $V$. If $0<\ell_{Z,\PP_n} \le  1$, then $Z^\dagger$ is not affected by $x$.
    Since the branch belonging to $V$ remains unchanged if $V$ lies in $ W_n^{(1)}$  or $ W_n^{(3)}$, it suffices to look at the situation where $V$ is in $ W_n^{(1)}$. If $\ell_{V,\PP_n} \le 1$, then
    by the same arguments, $V^{\dagger}$ is not affected by $x$ and hence $f(Z,V,\PP_n^x) = f(Z,V,\PP_n)$. If $\ell_{V,\PP_n} > 1$, then $f(Z,V,\PP_n) = 0$. However, by definition of $D(x,\rho_x)$,
    the insertion of $x$ cannot decrease the lifetime of $V$, and hence $f(Z,V,\PP_n^x) = 0$. Finally, if $Z$ lies in   $ W_n^{(1)}$  and $\ell_{Z,\PP_n} > 1$,
    we use the symmetric nature of the above argument to conclude that $f(Z,V,\PP_n^x) = f(Z,V,\PP_n)$ as well. 
 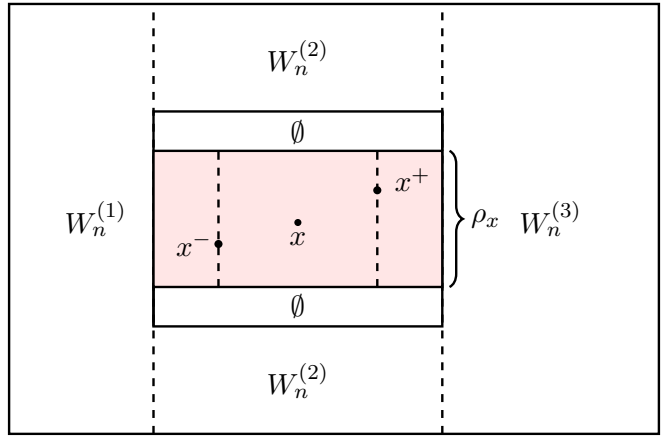
\begin{figure}[h!]
    \centering
    
    \begin{tikzpicture}[scale=0.95, every node/.style={font=\normalsize}]
  
  \draw[line width=0.9pt] (0,0) rectangle (9,6);

  \draw[dashed, line width=0.9pt] (2,0) -- (2,6);
  \draw[dashed, line width=0.9pt] (6,0) -- (6,6);

  \node at (1.2,3.0) {$W_n^{(1)}$};
  \node at (4.0,0.75) {$W_n^{(2)}$};
  \node at (4.0,5.25) {$W_n^{(2)}$};
  \node at (7.5,3.0) {$W_n^{(3)}$};

  \coordinate (A) at (2,1.5);
  \coordinate (B) at (6,4.5);

  \def\bandh{0.55}
  \coordinate (PinkBL) at ($(A)+(0,\bandh)$);
  \coordinate (PinkTR) at ($(B)+(0,-\bandh)$);

  \draw[line width=0.9pt] (A) rectangle (B);

  \draw[line width=0.9pt] ($(B)+(0,-\bandh)$) rectangle (B);
  \node at (4.0,4.22) {$\emptyset$};

  \draw[line width=0.9pt] (A) rectangle ($(A)+(0,\bandh)$);
  \node at (4.0,1.78)  {$\emptyset$};

  \path[fill=pink!40] (PinkBL) rectangle (PinkTR);
  \draw[line width=0.9pt] (PinkBL) rectangle (PinkTR);

  \coordinate (xL) at ($(PinkBL)+(0.9,0)$);
  \coordinate (xR) at ($(PinkTR)+(-0.9,0)$);
  \draw[dashed, line width=0.9pt] (xL |- PinkBL) -- (xL |- PinkTR);
  \draw[dashed, line width=0.9pt] (xR |- PinkBL) -- (xR |- PinkTR);

  \fill ($(xL)+(0.00,0.60)$) circle (1.6pt);
  \node[left] at ($(xL)+(0.05,0.60)$) {$x^-$};

  \fill ($0.5*(xL)+0.5*(xR) + (0,-0.05)$) circle (1.4pt);
  \node[below] at ($0.5*(xL)+0.5*(xR) + (0,-0.05)$) {$x$};

  \fill ($(xR)+(0.00,-0.55)$) circle (1.6pt);
  \node[right] at ($(xR)+(0.10,-0.4)$) {$x^+$};

  \draw[decorate, decoration={brace, amplitude=5pt}, line width=0.9pt]
    (6.1,3.95) -- (6.1,2.05)
    node[midway, right=4pt] {$\rho_x$};
\end{tikzpicture}
    \caption{The partition of the region $W_n$ into the three sub regions 
    $W_n^{(1)}$, $W_n^{(2)}$, and $W_n^{(3)}$ as in \eqref{eq:stab_three_regions},
    as well as the stabilization box $D(x,\rho_x)$ defined in \eqref{eq:stab_box}
     which is depicted in pink.
    The pads above and below are by construction of $D(x,\rho_x)$ void of Poisson points.}
    \label{fig:stab_cases}
\end{figure} 

 Next, we let
    $$
    \sigma_x= \lceil \max\{\rho_x, \vert x_1 - (x^-_1 - 1) \vert, \vert x_1 - (x^+ + 1) \vert \} \rceil,
    $$
    and note that since $D(x,\rho_x) \subseteq Q(x,\sigma_x)$,
    it also follows that $f$ stabilizes inside $Q(x,\sigma_x)$. Hence to conclude that $f$ stabilizes exponentially, it suffices to show that $\sigma_x$ has
    exponential tails. To that end, let $\eps > 0$, and note that
    \begin{equation}
        \label{eq:stab_max_sum}
    P(\sigma_x > n^\eps) \le P(\rho_x > n^\eps) + P(\vert x_1 - (x^-_1 - 1) \vert > n^\eps) + P(\vert x_1 - (x_1^+ + 1) \vert > n^\eps).
    \end{equation}
    First, it follows by the Poisson independence property that for all $n \gg 1$,
    \begin{equation}
        \label{eq:stab_1}
    P(\rho_x > n^\eps) \le P(\PP \cap ([-1,1]\times \pi_{2:d}(Q(x,1))) \neq \emptyset)^{\lfloor n^{\eps} \rfloor} \le (1-e^{-2^{d-1}})^{n^{\eps}}.
    \end{equation}
    Moreover, it follows that again by the void probability of the Poisson point process as well as the definition of the successor $x^-$ and ancestor $x^+$ that for all $n \gg 1$,
    \begin{equation}
        \label{eq:stab_2}
    P(\vert x_1 - (x_1^{\pm} + 1) \vert > n^\eps) \le e^{-n^\eps \vert \pi_{2:d}(Q(x,1)) \vert} \le (1-e^{-2^{d-1}})^{n^\eps}.
    \end{equation}

    Combining (\ref{eq:stab_max_sum}), (\ref{eq:stab_1}) and (\ref{eq:stab_2}) yields $f$ satisfies Assumption \ref{ass:A}(ii) as desired.

    Finally, we need to bound the variance of the inversion count $\Sigma(\PP_n)$, where we will rely on Lemma \ref{lem:var_bound} similarly as in Section \ref{sec:independent_uniform_lifetimes}.
    However, we now need to consider even more delicate configurations of points to ensure that we obtain inversions. 
Indeed, simply adding further Poisson points in a box could cause major changes in the tree structure, and therefore also modify the branch lengths of other points in the neighborhood. 
To avoid this issues, we will introduce a shield configuration which prevents that the addition of points influences other parts of the tree.
Let $r=4$, i.e., all boxes $Q_{n,j,r}$ are of width 8 (in the time direction) 
and all scaled boxes $Q_{n,j,4}(1/2)$ are of width 4. 
We now claim there with positive probability exists a set of configurations $\mc S_{n,j}$ such that
if the points in $Q_{n,j,4} \setminus Q_{n,j,4}(1/2)$ are in such a configuration, we have a \emph{shield}
in the sense that adding points inside $Q_{n,j,4}(1/2)$ does not change the tree structure outside of $Q_{n,j,4}$.
Let
$$
\mc B_{n,j}^{(0)} = \{ D \in \mathfrak{N}(Q_{n,j,4}) : D \cap(Q_{n,j,4} \setminus Q_{n,j,4}(1/2))\in \mc S_{n,j}\},
$$
denote all configuration of the cube $Q_{n,j,4}$, where a shield is present.
\bel[Existence of shield configurations]
\label{lem:shield}
For any $1 \le j \le \alpha_r n^d$, there is a set $\mc S_{n,j} \subseteq \mathfrak{N}(Q_{n,j,4} \setminus Q_{n,j,4}(1/2))$ such that 
if $\PP \cap(Q_{n,j,4} \setminus Q_{n,j,4}(1/2))\in \mc S_{n,j}$, then
$$
f(Z,V,\PP_n^x) = f(Z,V,\PP_n)
$$
for any $Z,V \in \PP \cap Q_{n,j,4}^c$ and $x \in Q_{n,j,4}(1/2)$. Moreover,
$$
\inf_{n \ge 1} P(\PP \cap Q_{n,j,4} \in \mc B_{n,j}^{(0)}) > 0.
$$
\enl
The idea is essentially to pad the boundary of $Q_{n,j,4}$
with points sufficiently dense such that no branches can enter the cube from the outside. However, the proof of Lemma \ref{lem:shield} containing the precise construction and verification
of the shield property can be found in Appendix \ref{sec:A}. Next, we further discretize the cube $Q_{n,j,4}(1/2)$ to ensure that we can create inversions inside the shield.
To that end, we introduce six shifted versions of $Q_{n,j,4}(1/32)$,
\begin{align*}
Q^\rightarrow_{n,j,4}(\tfrac 1 {32})  = & \{ x + (1/4,0,\ldots,0) \colon x \in Q_{n,j,4}(\tfrac 1 {32})\},\\
Q^{\twoheadrightarrow}_{n,j,4}(\tfrac 1 {32})  = & \{ x + (1/2,0,\ldots,0) \colon x \in Q_{n,j,4}(\tfrac 1 {32})\},\\
Q^\uparrow_{n,j,4}(\tfrac 1 {32})  = & \{ x + (0,\ldots,0,  3/4) \colon x \in Q_{n,j,4}(\tfrac 1 {32})\},\\
 Q^\downarrow_{n,j,4}(\tfrac 1 {32})  = & \{ x - (0,\ldots,0, 3/4) \colon x \in Q_{n,j,4}(\tfrac 1 {32})\},\\
 Q^\nwarrow_{n,j,4}(\tfrac 1 {32})  = & \{ x - ( 1/4,0,\ldots,0, -3/4) \colon x \in Q_{n,j,4}(\tfrac 1 {32})\},\\
 Q^\swarrow_{n,j,4}(\tfrac 1 {32})  = & \{ x - (1/4,0,\ldots,0,  3/4) \colon x \in Q_{n,j,4}(\tfrac 1 {32})\},
\end{align*}
which we note all are of width $1/4$.
Then, in lieu of Lemma \ref{lem:shield}, define the collections of shielded configurations $D \in \mc B_{n,j}^{(0)}$,
\begin{equation}
    \label{eq:lifetime_configurations}
        \begin{aligned}
        \mc B_{n,j}^{(1)}  = & \{   \vert D \cap Q_{n,j,4}(1/2) \vert = 0 \},\\[0.5em]
        \mc B_{n,j}^{(2)}  = & \big\{  \vert D \cap Q_{n,j,4}(1/2)  \vert = 3, \ \vert D \cap Q^\uparrow_{n,j,4}(\tfrac 1 {32}) \vert = \vert D \cap Q^\downarrow_{n,j,4}(\tfrac 1 {32}) \vert = \vert D \cap Q^\rightarrow_{n,j,4}(\tfrac 1 {32}) \vert = 1 \big\}, \\
         \mc B_{n,j}^{(3)}  = & \big\{  \vert D \cap Q_{n,j,4}(1/2)  \vert = 3, \ \vert D \cap Q^\nwarrow_{n,j,4}(\tfrac 1 {32}) \vert = \vert D \cap Q^\swarrow_{n,j,4}(\tfrac 1 {32}) \vert = \vert D \cap Q^{\twoheadrightarrow}_{n,j,4}(\tfrac 1 {32}) \vert = 1 \big\},
        \end{aligned}
        \end{equation}
i.e., $ \mc B_{n,j}^{(2)} $ are the shielded configurations where there are exactly one point above, below and to the right of $Q_{n,j,4}(1/32)$,
while $ \mc B_{n,j}^{(3)} $ are the shielded configurations, where there are exactly north-west, south-west and further right of $Q_{n,j,4}(1/32)$ in the time-directional sense.
By the Poisson void probabilities and Lemma \ref{lem:shield}, then assumption (V1) in Lemma \ref{lem:var_bound} is satisfied.
Next, we again consider the index set $I_{n,j}^1$ as defined in \eqref{eq:index_set}, and note that for any $i \in I_{n,j}^1$,
any configuration in $\mc B_{n,i}^{(3)}$ will add at least one inversion with the shorter bar in any configuration of $Q_{n,j,4}$ in $\mc B_{n,j}^{(2)} $ due to the shield property.
Thus, it follows that assumption (V2) in Lemma \ref{lem:var_bound} is satisfied. Hence by Lemma \ref{lem:var_bound},
$$
\V[\Sigma(\PP_n)] \ge C n^{3d - 2},
$$
for some $C >0$, which by Theorem \ref{thm:sum_clt} implies that for any $\delta > 0$,
$$
d_W\Big(\f{\Sigma(\PP_n)- \E[\Sigma(\PP_n)]}{\sqrt{\V[\Sigma(\PP_n)]}}, \mc N(0,1)\Big)
\le n^{-d/2 + \delta},
$$
i.e., we have the desired Normal approximation of the inversion count.

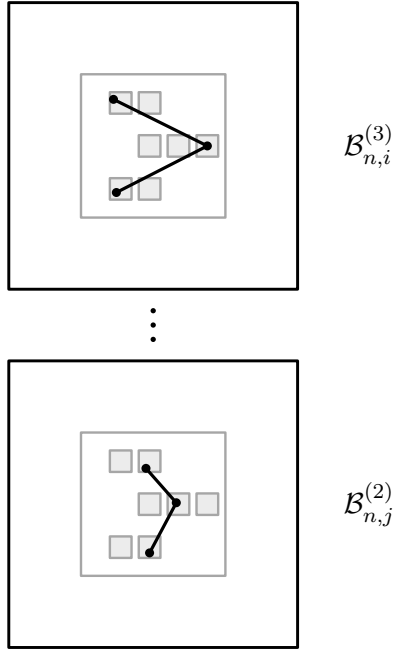
\begin{figure}[h!]
\centering
\begin{tikzpicture}[x=1cm,y=1cm,line cap=round,line join=round, scale=0.95]
\tikzset{
    outer/.style={draw=black,line width=1.1pt},
    inner/.style={draw=gray!70,line width=0.9pt},
    cell/.style={draw=gray!70,fill=gray!15,line width=0.9pt},
    wire/.style={draw=black,line width=1.2pt},
    dot/.style={circle,fill=black,inner sep=1.6pt}
}

\def\DrawCommonPanel{
    
    \draw[outer] (0,0) rectangle (4,4);
    
    \draw[inner] (1,1) rectangle (3,3);

    \foreach \x/\y in {
        1.4/2.45, 
        1.4/1.25, 
        1.8/2.45, 
        1.8/1.85, 
        1.8/1.25, 
        2.2/1.85,  
        2.6/1.85  
    }{
        \draw[cell] (\x,\y) rectangle ++(0.3,0.3);
    }
}

\begin{scope}[shift={(0,5)}]
    \DrawCommonPanel

    \coordinate (T1) at (1.45,2.65);
    \coordinate (T2) at (1.49,1.35);
    \coordinate (T3) at (2.75,2);

    \draw[wire] (T1)--(T3);
    \draw[wire] (T2)--(T3);

    \node[circle,fill=black,inner sep=1.2pt]at (T1) {};
    \node[circle,fill=black,inner sep=1.2pt] at (T2) {};
    \node[circle,fill=black,inner sep=1.2pt] at (T3) {};

    \node[right] at (4.5,2) {$\mc B_{n,i}^{(3)}$};
\end{scope}

\node[circle,fill=black,inner sep=0.7pt] at (2,4.7) {};
\node[circle,fill=black,inner sep=0.7pt] at (2,4.5) {};
\node[circle,fill=black,inner sep=0.7pt] at (2,4.3) {};

\begin{scope}[shift={(0,0)}]
    \DrawCommonPanel

    \coordinate (B1) at (1.9,2.50);
    \coordinate (B2) at (1.95,1.32);
    \coordinate (B3) at (2.32,2.02);

    \draw[wire] (B1)--(B3);
    \draw[wire] (B2)--(B3);

    \node[circle,fill=black,inner sep=1.2pt]at (B1) {};
    \node[circle,fill=black,inner sep=1.2pt] at (B2) {};
    \node[circle,fill=black,inner sep=1.2pt] at (B3) {};

    \node[right] at (4.5,2) {$\mc B_{n,j}^{(2)}$};
    \node[right] at (-0.5,2) {$\ $};
\end{scope}

\end{tikzpicture}

\caption{Illustration of the configurations in \eqref{eq:lifetime_configurations}, where the bottom is cube $j$ and displays a point configuration inside $\mc B_{n,j}^{(2)}$
and the top is a cube $i$ in $I_{n,j}^{1}$ that displays a point configuration inside $\mc B_{n,i}^{(3)}$. 
The last-born edge in the bottom cube creates an inversion with the last-born edge in the top cube due to the Elder rule in the Poisson tree model.}
\label{fig:no_shields}
\end{figure}

\subsection{Barcodes II: The tree realization number}
\label{sec:3.4}
In TDA, barcodes serve as concise summaries of the multiscale structure of data. 
Each barcode represents a family of intervals encoding the lifespans of topological features along a filtration,
 and is often viewed as a simplification of richer geometric or hierarchical information.
 Here, we focus on $0$-th dimensional information, i.e. connected components. As in \cite{adelie}, we focus on the map from a merge tree to a barcode which summarizes how the connected components intertwine. 
  Yet the process of mapping from an underlying merge tree to a barcode is inherently many-to-one.
   The \emph{tree realization number} quantifies this non-uniqueness: it measures how many combinatorially distinct merge trees can realize a given barcode.
    This number provides a discrete gauge of the information lost when passing from trees to barcodes and thus offers a combinatorial lens 
    on the geometry of persistent homology.

The combinatorial and probabilistic theory of the tree realization number has been developed in recent years \cite{adelie}. Their work builds an explicit connection between barcodes and permutations, showing that each barcode corresponds to a specific ordering of birth and death times, and that the realization number depends on the structure of that permutation. This correspondence allows classical tools from combinatorics and probability to be brought to bear on the analysis of barcodes. In particular, \cite{curry2024trees2} also study the tree realization number in  random barcode models.

We also note that tree realization numbers are of high interest for applications in neuroscience where the tree realization number can be used to study the shape of neuronal trees. Together, these works embed the tree realization number in a richer combinatorial and topological context, linking symmetric group structure, lattice theory, and inverse problems in TDA.

Having motivated the tree realization number,
 we now embed in the framework introduced in Section \ref{sec:main_results}. 
 More precisely, to deal with boundary effects, we also, for $0 < \alpha_1 < 1$, 
 introduce the shrunk version of \eqref{eq:W_n_simple} as
\begin{equation}
\widetilde{W}_n = [n^{\alpha_1},n - n^{\alpha_1}] \times [n^{\alpha_1},a_2 n-n^{\alpha_1}] \times \dots \times  [n^{\alpha_1},a_d n-n^{\alpha_1}].
\end{equation} 
Then, we consider the admissible condition, 
$$
\mc A_n(\PP_n) = \{Z \in \PP_n \colon Z \in \widetilde  W_n, \ \ell_{Z,\PP_n} \in (0,1)\}
$$
such that for any $Z,V \in \PP_n$, the score between $Z$ and $V$ is defined as
$$
\one\{Z \in \mc A_n^+(\PP_n) \}\one\big\{ \ell_{W,\PP_n} \in (0,1), \ \big(Z_1^{(d)} - W_1^{(d)}\big)\big(Z_1^{(d)} - W_1^{(d)} + \ell_{Z,\PP}  - \ell_{W,\PP}  \big) <0 \big\}.
$$
Similar to the previous section, we now consider the tree realization number when the lifetimes are independent and uniform as well as when they stem from Poisson trees.
\subsubsection{Lifetimes from independent and uniform variables}
\label{sec:3.4.1}
First, we consider the model in which the points exhibits complete spatial independence and
with lifetimes that are uniform and independent of each other as well of the spatial location of the corresponding point, i.e., the intensity measure $\lambda$ is the Lebesgue measure on $ W_n \times [0,1]$ 
and the lifetimes $\ell_{\dot x}$ are of the form
$$
\ell_{\dot x}\overset{\text{i.i.d.}}{\sim} \mathrm{Unif}([0, 1]),
$$
which is the same model as in Section \ref{sec:independent_uniform_lifetimes}. As a consequence,
we have already verified Assumptions \ref{ass:A}(i) to \ref{ass:A}(iii) in Assumption \ref{ass:A}, and hence to apply Theorem \ref{thm:prod_clt},
it suffices to verify Assumptions \ref{ass:B}(i) and \ref{ass:B}(ii) in Assumption \ref{ass:B}. Since the lifetimes are independent, it follows that
the stabilization Assumption \ref{ass:B}(i) is automatically satisfied in the same manner as Assumption \ref{ass:A}(ii) before. 
Next, we turn to the concentration property in \ref{ass:B}(ii);
Let $x \in  W_n$, $Z \in \mc A_n^+(\PP_n^x)$ and $\mathrm{x}\in \mc R$. If $\ell_{\dot Z} \ge 1/2$, let
$$
S_n^1(Z, 1/4)^+ = [Z_1,Z_1 + 1/4] \times \pi_{2:d}(S_n^k), 
$$
denote the column to the right of $Z$ of width $1/4$. Let  $ \{S_n^1(Z, 1/4)^+_j \colon 1 \le j \le n^{d-1} \}$ denote the lexiographic ordering of $n^{d-1}$ boxes of $S_n^1(Z, 1/4)^+$
of side lengths $1/4,a_2,\ldots, a_d$ 
in a similar fashion as in (\ref{eq:box_Q}). Note that any $V$ in $S_n^1(Z, 1/4)^+_j $ with lifetime shorter than 1/4 gives rise to an inversion, i.e. $f(Z,V,\PP_n)=1$.
Hence, it follows that
$$
P\Big(G(Z,\PP_n^\mathrm{x}) < \frac{\vert S_n^1 \vert}{4}  \Big)  
\le P\Big(\sum_{j=1}^{n^{d-1}} \one\{ \PP_n \cap (S_n^1(Z, 1/4)^+_j \times [0,1/4]) \neq \emptyset \} < \frac{\vert S_n^1 \vert}{4} \Big).
$$
From spatial independence and stationarity of the Poisson point process, it follows that
$$
\sum_{j=1}^{n^{d-1}} \one\{ \PP_n \cap (S_n^1(Z, 1/4)^+_j \times [0,1/4]) \neq \emptyset \}
$$
is a Binomial variable with mean $n^{d-1} q$, where $q = P(\PP_n \cap (S_n^1(Z, 1/4)^+_1 \times [0,1/4]) \neq \emptyset) \in (0,1)$. Thus, by Lemma \ref{lem:binom_conc},
$$
P\Big(G(Z,\PP_n^\mathrm{x}) < \frac{\vert S_n^1 \vert}{4} \Big)  
\le \exp\Big( - n^{d-1}q\big(\tfrac{1}{2}+\tfrac{1}{2}\log(\tfrac{1}{2})\big)\Big),
$$
Hence we may choose the exponent $\beta_5 = q\big(\tfrac{1}{2}+\tfrac{1}{2}\log(\tfrac{1}{2})\big)/2 > 0$ which yields the claim.
If $\ell_{\dot Z} < 1/2$, we instead let
$
S_n^1(Z, 1/4)^- = [Z_1-1/4,Z_1] \times \pi_{2:d}(S_n^k), 
$
which is inside $ W_n$ since $Z \in \widetilde  W_n$ for all $n \gg 1$.  Let  $ \{S_n^1(Z, 1/4)^-_j \colon 1 \le j \le n^{d-1} \}$ denote the lexicographic ordering of $n^{d-1}$ boxes of $S_n^1(Z, 1/4)^-$
and note this time that any $V$ in $S_n^1(Z, 1/4)^-_j $ with lifetime larger than 3/4 yields an inversion. Thus, using Lemma \ref{lem:binom_conc},
$$
P\Big(G(Z,\PP_n^\mathrm{x}) < \frac{\vert S_n^1 \vert}{4} \Big)  
\le \exp\Big( - n^{d-1}q\big(\tfrac{1}{2}+\tfrac{1}{2}\log(\tfrac{1}{2})\big)\Big),
$$
where $q = P(\PP_n \cap (S_n^1(Z, 1/4)^+_1 \times [3/4,1]) \neq \emptyset) \in (0,1)$, which completes the verification of Assumption \ref{ass:B}(ii) with exponent $\beta_5 = q\big(\tfrac{1}{2}+\tfrac{1}{2}\log(\tfrac{1}{2})\big)/2 > 0$ once more.

Finally, since we already have verified condition (V1) and (V2), it follows by Lemma \ref{lem:var_bound} that $\V[\Sigma^{\log}_n(\PP_n)] \ge C n^d$ for some $C > 0$.
Hence, by Theorem \ref{thm:log_clt}, for any $\delta > 0$,
\begin{equation}
d_W\Big(\f{\Sigma^{\log}_n(\PP_n)- \E[\Sigma^{\log}_n(\PP_n)]}{\sqrt{\V[\Sigma^{\log}_n(\PP_n)]}}, \mc N(0,1) \Big) \le n^{- d/2 + \delta}.
\end{equation}
By Corollary \ref{thm:prod_clt}, we also have asymptotic normality of the tree realization number itself.

\subsubsection{Lifetimes from Poisson trees}
\label{sec:3.4.2}
\label{sssec:lpt}
Next, we consider the Poisson tree model as in Section \ref{sec:poisson_tree_lifetimes}, where recall
for any leaf $Z \in \mc L(\PP_n)$ the lifetime $\ell_{Z,\PP_n}$ is defined as
$
    \ell_{Z,\PP_n} = Z^\dagger_1 - Z_1,
    $
where $Z^\dagger$ is the deathpoint and $\ell_{Z,\PP_n}=0$ otherwise. Again,
we have already verified Assumptions \ref{ass:A}(i) to \ref{ass:A}(iii) in Assumption \ref{ass:A}, and hence to apply Theorem \ref{thm:prod_clt},
it suffices to verify Assumptions \ref{ass:B}(i) and \ref{ass:B}(ii) in Assumption \ref{ass:B}.

First, we prove the additional stabilization condition in \ref{ass:B}(i) for $\mc A_n^+(\PP_n)$
as specified in the beginning of the section; Note that
the insertion of $x$ does not change if $Z$ lies in $\widetilde  W_n$ or not.
Consider the box $D(x,\rho_x)$ as defined in 
\eqref{eq:stab_box} and let $Z \notin D(x,\rho_x)$. If
$Z$ has lifetime 0, then the insertion of $x$ cannot alter the ancestor of $Z$
by construction of $D(x,\rho_x)$. If $Z$ has lifetime strictly between 0 and 1,
then the branch belonging to $Z$ is unaffected by the insertion of $x$
again by construction of the $\pm 1$ time buffer of $D(x,\rho_x)$.
Finally, if the lifetime of $Z$ is larger than 1, then Elder rule ensures that
the insertion of $x$ cannot make the lifetime less than 1. Hence, we take
$R_n(x)$ as the same radius as considered in Section \ref{sec:poisson_tree_lifetimes}, which we have already proven has exponential tails.

Thus, we turn to the concentration property in \ref{ass:B}(ii);

Let $x \in  W_n$, $Z \in \mc A_n^+(\PP_n^x)$ and $\mathrm{x}\in \mc R$. If $\ell_{\dot Z} \ge 1/2$, let
$
S_n^1(Z, 1/2)^+ 
$
denote the column to the right of $Z$ of width $1/4$. Let  $ \{S_n^1(Z, 1/4)^+_j \colon 1 \le j \le \alpha_r n^{d-1} \}$ denote the lexiographic ordering of $\alpha_r  n^{d-1}$ boxes of $S_n^1(Z, 1/4)^+$
in a similar fashion as in (\ref{eq:box_Q}). Note that any shielded configuration as described in $\mc B_{n,j}^{(3)}$ with $Q_{n,r}$ replaced by $S_n^1(Z, 1/4)^+_j $ in Section \ref{sec:poisson_tree_lifetimes} gives rise to an inversion with the short bar of $S_n^1(Z, 1/4)^+_j$.
Hence, it follows that
$$
P\Big(G(Z,\PP_n^\mathrm{x}) < \frac{\vert S_n^1 \vert}{4} \Big)  
\le P\Big(\sum_{j=1}^{n^{d-1}} \one\{ \PP \cap S_n^1(Z, 1/4)^+_j \in \mc B_{n,j}^{(3)} \} < \frac{\vert S_n^1 \vert}{4} \Big).
$$
From spatial independence and stationarity of Poisson point processes, and Lemma \ref{lem:binom_conc},
$$
P\Big(G(Z,\PP_n^\mathrm{x}) < \frac{\vert S_n^1 \vert}{4} \Big)  
\le \exp\Big( - n^{d-1}q\big(\tfrac{1}{2}+\tfrac{1}{2}\log(\tfrac{1}{2})\big)\Big),
$$
where $q = P(\PP \cap S_n^1(Z, 1/4)^+_j \in \mc B_{n,j}^{(3)}) \in (0,1)$. Hence we may choose the exponent $\beta_5 = q\big(\tfrac{1}{2}+\tfrac{1}{2}\log(\tfrac{1}{2})\big)/2 > 0$ which yields the claim.

Analagously if $\ell_{\dot Z} < 1/2$, we may define $S_n^1(Z, 1/4)^-$ to the left of $Z$ and use a similar construction and Lemma \ref{lem:binom_conc} to obtain the same bound.
Thus, Assumption \ref{ass:B}(ii) is satisfied with exponent $\beta_5 > 0$ once more.

Once more, since we already have verified condition (V1) and (V2), it follows by Lemma \ref{lem:var_bound} that $\V[\Sigma^{\log}_n(\PP_n)] \ge C n^d$ for some $C > 0$.
Hence, by Theorem \ref{thm:log_clt}, for any $\delta > 0$,
\begin{equation}
d_W\Big(\f{\Sigma^{\log}_n(\PP_n)- \E[\Sigma^{\log}_n(\PP_n)]}{\sqrt{\V[\Sigma^{\log}_n(\PP_n)]}}, \mc N(0,1) \Big) \le n^{- d/2 + \delta}.
\end{equation}
Additionally, by Corollary \ref{thm:prod_clt}, we also have asymptotic normality of the tree realization number itself when the lifetimes stem from Poisson trees.
\subsubsection{Lifetimes from other tree models}
\label{sec:3.4.3}

Now, we briefly discuss additional tree models
and if the Normal approximation in Theorems \ref{thm:sum_clt} and \ref{thm:log_clt} 
can be applied to the inversion count and tree realization number.

First, if one were to generalize the Poisson tree model to consider the case where
the cylinder $C_n(Z)$ has radius $\ell > 0$ instead of 1, then
by modifying the size of the cubes $Q_{n,j,r}$, we will still
have Assumptions \ref{ass:A}(i) to \ref{ass:A}(iii) in Assumption \ref{ass:A} as well as \ref{ass:B}(i) and \ref{ass:B}(ii) in Assumption \ref{ass:B} satisfied
with same type of arguments as before. Even going one step further and allowing 
the spatial constraint to be an ellipsoid, cube or any bounded convex body rather than a ball
would not alter the arguments in any significant way.

The next natural step would be to consider the case where we lift
the spatial constraint altogether. By doing this, and letting each
point connect to its nearest neighbor in Euclidean distance to
the right in the time direction, we obtain the directed spanning tree model as in \cite{MR2035772}.
This model is more challenging, and would require some modification to the
construction of the shielded configuration as well as the
box of stabilization. In particular for stabilization,
it will not be possible to create regions of no-points where connections
cannot be made across. Instead this regions would be replaced 
by shields that instead absorb any edge connections appearing.
We leave this extension as future work.

Finally, we can also consider the radial spanning tree model as in \cite{MR2292589},
where each point connects to its nearest neighbor in Euclidean distance that
is closer to the origin than itself. This model adds another layer of complexity,
since now we either keep the rectangular slabs $S_n^k$, but then consider a non-homogeneous Poisson point process
as input so we can go back and forth between polar and cartesian coordinates, or we need
to change the rectangular slabs to be annuli instead. 
Both approaches would require extensions of the general framework as discussed in Section \ref{sec:limitations}.

\section{Bounding error terms: Double-sum case}
\label{sec:4}
This section is dedicated to proving the error terms bounds which comes out of applying the Malliavin-Stein Normal approximation to the double-sum functional. To that end, we first record some important consequences of Assumption \ref{ass:A} in Section 5.1,
then proceed to use the consequences to bound moments of the difference operators in Section 5.2. Finally, we prove all three bounds found in Lemmas \ref{lem:en1} in Section 5.3.

\subsection{Consequences of Assumption \ref{ass:A}}
\label{sec:4.1}
We now record a series of preliminary results which will aid us in controlling the integrals $I_{n,1}$, $I_{n,2}$, and $I_{n,3}$ in \eqref{eq:error_terms_sum}.
The first lemma is
a simple, yet immensely useful moment bound on Poisson random variables, which follows directly from the Touchard polynomial representation for moments of Poisson random variables (see \cite{Pinsky2017_ConnectionsBetweenPermutationCyclesAndTouchardPolynomials}).
\begin{lemma}[Poisson moment bound]
    \label{lem:poisson_moment_bound}
    Let $m \ge 1$ and $X \sim \text{Poisson}(\ell)$. If $\ell \gg 1$, then
    $$
    \E[X^m]\le 2 \ell^m.
    $$
\end{lemma}
Recall that we write $\Q(x,R_n)$ instead of $\Q(x,R_n(x))$ for brevity for the non-stable cube. Next, we prove that Assumption \ref{ass:A}(ii), i.e., exponential stabilization, yields that all moments of the size of the non-stable cube are of constant order.
\begin{lemma}[Moments of $\mathbb{Q}(x,R_n)$]
    \label{lem:moments_of_Q}
    Let $m \ge 1$. Under Assumption \ref{ass:A}, there is a constant $C_0(m)>0$ depending on $m$ such that

       $$
   \sup_{n \gg 1} \sup_{x \in  \W_n} \E\big[\PP(\mathbb{Q}(x,R_n))^{m}\big] \le C_0(m).
   $$
    \end{lemma}
\begin{proof}
By the union bound,
$$
\E\big[\PP(\mathbb{Q}(x,R_n))^{m}\big] \le \sum_{j=1}^{d_n} \E\Big[\PP(\mathbb{Q}(x,j))^{m} \one\{R_n(x) = j\} \Big],
$$
where $ d_n \in \N $ is larger than all side lengths of $ W_n$. By the Cauchy--Schwarz inequality, 
$$
\E\big[\PP(\mathbb{Q}(x,R_n))^{m}\big] \le \sum_{j=1}^{d_n} \E\big[\PP(\mathbb{Q}(x,j))^{2m}\big]^{1/2} P(R_n(x) = j)^{1/2},
$$
and by Assumption \ref{ass:A}(ii) and Lemma \ref{lem:poisson_moment_bound},
$
\E\big[\PP(\mathbb{Q}(x,R_n))^{m}\big] \le \sqrt{2}  \sum_{j=1}^{d_n} j^{d m} \e^{-\beta_1(\eps) j /2},
$
for some $\eps > 0$. Finally, since there is a natural number $j_0 = j_0(m,\eps)$ depending on $m$ and $\eps$ such that $j_0^{d m} \e^{-\beta_1(\eps) j_0 /2} \le 1$
and $j^{d m} \e^{-\beta_1(\eps) j /2} \to 0$ monotonically when $j \ge j_0$, then for all $n \gg 1$, it follows by the integral test that
\begin{equation}
    \label{eq:integral_test}
    \begin{aligned}
\sum_{j=1}^{d_n} j^{d m} \e^{-\beta_1(\eps) j /2} & \le \sum_{j=1}^{j_0 - 1} j^{d m} \e^{-\beta_1(\eps) j /2} + 1 + \int_{j_0}^\infty j^{dm}\e^{-\beta_1(\eps) j /2} \ \text dj \\
& \le \sum_{j=1}^{j_0 - 1} j^{d m} \e^{-\beta_1(\eps) j /2} + 1+\frac{2^{m+1}m!}{\beta_1(\eps)^{m+1}}.
\end{aligned}
\end{equation}
Hence, as the last bound in \eqref{eq:integral_test} does not depend on $n$ and choosing $\eps = 1/2$, this completes the proof.
\end{proof}
For $v, v' \in \mc R$, introduce the \textit{compound score} in $Z$ as
$
F(Z,\PP_n^\mathrm{x}, \PP_n^{\mathrm{x}'}) = \sum_{W \in\PP_n^\mathrm{x}}f(Z,V, \PP_n^{\mathrm{x}'}),
$
i.e., the total scores between the point $Z$ and all other points in the Poisson process $\PP_n$ as well as the added points in $v$.

Then, we can record an immediate decomposition of the first and second-order difference operators $D_x$ and $D_{xy}^2$ in terms of the compound scores.
\begin{lemma}[Decomposition]
    \label{lem:decomposition_simple}
    For every $x,y \in  \W_n$, it holds that
\begin{equation}
    \label{eq:decomposition_simple_1}
    D_x\Sigma(\PP_n) =  F(x,\PP_n,\PP_n^x) + \sum_{Z \in \PP_n}\big( F(Z,\PP_n,\PP_n^x)- F(Z,\PP_n,\PP_n)\big),
\end{equation}
\begin{equation}
    \label{eq:decomposition_simple_2}
    \begin{aligned}
D_{xy}^2\Sigma(\PP_n)  = & f(x,y,\PP_n^{xy})+  F(x, \PP_n, \PP_n^{xy}) - F(x,\PP_n, \PP_n^x) + F(y,\PP_n, \PP_n^{xy}) - F(y,\PP_n, \PP_n^y)\\
&   + \sum_{Z \in \PP_n} \big( F(Z,\PP_n,\PP_n^{xy}) - F(Z,\PP_n,\PP_n^x) - F(Z,\PP_n,\PP_n^y) + F(Z,\PP_n,\PP_n) \big).
    \end{aligned}
\end{equation}
\end{lemma}
\begin{proof}
Follows directly from using that 
$$
 D_x\Sigma(\PP_n) = \Sigma(\PP_n^x) - \Sigma(\PP_n) \quad \text{ and } \quad D_{xy}^2\Sigma(\PP_n) = \Sigma(\PP_n^{xy}) - \Sigma(\PP_n^x) - \Sigma(\PP_n^y) + \Sigma(\PP_n),
$$
and subsequently inserting the definition of the double-sum functional $\Sigma(\cdot)$ in \eqref{eq:double_sum_formal}.
\end{proof}
We now use Assumption \ref{ass:A} to bound compound score moments.
\begin{lemma}[Compound score $F$ bounds]
    \label{lem:compound_score_bounds}
Let $m \ge 1$ and $\eps > 0$. Under Assumption \ref{ass:A}, it holds for all $n \gg 1$,
    \begin{equation}
        \label{eq:compound_score_bound_1}
    \sup_{x \in  \W_n} \sup_{\mathrm{x}\in \mc R}\E\big[F(x,\PP_n, \PP_n^{v})^m\big] \le n^{m \eps} \vert S_n^k \vert^m,
    \end{equation}
     \begin{equation}
        \label{eq:compound_score_bound_3}
  \sup_{\mathrm{x}\in \mc R}\E\Big[ \Big(\sum_{Z \in \PP_n}F(Z,\PP_n, \PP_n^{v})\Big)^m \Big] \le  n^{m \eps} \vert  W_n \vert^{2m},
    \end{equation}
    \begin{equation}
        \label{eq:compound_score_bound_2}
    \sup_{x \in  \W_n} \E\Big[ \Big( \sum_{Z \in \PP_n}
    \big\vert F(Z,\PP_n^x,\PP_n^x) - F(Z,\PP_n, \PP_n) \big\vert\Big)^{m}\Big] \le n^{m \eps} \vert Q(0,n^\eps) \vert^m  \vert S_n^k(0,n^\eps) \vert^m.
    \end{equation}
\end{lemma}
\begin{proof}
First, we prove the bound in \eqref{eq:compound_score_bound_1}.  By Assumption \ref{ass:A}(i), i.e., $k$-locality, it follows that $f(x, V, \PP_n^\mathrm{x}) = 0$ for
    any $V \in \PP_n \cap \mathbb{S}_n^k(x,1)^c$. Thus,
    $$
    \E\big[F(x,\PP_n,\PP_n^\mathrm{x})^m\big] = \E\Big[\Big(\sum_{V \in \PP_n \cap \mathbb{S}_n^k(x,1)} f(x, V, \PP_n^\mathrm{x})\Big)^m\Big].
    $$
By the Cauchy--Schwarz inequality and stationarity of $\PP_n$, it follows that
$$
    \E\big[F(x,\PP_n,\PP_n^\mathrm{x})^m\big]  \le \E\big[\overline f_{\sup}(\PP_n)^{2m}\big]^{1/2} \E\Big[\Big(\sum_{V \in \PP_n\cap \mathbb{S}_n^k(0,1) } 1 \Big)^{2m}\Big]^{1/2}
$$
Applying Lemma \ref{lem:poisson_moment_bound} and Assumption \ref{ass:A}(iii), i.e., sub-polynomial moments, yields the claim. 

Next, we consider the bound in \eqref{eq:compound_score_bound_3}. Using the Cauchy--Schwarz inequality, then
$$
\E\Big[ \Big(\sum_{Z \in \PP_n}F(Z,\PP_n, \PP_n^{v})\Big)^m \Big] \le \E\big[\overline f_{\sup}(\PP_n)^{2m}\big]^{1/2} \E\Big[\Big(\sum_{Z \in \PP_n} \sum_{V \in \PP_n \setminus \{Z\}} 1 \Big)^{2m}\Big]^{1/2}.
$$
Invoking Assumption \ref{ass:A}(iii) and \ref{lem:poisson_moment_bound} yields the claim.

Finally, we turn to the bound in \eqref{eq:compound_score_bound_2}. For convenience, let
$$
D_n = \E\Big[ \Big( \sum_{Z \in \PP_n}
    \big\vert F(Z,\PP_n^x,\PP_n^x) - F(Z, \PP_n,\PP_n) \big\vert\Big)^{m}\Big].
$$
If $f(Z,V, \PP_n^x) \neq f(Z,V,\PP_n)$, then
    by definition of the stabilization radius $R_n = R_n(x)$, it holds that $Z \in \mathbb{Q}(x,R_n)$ or $V \in \mathbb{Q}(x,R_n)$. 
    Assuming without loss of generality that $Z \in \mathbb{Q}(x,R_n)$,
    then if $1+R_n \le n^\eps$ it follows that $V \in \mathbb{S}_n^k(x,1+R_n) \subseteq \mathbb{S}_n^k(x,n^\eps)$.
    Thus,
    \begin{align*}
D_n &\le 
       \E\Big[\Big( \sum_{(Z,V) \in 
        \PP_n^{(2)}}
        \vert f(Z,V, \PP_n^x) - f(Z,V,\PP_n) \vert \Big)^m 
        \one\{1+R_n > n^\eps\}\Big]\\
        &  +
      \E\Big[\Big( \sum_{(Z,V) \in 
        \PP_n^{(2)}\cap (\Q(x,n^\eps) \times \S_n^k(x,n^\eps))}
        \vert f(Z,V, \PP_n^x) - f(Z,V,\PP_n) \vert \Big)^m 
       \Big].
     \end{align*}
By the Cauchy--Schwarz inequality (twice) and stationarity of $\PP_n$, 
    \begin{align*}
D_n &\le 
      \E\big[\overline f_{\sup}(\PP_n)^{2m}\big]^{1/2}   \E\Big[\Big( \sum_{(Z,V) \in 
        \PP_n^{(2)}}
       1 \Big)^{4m}\Big]^{1/4} 
        P(1+R_n > n^\eps)^{1/4}\\
        &   +
       \E\big[\overline f_{\sup}(\PP_n)^{2m}\big]^{1/2}  \E\Big[\Big( \sum_{(Z,V) \in 
        \PP_n^{(2)}\cap (\Q(0,n^\eps) \times \S_n^k(0,n^\eps))}
        1 \Big)^{2m} 
       \Big]^{1/2}.
     \end{align*}
By applying Assumption \ref{ass:A}(iii) and Lemma \ref{lem:poisson_moment_bound},
  $$
  D_n \le n^{m \eps} \vert  W_n \vert^{2m}  P(1+R_n > n^\eps)^{1/4} +  n^{m \eps} \vert Q(0,n^\eps)\vert^m \vert S_n^k(0,n^\eps) \vert^m.
  $$
Invoking Assumption \ref{ass:A}(ii), i.e., exponential stabilization, and Lemma \ref{lem:normalization}, i.e., that we may omit the constants, completes the proof.
\end{proof}

\subsection{Moment bounds on $D_x \Sigma$ and proof of Lemma \ref{lem:en1}(iii)}
\label{sec:4.2}
In this section, we prove the bound on the third error terms, which only involves the first-order difference operator,
and hence we want control over the moments of $D_x\Sigma(\PP_n)$.

To that end, the following basic inequality---which is a direct consequence of the definition of convexity---will be useful in bounding moments of sums, when
the size of the constants are of less importance.
\begin{lemma}["Freshman's reality"]
    \label{lem:reality}
   For any  $m,M \in \N$ and $a_1, \ldots,a_M \ge 0$,
   $$
   \Big(\sum_{i=1}^M a_i \Big)^m \le M^{m-1} \sum_{i=1}^M a_i^m.
   $$
\end{lemma}
We are now ready to establish that the third and fourth moment of the first-order difference operator can essentially---up to a small error of order $n^{\eps}$---be bounded
by the volume of the slab $S_n^k$ to the power of 3 and 4, respectively.
\begin{lemma}[$3^{\text{rd}}$/$4^{\text{th}}$ moment of $D_{x}$]
    \label{lem:fourth_moment}
   Let $\eps > 0$ and $n \gg 1$. Under Assumption \ref{ass:A},
    \begin{align}
    \sup_{x \in  \W_n}\E[\vert D_{x}\Sigma(\PP_n)\vert ^3] \le & n^{\eps} \vert S_n^k \vert^3,\\
    \sup_{x \in  \W_n} \E[\vert D_{x}\Sigma(\PP_n)\vert ^4] \le & n^{\eps} \vert S_n^k \vert^4.
    \end{align}
    \end{lemma}
\begin{proof}
    Let $m \in \{3,4\}$ and $\eps' = \eps/(m(d+k+1))$. First, by Lemmas \ref{lem:decomposition_simple}
    and \ref{lem:reality},
    \begin{align*}
    \E[\vert D_{x}\Sigma(\PP_n) \vert^m] \le & 
    2^{m-1} \E\big[F(x,\PP_n^x)^m\big] + 2^{m-1} \E\Big[ \Big(\sum_{Z \in \PP_n}
    \vert F(Z,\PP_n^x) - F(Z, \PP_n) \vert \Big)^m \Big].
    \end{align*}
Combining Lemmas \ref{lem:normalization} and \ref{lem:compound_score_bounds} (using $\eps'$) with the observation that $\vert Q(x,n^{\eps'})\vert = 2^d n^{d\eps'}$
and $\vert S_n^k(0,n^{\eps'}) \vert = 2^k n^{k\eps'} \vert S_n^k \vert$ completes the proof. 
\end{proof}

We can now prove the error bound on the third error term $I_{n,3}$. 

\begin{proof}[Proof of Lemma \ref{lem:en1}(iii)]
    Recall we want to prove that for $\eps > 0$ and $n \gg 1$,
    $$
    I_{n,3} = \int_{ \W_n}\E\big[|D_x\Sigma(\PP_n)|^3\big] \ \text dx  \le n^{\eps} \vert  W_n \vert  \vert S_n^k \vert^3.
    $$
    First by Lemma \ref{lem:fourth_moment},
    $
    I_{n,3} \le \int_{ \W_n}n^{\eps} \vert S_n^k \vert^3  \ \text dx.
    $
    As the integrand is constant in $x$ and as $\lambda(\W_n)= \vert  W_n \vert $,
    this completes the proof.
\end{proof}

\subsection{Moment bounds on $D_{xy}^2\Sigma$ and proof of Lemma \ref{lem:en1}(i)-(ii)}
\label{sec:4.3}
Next, in this section, we prove the bounds on the first two error terms, and hence we need to additionally control the fourth
moment of the second-order difference operator $D_{xy}\Sigma(\PP_n)$. Let $x,y \in  \W_n$ and $\eps > 0$. We
recall the three cases of the spatial position of $y$ in relation to $x$ as outlined in Section \ref{sec:proof_strategy};
\begin{align*}
\textbf{Case I:} \qquad & y \in  \W_n \setminus \S_n^k(x,n^\eps),\\
\textbf{Case II:} \qquad & y \in \S_n^k(x,n^\eps) \setminus \Q(x,n^\eps),\\
\textbf{Case III:} \qquad & y \in \Q(x,n^\eps).
\end{align*}
We now establish fourth moment bounds on $D_{xy}^2\Sigma(\PP_n)$ based on each of three cases above. Recall that we write $\Q(x,R_n)$ instead of $\Q(x,R_n(x))$ for brevity for the non-stable cube,
and that $\Q(x,R_n)$ and $\Q(y,R_n)$ in general are of different sizes.

\begin{lemma}[$4^{\text{th}}$ moment of $D_{xy}^2$; Case I]
    \label{lem:fourth_moment_first_case}
    Let $\eps > 0$ and $n \gg 1$. Under Assumption \ref{ass:A},
   $$
        \sup_{x \in  \W_n} \sup_{y \in \W_n \setminus \S_n^k(x,n^\eps) }\E[(D_{xy}^2\Sigma(\PP_n))^4] \le \e^{-\beta_1(\eps) n^{\eps} / 16}.
    $$
    \end{lemma}
    \begin{proof}
    First, by Lemmas \ref{lem:decomposition_simple}
    and \ref{lem:reality}, we have that
    \begin{equation}
        \label{eq:fourth_moment_expansion}
        \begin{aligned}
\E[(D_{xy}^2  & \Sigma(\PP_n))^4]  = 4^4 \E[f(x,y,\PP_n^{xy})^4]+ 4^4  \E\Big[\big(F(x, \PP_n, \PP_n^{xy}) - F(x,\PP_n, \PP_n^x)\big)^4 \Big]\\
&  + 4^4  \E\Big[\big( F(y,\PP_n, \PP_n^{xy}) - F(y,\PP_n, \PP_n^y)\big)^4 \Big]\\
&   +  4^4  \E\Big[\Big(\sum_{Z \in \PP_n} \big( F(Z,\PP_n,\PP_n^{xy}) - F(Z,\PP_n,\PP_n^x) - F(Z,\PP_n,\PP_n^y) + F(Z,\PP_n,\PP_n) \big) \Big)^4 \Big].
    \end{aligned}
\end{equation}

    Let $\gamma > 0$. By Assumption \ref{ass:A}(iii), i.e., sub-polynomial moments, as well as Lemmas \ref{lem:reality} and \ref{lem:compound_score_bounds} (specifically the bounds (\ref{eq:compound_score_bound_1}) and (\ref{eq:compound_score_bound_3})), and finally the Cauchy--Schwarz inequality,
    \begin{equation}
        \label{eq:crude_bound}
        \E\Big[(D_{xy}^2\Sigma(\PP_n))^4 \one\{R_n(x)+ R_n(y) > \gamma \}\Big] \le 1024 n^{4\eps} \vert  W_n\vert^8 P(R_n(x)+ R_n(y) > \gamma)^{1/2}.
    \end{equation}
    By Assumption \ref{ass:A}(ii), i.e., exponential stabilization, there is a $\beta_1(\eps) > 0$ such that
    \begin{equation}
        \label{eq:general_fourth_bound}
        \begin{aligned}
       \E\Big[(D_{xy}^2\Sigma(\PP_n))^4 \one\{R_n(x)+ R_n(y) > \gamma \}\Big] &  \le  2048 n^{4\eps}\vert  W_n\vert^8 \e^{-\beta_1(\eps)\gamma / 4}.
    \end{aligned}
\end{equation}
   Now, put $\gamma = |\pi_1(x - y)|-1$. 
    Then, on the event that $ R_n(x) + R_n(y) < |\pi_1(x - y)| - 1$, it follows that $\Q(x, R_n) \cap \Q(y, R_n)= \es$
    and by $k$-locality that $f(x,y,\PP_n^{xy})=0$.
    Now, suppose that $f(x,V,\PP_n^{xy}) \neq f(x, V, \PP_n^x)$ for some $V \in \PP_n$. Since $x \in \Q(y, R_n)^c$, it follows by the definition
    of the stabilization radius that $V \in \Q(y, R_n)$. Moreover, it follows by the assumption of $k$-locality that $|\pi_1(x - V)| \le 1$.
    Thus,
    \begin{equation}
        \label{eq:diff_0}
   F(x, \PP_n, \PP_n^{xy}) - F(x,\PP_n, \PP_n^x) = F(y, \PP_n, \PP_n^{xy}) - F(y,\PP_n, \PP_n^y) = 0.
    \end{equation}
    Next suppose that
    $f(Z,V,\PP_n^{xy}) - f(Z,V, \PP_n^x) \neq f(Z,V,\PP_n^{y}) - f(Z,V,\PP_n)$ for some points $Z,V \in \PP_n$ and assume without loss of generality that
    $f(Z,V,\PP_n^{y}) - f(Z,V,\PP_n) \neq 0$. Hence as before, we conclude that either $Z$ or $V$ lies in $\Q(y, R_n)$. 
    Moreover, it must hold that $f(Z,V,\PP_n^{xy}) \neq f(Z,V,\PP_n^{y})$ or $f(Z,V,\PP_n^{x}) \neq f(Z,V,\PP_n)$,
    which implies either $Z$ or $V$ lies in $\Q(x, R_n)$. Since $\Q(x, R_n) \cap \Q(y, R_n)= \es$, then by symmetry, we may conclude that 
    $(Z,V) \in \Q(x, R_n) \times \Q(y, R_n)$.
    Finally, the assumption of $k$-locality again ensures that $|\pi_1(Z - V)| \le 1$. Thus,

   \begin{equation}
    \label{eq:diff_0_more}
    \sum_{Z \in \PP_n} \big( F(Z,\PP_n,\PP_n^{xy}) - F(Z,\PP_n,\PP_n^x) - F(Z,\PP_n,\PP_n^y) + F(Z,\PP_n,\PP_n) \big) = 0
    \end{equation}
    since $|\pi_1(x - y)| > 1 + R_n(x) + R_n(y)$ implies that for any $(Z,V) \in \Q(x, R_n) \times \Q(y, R_n)$,
    then $\pi_1(Z-W) > 1$. Hence, combining equations \eqref{eq:fourth_moment_expansion}, \eqref{eq:general_fourth_bound},\eqref{eq:diff_0} and \eqref{eq:diff_0_more},
    $$
    \E[(D_{xy}^2\Sigma(\PP_n))^4] \le 2048 n^\eps \vert  W_n\vert^8 \e^{-\beta_1(\eps)( \vert \pi_1(x-y) \vert - 1) / 4}.
    $$
    Using that $y \in \W_n \setminus \S_n^k(x,n^\eps)$, that is $\vert \pi_1(x-y) \vert  > n^\eps $, as well as Lemma \ref{lem:normalization} completes the proof.
\end{proof}

\begin{lemma}[$4^{\text{th}}$ moment of $D_{xy}^2$; Case II]
    \label{lem:fourth_moment_second_case_two}
    Let $\eps > 0$ and $n \gg 1$. Under Assumption \ref{ass:A},
   $$
        \sup_{x \in  \W_n} \sup_{y \in \S_n^k(x,n^\eps) \setminus \Q(x,n^\eps)}\E[(D_{xy}^2\Sigma(\PP_n))^4] \le n^{4\eps}.
    $$
\end{lemma}   
\begin{proof} 
  The idea is to follow the same strategy as in the proof of Lemma \ref{lem:fourth_moment_first_case}, where this time we choose $\gamma = n^\eps$. Note that on the event that 
    $R_n(x)+R_n(y)\le n^\eps$, it follows that $\Q(x,R_n)\cap \Q(y,R_n) = \emptyset$, and hence
    $$
    \vert F(x, \PP_n, \PP_n^{xy}) - F(x,\PP_n, \PP_n^x) \vert \le \overline f_{\sup}(\PP_n)  \PP(\Q(y,R_n)),
    $$
    $$
    \vert F(y, \PP_n, \PP_n^{xy}) - F(y,\PP_n, \PP_n^y) \vert \le \overline f_{\sup}(\PP_n) \PP(\Q(x,R_n)),
    $$
    as well as
    \begin{align*}
    \sum_{Z \in \PP_n}& \big\vert  F(Z,\PP_n,\PP_n^{xy}) - F(Z,\PP_n,\PP_n^x) - F(Z,\PP_n,\PP_n^y) + F(Z,\PP_n,\PP_n) \big \vert \\
    & \qquad \qquad \qquad \qquad \qquad \qquad \qquad  \le 2 \overline f_{\sup}(\PP_n) \PP(\Q(x,R_n))\PP(\Q(y,R_n)).
    \end{align*}
    Now, since
\begin{align*}
 \PP(\Q(y,R_n))^4 + \PP(\Q(x,R_n))^4
    & + 2 \PP(\Q(x,R_n))^4\PP(\Q(y,R_n))^4 \\
    & \qquad \le 4 \big(\PP(\Q(x,R_n))\vee 1\big)^4 \times
     \big( \PP(\Q(y,R_n))\vee 1\big)^4,
\end{align*}
then using (\ref{eq:fourth_moment_expansion}) and (\ref{eq:general_fourth_bound}) as well as the Cauchy--Schwarz inequality (twice),
$$
    \begin{aligned}
        \E[(D_{xy}^2\Sigma(\PP_n))^4]&  \le 
        4^8
        \E\big[\overline f_{\sup}(\PP_n)^8\big]^{1/2}
        \E\big[\PP(Q(x,R_n))^{16} \big]^{1/4} \E\big[\PP(Q(y,R_n))^{16} \big]^{1/4}\\
        &     
       \qquad \qquad \qquad \qquad \qquad + 512 \E\big[\overline f_{\sup}(\PP_n)^4\big]  +   2048 n^{4\eps} \vert  W_n\vert^8 \e^{-\beta_1(\eps) n^\eps / 4}.
    \end{aligned}
    $$
    Invoking Assumption \ref{ass:A}(iii), i.e., sub-polynomial moments, as well as Lemmas \ref{lem:moments_of_Q} and \ref{lem:normalization} completes the proof.
    \end{proof}
\begin{lemma}[$4^{\text{th}}$ moment of $D_{xy}^2$; Case III]
    \label{lem:fourth_moment_second_case_three}
    Let $\eps > 0$, $n \gg 1$. Under Assumption \ref{ass:A},
   $$
        \sup_{x \in  \W_n} \sup_{y \in \Q(x,n^\eps)}\E[(D_{xy}^2\Sigma(\PP_n))^4]
         \le n^{4k\eps} \vert S_n^k \vert^4.
    $$
\end{lemma} 
\begin{proof}
    The idea is to follow the same strategy as in the proof of Lemma \ref{lem:fourth_moment_first_case}, where again we choose $\gamma = n^\eps$. By $k$-locality, it follows that
    $$
    \vert F(x, \PP_n, \PP_n^{xy}) - F(x,\PP_n, \PP_n^x) \vert \le \overline f_{\sup}(\PP_n) \PP(S_n^k(x,1)) ,
    $$
    $$
    \vert F(y, \PP_n, \PP_n^{xy}) - F(y,\PP_n, \PP_n^y) \vert \le \overline f_{\sup}(\PP_n) \PP(S_n^k(y,1)),
    $$
    and since $y \in \Q(x,n^\eps)$, then further that
    \begin{align*}
    \sum_{Z \in \PP_n}  \big( F(Z,\PP_n,\PP_n^{xy}) - F(Z,\PP_n,\PP_n^x)
     & - F(Z,\PP_n,\PP_n^y) + F(Z,\PP_n,\PP_n) \big) \\
     & \qquad \le 2 \overline f_{\sup}(\PP_n)\PP(Q(x,R_n))\PP(Q(x,2n^\eps)).
    \end{align*}
Since
$$
1+ 2\PP(S_n^k(0,1))^4
    + 2 \PP(Q(x,R_n))^4\PP(Q(0,2n^\eps))^4 \le 4 \big( \PP(Q(x,R_n)) \vee 1 \big)^4 \big( \PP(S_n^k(0,2n^\eps))\vee 1 \big)^4,
$$
then using (\ref{eq:fourth_moment_expansion}) and (\ref{eq:general_fourth_bound}) as well as the Cauchy--Schwarz inequality (twice),
$$
    \begin{aligned}
        \E[(D_{xy}^2\Sigma(\PP_n))^4]&  \le 
        256
        \E\big[\overline f_{\sup}(\PP_n)^8\big]^{1/2}
        \E\big[\PP(Q(x,R_n))^{16} \big]^{1/4} \E\big[\PP(S_n^k(0,2n^\eps))^{16} \big]^{1/4} \\
        &       
    \qquad +   2048 n^{4\eps} \vert  W_n\vert^8 \e^{-\beta_1(\eps) n^\eps / 4}.
    \end{aligned}
    $$
    Since $\vert S_n^k(0,2n^{\eps}) \vert = 4^k n^{k\eps} \vert S_n^k \vert$, then invoking Assumption \ref{ass:A}(iii), i.e., sub-polynomial moments, as well as Lemmas \ref{lem:poisson_moment_bound}, \ref{lem:moments_of_Q} and \ref{lem:normalization} completes the proof.
\end{proof}

\begin{proof}[Proof of Lemma \ref{lem:en1}(i)]
Recall we want to prove that
\begin{align*}
    I_{n,1}  = &\int_{\W_n^3}
     (\E\big[D_x \Sigma(\PP_n)^2
    D_y \Sigma(\PP_n)^2\big]\E\big[D_{x,z}^2 \Sigma(\PP_n)^2
    D_{y,z}^2 \Sigma(\PP_n)^2\big])^{\tfrac 1 2} \ \text d(x,y,z) \le n^{\eps} \vert  W_n \vert \vert S_n^k \vert^4.
\end{align*}
First, by applying the Cauchy--Schwarz inequality to both expectations,
\begin{align*}
    I_{n,1} \le &  \int_{ \W_n}\Big(\int_{ \W_n }
    \E\big[D_y \Sigma(\PP_n)^4\big]^{\tfrac 1 4}\E\big[D_{xy}^2 \Sigma(\PP_n)^4\big]^{\tfrac 1 4} \ \text dy\Big)^2 \ \text dx.
\end{align*}
Next, by decomposing the inner integral using $\eps' = \eps$,
\begin{align*}
	I_{n,1}  = &   \int_{ \W_n}\Big(\int_{ \W_n \setminus \S_n^k(x,n^\eps)  }
    \E\big[D_y \Sigma(\PP_n)^4\big]^{\tfrac 1 4}\E\big[D_{xy}^2 \Sigma(\PP_n)^4\big]^{\tfrac 1 4} \ \text dy\Big)^2 \ \text dx\\
    & \qquad +  \int_{ \W_n}\Big(\int_{ \S_n^k(x,n^\eps) \setminus \Q(x,n^\eps) }
    \E\big[D_y \Sigma(\PP_n)^4\big]^{\tfrac 1 4}\E\big[D_{xy}^2 \Sigma(\PP_n)^4\big]^{\tfrac 1 4} \ \text dy\Big)^2 \ \text dx\\
    & \qquad \qquad  +  \int_{ \W_n}\Big(\int_{ \Q(x,n^\eps) }
    \E\big[D_y \Sigma(\PP_n)^4\big]^{\tfrac 1 4}\E\big[D_{xy}^2 \Sigma(\PP_n)^4\big]^{\tfrac 1 4} \ \text dy\Big)^2 \ \text dx,
\end{align*}
then by Lemmas \ref{lem:fourth_moment} - \ref{lem:fourth_moment_second_case_three},
\begin{align*}
	I_{n,1}  \le&\int_{ \W_n}\Big(\int_{\W_n \setminus \S_n^k(x,n^\eps) } n^{\eps'/4} \vert S_n^k \vert \e^{-\beta_1(\eps' )n^{\eps'} / 64}\ \text dy\Big)^2 \ \text dx\\
    & \qquad  +  \int_{ \W_n}\Big(\int_{\S_n^k(x,n^\eps) \setminus \Q(x,n^\eps)}  n^{(d+1)\eps} \vert S_n^k(0,n^\eps)\vert  n^\eps \text dy\Big)^2 \text dx\\
    &  \qquad \qquad  + \int_{ \W_n}\Big(\int_{\Q(x,n^\eps)}   n^{(d+1)\eps} \vert S_n^k(0,n^\eps) \vert  n^{\eps} \big\vert S_n^k(0,2n^\eps) \big\vert \text dy\Big)^2 \text dx.
\end{align*}
As the first integral goes to zero at exponential speed, it follows that
\begin{align*}
	I_{n,1}  \le&  32\vert  W_n \vert \big(\vert S_n^k(0,n^\eps) \vert - \vert Q(0,n^\eps)\vert \big)^2 n^{2(d+2)\eps} \vert S_n^k(0,n^\eps)\vert^2\\
    &    + 8  \vert  W_n \vert \vert Q(x,n^\eps) \vert^2 n^{2(d+2)\eps} \vert S_n^k(0,n^\eps) \vert^2 \vert S_n^k(0,2n^\eps) \vert^2.
\end{align*}
Using the definition of the non-stable cube $Q(\cdot,\cdot)$ and the $k$-fold vertical slab $S_n^k(\cdot,\cdot)$, it follows
for all $n \gg 1$ that
$
I_{n,1} \le 2^{2(d-k+4)} n^{4(d+1)\eps} \vert  W_n \vert \vert S_n^k(0,n^\eps) \vert^4,
$
which completes the proof.
\end{proof}

\begin{proof}[Proof of Lemma \ref{lem:en1}(ii)]
Recall we want to prove that
    $$
    I_{n,2} = \int_{( \W_n)^3} \E\big[(D_{x,z}^2 \Sigma(\PP_n))^2
    (D_{y,z}^2 \Sigma(\PP_n))^2\big] \ \text d(x,y,z)
     \le n^{2(d+2) \eps } \vert  W_n \vert \vert S_n^k(0,n^\eps)\vert^4.
    $$
Following the steps in Lemma \ref{lem:en1}: First, by the Cauchy--Schwarz inequality,
$$
I_{n,2} \le \int_{ \W_n}\Big(\int_{ \W_n}\sqrt{\E\big[(D_{xy}^2 \Sigma(\PP_n))^4\big]} \text dy\Big)^2 \text dx.
$$
Using the same decomposition as previously, 
it follows from Lemma \ref{lem:fourth_moment_first_case}---\ref{lem:fourth_moment_second_case_three} that
\begin{align*}
	I_{n,2}  \le& \int_{ \W_n}\Big(\int_{\W_n \setminus \S_n^k(x,n^\eps) } n^{2\eps} \vert  W_n\vert^4 \e^{-\beta_1(\eps) n^\eps / 8}\text dy\Big)^2 \text dx\\
    &  +  \int_{ \W_n}\Big(\int_{\S_n^k(x,n^\eps) \setminus \Q(x,n^\eps)}   n^{2\eps} \text dy\Big)^2 \text dx + \int_{ \W_n}\Big(\int_{\Q(x,n^\eps)}  n^{2\eps} \big\vert S_n^k(0,2n^\eps) \big\vert^2 \text dy\Big)^2 \text dx.
\end{align*}
Again the first integral goes to zero at exponential speed, and for all  $n \gg 1$,
$$
I_{n,2} \le 2^{4(d-k)+2}  n^{2(d+2) \eps }\vert  W_n \vert  \vert S_n^k(0,n^\eps)\vert^4,
$$
which completes the proof.
\end{proof}

\section{Bounding error terms: Sum-log-sum case}
\label{sec:5}

This section is dedicated to proving the error terms bounds which comes out of applying the Malliavin-Stein Normal approximation to the sum-log-sum functional. 
To that end, we first record some initial consequences of Assumption \ref{ass:B} in Section \ref{sec:5.1},
then proceed to use these consequences to bound moments of the first-order difference operators in Section \ref{sec:5.2}
and prove Lemma \ref{lem:en1_sim}(iii). Next, in Section \ref{sec:5.3}, we record some further properties guaranteed by Assumption \ref{ass:B},
and then finally in Section \ref{sec:5.4}, we can bound the fourth moment of the second-order difference operators
as well as prove Lemma \ref{lem:en1_sim}(i)-(ii).

\subsection{First consequences of Assumption \ref{ass:B}}
\label{sec:5.1}

First of all, as Assumption \ref{ass:B}(i) is exactly the same type as Assumption \ref{ass:A}(ii), we can use the same approach in the proof of Lemma \ref{lem:moments_of_Q} to obtain
an analogue result that moments of the size of the cube with side length $R_n(x)$ are all of constant order, even when $R_n(x)$ is defined
 as in \eqref{eq:stabilization_radius} compared to as in \eqref{eq:stabilization_radius_original}.
\begin{lemma}[Moments of $\mathbb{Q}(x,R_n)$]
    \label{lem:moments_of_Q_tilde}
    Let $m \ge 1$. Under Assumption \ref{ass:B}, there is a constant $C_1(m)>0$ depending on $m$ such that 
$$
   \sup_{n \gg 1} \sup_{x \in  \W_n} \E\big[\PP(\mathbb{Q}(x, R_n))^{m}\big] \le C_1(m).
$$
    \end{lemma}

Before proceeding, similar to the previous section, we record some useful moment bounds on the log-transformed compound scores.
Let $F_n$ denote the set from \eqref{eq:exponential_concentration} and $\widetilde F_n$ denote a measurable extension
as described in Assumption \ref{ass:B}(ii).

\begin{lemma}[Compound score bounds; Part I]
    \label{lem:log_moments_G}
    Let $m \ge 1$, $\eps > 0$. Under Assumption \ref{ass:B}, then for all $n \gg 1$,

    \begin{equation}
        \label{eq:G_bound_1}
    \sup_{\mathrm{x}\in \mc R}\E \Big[ \Big \vert \sum_{Z \in \mc A_n^+(\PP_n^\mathrm{x})}\log G(Z,\PP_n^\mathrm{x}) \Big \vert^m \Big] \le n^{m \eps} \vert  W_n \vert^{2m},
    \end{equation}
    \begin{equation}
        \label{eq:G_bound_2}
\sup_{x \in  \W_n}\E \Big[ \Big\vert \sum_{Z \in \mc A_n^+(\PP_n^x) \cap \mc A_n^+(\PP_n)} \log \Big(\frac{G(Z,\PP_n^x)}{G(Z,\PP_n)}\Big)\Big\vert^m \one_{\widetilde F_n^c} \Big] \le n^{m \eps},
    \end{equation}
        \begin{equation}
        \label{eq:G_bound_3}
\sup_{x, y \in  \W_n}\E \Big[ \Big\vert \sum_{Z \in \mc A_n^+(\PP_n^{xy}) \cap \mc A_n^+(\PP_n^y)} \log \Big(\frac{G(Z,\PP_n^{xy})}{G(Z,\PP_n^y)}\Big)\Big\vert^m \one_{\widetilde F_n^c} \Big] \le n^{m \eps},
    \end{equation}
\end{lemma}
While we currently only need to prove bounds involving a single added point $x$, we also record the bound in \eqref{eq:G_bound_3} involving two added points $x,y$
as prepartion for bounding the second-order difference operators in Section \ref{sec:5.4}, since this bound can be obtained in exactly the same way as \eqref{eq:G_bound_2}, 
and hence we only prove \eqref{eq:G_bound_1} and \eqref{eq:G_bound_2} in detail. Moreover, it is possible to obtain a tighter bound in \eqref{eq:G_bound_1} as $n^{m\eps}\vert W_n \vert^m$ using more intricate steps via Poisson concentration, but the current bound suffices for our purposes.
\begin{proof}[Proof of Lemma \ref{lem:log_moments_G}]

First, applying the inequality $ \log(x) \le x $ and Lemma \ref{lem:compound_score_bounds} immediately yields \eqref{eq:G_bound_1}.
Next, to prove \eqref{eq:G_bound_2}, we let $\mc V_n^+(\PP_n^x) = \mc A_n^+(\PP_n^x) \cap \mc A_n^+(\PP_n)$ and
    $$
     D_{n,1}  = \sum_{Z \in \mc V_n^+(\PP_n^x)} \log \Big(\frac{G(Z,\PP_n^x)}{G(Z,\PP_n)}\Big).
     $$
Since $G(Z,\PP_n)\ge 1$ for all $Z \in \mc A_n^+(\PP_n)$, it follows by the Cauchy--Schwarz inequality, Lemma \ref{lem:log_moments_G} and Assumption \ref{ass:B}(ii) that
$
\E \big[ \vert D_{n,1}\vert^m \one_{F_n}\big] \le n^{m \eps} \e^{-\beta_4 \vert S_n^k \vert /2}.
$
Next, we let
$$
a(Z) = \frac{G(Z,\PP_n^x)- G(Z,\PP_n)}{G(Z,\PP_n)},
$$
and note that
\begin{align*}
\vert D_{n,1} \vert \one_{F_n^c}
\le &
\Big \vert \sum_{Z \in \mc V_n^+(\PP_n^x) } \log 
\big( 1+  a(Z)  \big)\one\{\vert a(Z) \vert \le 1/2\} \Big  \vert \one_{\widetilde F_n^c} \\
&  \qquad \qquad +  
\Big \vert \sum_{Z \in \mc V_n^+(\PP_n^x)} \log 
\big( 1+  a(Z) \big) \one\{\vert a(Z) \vert > 1/2\} \Big \vert \one_{\widetilde F_n^c}.
\end{align*}
Note that when $\vert a(Z) \vert \le 1/2$, it follows that $\log(1+  a(Z) ) \le  2 \vert a(Z) \vert$.
Thus, by an additional case distinction, it follows by definition
of the non-stabilizing cube, Assumption \ref{ass:A}(i), i.e., $k$-locality, and the fact that $G(Z,\PP_n) > \beta_3 \vert S_n^k \vert$ on $\widetilde F_n^c$ that
\begin{equation}
    \label{eq:log_moments_decomposition}
\begin{aligned}
  & \vert D_{n,1} \vert \one_{\widetilde F_n^c}
    \le  \frac{2}{\beta_3 \vert S_n^k \vert}
    \Big \vert \sum_{Z \in  \PP_n \cap \S_n^k(x,1)} 
    g(Z,x, \PP_n^x) \Big\vert  \one_{\widetilde F_n^c} \\
    & \ \ + \frac{2}{\beta_3 \vert S_n^k \vert}
   \Big \vert \sum_{Z \in  \PP_n \cap \Q(x,R_n)} \sum_{V \in \PP_n \cap \S_n^k(Z,1)} \vert g(Z,V, \PP_n^x) 
    - g(Z,V, \PP_n) \vert\Big  \vert  \one_{\widetilde F_n^c}\\
    & \ \ + \frac{2}{\beta_3 \vert S_n^k \vert}
   \Big \vert  \sum_{V \in \PP_n \cap \Q(x,R_n)}\sum_{Z \in  \PP_n \cap \S_n^k(V,1)} \vert g(Z,V, \PP_n^x) 
    - g(Z,V, \PP_n) \vert\Big  \vert  \one_{ \widetilde F_n^c}\\
    & \ \ +  
    \Big \vert \sum_{Z \in \mc V_n^+(\PP_n^x) \cap \Q(x,R_n)} \log 
    \big( 1+ a(Z) \big) \one\{\vert G(Z,\PP_n^x)- G(Z,\PP_n)\vert > \tfrac{\beta_3  \vert S_n^k \vert}{2} \vert\} \Big \vert  \one_{\widetilde F_n^c} \\
    & \ \ +
   \Big \vert \sum_{Z \in \mc V_n^+(\PP_n) \cap \Q(x,R_n)^c} \log 
    \big( 1+  a(Z) \big) \one\{\vert G(Z,\PP_n^x)- G(Z,\PP_n)\vert > \tfrac{\beta_3  \vert S_n^k \vert }{2}\} \Big \vert  \one_{\widetilde F_n^c}.
    \end{aligned}
\end{equation}

Note that again by the non-stabilization cube definition, then
if $Z \in \Q(x,R_n)^c$,
\begin{align*}
\vert G(Z, \PP_n^x)- G(Z,\PP_n)\vert 
\vert  g(Z,V, \PP_n^x) - g(Z,V, \PP_n) \vert
\le \overline g_{\sup}(\PP_n) \PP(\Q(x,R_n)),
\end{align*}
where $\overline g_{\sup}(\PP_n)$ is as in Assumption \ref{ass:A}(iii). By the Cauchy--Schwarz inequality, reinserting the definition of $a(Z)$ and using that $G(Z,\PP_n) \ge 1$ whenever $Z \in \mc A_n^+(\PP_n)$, then for $n \gg 1$,
\begin{equation}
    \label{eq:log_moments_term_5_pre}
    \begin{aligned}
\E& \Big[\Big \vert  \sum_{Z \in \mc V_n^+(\PP_n^x) \cap \Q(x,R_n)^c}  \log 
    \big( 1+  a(Z) \big) \one\{\vert G(Z,\PP_n^x)- G(Z,\PP_n)\vert > \tfrac{\beta_3  \vert S_n^k \vert}{2} \vert\} \Big \vert^m  \one_{\widetilde F_n^c} \Big] \\
&  \ \le \E\Big[\Big \vert \sum_{Z \in \PP_n^x \cap \Q(x,R_n)^c}
     G(Z,\PP_n^x) \Big\vert^{2m}  \one_{\widetilde F_n^c}  \Big]^{1/2}
     P\Big(\overline g_{\sup}(\PP_n) \PP(\Q(x,R_n)) > \tfrac{\beta_3  \vert S_n^k \vert}{2}\Big)^{1/2}.
    \end{aligned}
\end{equation}
Using that $\{ rs > t \} \su \{ r  > \sqrt{t} \} \cup \{ s > \sqrt{t} \}$ and Markov's inequality,
$$
P\big(\overline g_{\sup}(\PP_n) \PP(\Q(x,R_n)) > \tfrac{\beta_3  \vert S_n^k \vert}{2}\big) \le \Big(\frac{2}{\beta_3 \vert S_n^k \vert}\Big)^{4md}\big( \E\big[\overline g_{\sup}(\PP_n)^{4md}\big] + \E\big[\PP(\Q(x,R_n))^{4md}\big]\big).
$$
Hence, by Assumption \ref{ass:A}(iii) as well as Lemmas \ref{lem:moments_of_Q_tilde} and \ref{lem:compound_score_bounds}, then for any $n \gg 1$,
\begin{equation}
    \label{eq:log_moments_term_5}
    \begin{aligned}
\E\Big[\Big \vert & \sum_{Z \in \mc V_n^+(\PP_n^x) \cap \Q(x,R_n)^c}  \log 
    \big( 1+  a(Z) \big) \one\{\vert G(Z,\PP_n^x)- G(Z,\PP_n)\vert > \tfrac{\beta_3  \vert S_n^k \vert}{2} \vert\} \Big \vert^m  \one_{\widetilde F_n^c} \Big] \\
& \qquad \le 2 n^{m \eps }\vert  W_n \vert^{2m} \Big(\frac{2}{\beta_3 \vert S_n^k \vert}\Big)^{2md} n^{2md\eps } \le 1.
    \end{aligned}
\end{equation}
Bounding the indicator function by 1, and combinining the Cauchy--Schwarz inequality with Lemma \ref{lem:moments_of_Q_tilde} and Assumption \ref{ass:A}(iii), 
we obtain with $C_1 = C_1(m)$ for all $n \gg 1$ that
\begin{equation}
    \label{eq:log_moments_term_4}
    \begin{aligned}
\E\Big[  & \Big \vert \sum_{Z \in \mc V_n^+(\PP_n^x) \cap \Q(x,R_n)} \log 
    \big( 1+ a(Z) \big) \one\{\vert G(Z,\PP_n^x)- G(Z,\PP_n)\vert > \tfrac{\beta_3  \vert S_n^k \vert}{2} \vert\} \Big \vert^m  \one_{\widetilde F_n^c} \Big]
    \le C_1n^{m\eps}.
    \end{aligned}
\end{equation}
Hence by combining (\ref{eq:log_moments_decomposition}), (\ref{eq:log_moments_term_5}), (\ref{eq:log_moments_term_4}), Lemmas \ref{lem:compound_score_bounds} 
and \ref{lem:reality},
$$
\E[\vert D_{n,1} \vert^m \one_{\widetilde F_n^c}] \le \frac{5^{m-1}}{\beta_3^m} n^{m \eps} +\frac{5^{m-1}}{\beta_3^m} n^{m \eps} + \frac{5^{m-1}}{\beta_3^m} n^{m \eps} + 5^{m-1} + 5^{m-1} C_1 n^{m \eps},
$$
for all $n \gg 1$. Invoking Lemma \ref{lem:normalization} completes the proof.
\end{proof}

\subsection{Moment bounds on $D_x\Sigma^{\log}_n$ and proof of Lemma \ref{lem:en1_sim}(iii)}
\label{sec:5.2}
In this section, we prove the bound on the third error term $\widetilde I_{n,3}$, which only involves the first-order difference operator,
and hence we want to control the moments of $D_x\Sigma^{\log}_n(\PP_n)$. For the rest of Section \ref{sec:5}, we now omit all explicit references to \ref{lem:normalization} and \ref{lem:reality} inside proofs
make the arguments more readable.

\begin{lemma}[$3^{\text{rd}}$/$4^{\text{th}}$ moment of $D_{x}\Sigma^{\log}_n$]
    \label{lem:fourth_moment_log}
   Let $\eps > 0$ and $n \gg 1$. Under Assumption \ref{ass:B},
    \begin{align}
   \sup_{x \in  \W_n} \E[\vert D_{x}\Sigma^{\log}_n(\PP_n)\vert ^3] \le &  n^{\eps },\\
    \sup_{x \in  \W_n} \E[\vert D_{x}\Sigma^{\log}_n(\PP_n)\vert^4] \le & n^{\eps }.
    \end{align}
    \end{lemma}
\begin{proof}
First, recall that for $x \in \R^d$,
$$
D_x \Sigma^{\log}_n (\PP_n) = \sum_{Z \in \mc A_n^+(\PP_n^x)} \log G(Z,\PP_n^x) - \sum_{Z \in \mc A_n^+(\PP_n)} \log G(Z,\PP_n).
$$
By definition of the stabilization radius $R_n(x)$, it follows that if $Z$ lies in $\mc A_n^+(\PP_n^x) \cap \mc A_n^+(\PP_n)^c$ or in $\mc A_n^+(\PP_n) \cap \mc A_n^+(\PP_n^x)^c$,
 then $Z \in \Q(x,R_n)$. Thus,
 $
\vert D_x \Sigma^{\log}_n (\PP_n) \vert \le \vert D_{n,1} \vert + \vert D_{n,2} \vert + \vert D_{n,3} \vert,
$
where 
\begin{align*}
    D_{n,1}  = & \sum_{Z \in \mc A_n^+(\PP_n^x) \cap \mc A_n^+(\PP_n)} \log \Big(\frac{G(Z,\PP_n^x)}{G(Z,\PP_n)}\Big),\\
    D_{n,2}  = & \sum_{Z \in \mc A_n^+(\PP_n^x)\cap \Q(x,R_n)} \log G(Z,\PP_n^x),\\
    D_{n,3}  = & \sum_{Z \in \mc A_n^+(\PP_n)\cap \Q(x,R_n)} \log G(Z,\PP_n).
\end{align*}
By Lemmas \ref{lem:moments_of_Q_tilde} and \ref{lem:log_moments_G}, it follows for $j \in \{2,3\}$ and $F_n$ as the set defined in \eqref{eq:exponential_concentration} that,
\begin{equation}
    \label{eq:log_moments_2_3}
    \E[\vert D_{n,j} \vert^m \one_{F_n^c}] \le C_1(m) n^{m \eps}.
\end{equation}
Moreover, since $G(Z,\PP_n)\ge 1$ for all $Z \in \mc A_n^+(\PP_n)$, it follows by the Cauchy--Schwarz inequality, Lemma \ref{lem:log_moments_G} and Assumption \ref{ass:B}(ii) that for any $j \in \{1,2,3 \}$ that
\begin{equation}
    \label{eq:log_moments_1}
\E[\vert D_{n,j} \vert^m \one_{F_n}] \le n^{m \eps} \e^{-\beta_4 \vert S_n^k \vert /2} \le n^{m \eps}.
\end{equation}
Thus, combining (\ref{eq:log_moments_2_3}), (\ref{eq:log_moments_1}) and Lemma \ref{lem:log_moments_G},
$$
\E \big[\vert  D_x \Sigma^{\log}_n (\PP_n)  \vert^m \big] \le 2 \cdot 6^{m-1} C_1(m) n^{m \eps} + 4 \cdot 6^{m-1} n^{m \eps},
$$
which completes the proof.

\end{proof}

\begin{proof}[Proof of Lemma \ref{lem:en1_sim}(iii)]
    We want to prove that
    $$
    \widetilde I_{n,3} = \int_{ \W_n }\E\big[|D_x\Sigma(\PP_n)|^3\big] \ \text dx  \le  n^{6 \eps }  \vert  W_n \vert.
    $$
    However, by Lemma \ref{lem:fourth_moment_log} and that $\mu(\M)=1$, this immediately completes the proof.
\end{proof}

\subsection{Additional consequences of Assumption \ref{ass:B}}
\label{sec:5.3}
In this section, we establish some further consequences of Assumption \ref{ass:B}. Momentarily, we decompose the second-order difference operator into a main term and an error term, and this subsection is thus
dedicated to provide moment bounds on components of these terms. In particular when controlling the fourth moment of this difference operator, we need to bound sums of terms on the form,
$$
\Big\vert\log \Big(\frac{G(Z,\PP_n^y)}{G(Z,\PP_n)}\Big)\Big\vert^4, \quad \text{ and } \quad \Big\vert \log \Big( \frac{G(Z,\PP_n^{xy})G(Z,\PP_n)}{G(Z,\PP_n^x)G(Z,\PP_n^y)} \Big) \Big\vert^4,
$$
where in particular, we need some strong bounds when $\Q(x,R_n) \cap \Q(y,R_n) = \emptyset$, and hence we will prove different bounds based on whether or not this condition is met. When handling
then right-most logarithm, we also utilize a Poisson concentration inequality.

To formalize the above, we begin by considering the second-order difference operator,
\begin{align*}
	D_{xy}^2 \Sigma^{\log}_n(\PP_n) =&\sum_{Z\in \mc A_n^+(\PP_n^{xy})}\log G\!\left(Z,\mathcal{P}_n^{xy}\right)
         -\sum_{Z\in \mc A_n^+(\PP_n^{y})}\log G\!\left(Z,\mathcal{P}_n^{y} \right) \\
      & \qquad \qquad \qquad \qquad  -\sum_{Z\in \mc A_n^+(\PP_n^{x})}\log G\!\left(Z,\mathcal{P}_n^{x} \right)
        +\sum_{Z\in \mc A_n^+(\PP_n)}\log G\!\left(Z,\mathcal{P}_n \right).
\end{align*}
Let $\mc U_n^+(\PP_n^{xy})=\mc A_n^+(\PP_n^{xy}) \cap \mc A_n^+(\PP_n^{x}) \cap \mc A_n^+(\PP_n^{y}) \cap \mc A_n^+(\PP_n)$. It will prove useful to then consider the following decomposition,
\begin{equation}
    \label{eq:decomposition_D2_log}
D_{xy}^2 \Sigma^{\log}_n(\PP_n)= M_n + C_n
\end{equation}
where
\begin{align*}
    M_n  = & \sum_{Z \in \mc  U_n^+(\PP_n^{xy})}
    \log \Big( \frac{G(Z,\PP_n^{xy})G(Z,\PP_n)}{G(Z,\PP_n^x)G(Z,\PP_n^y)} \Big) \\[0.5em]
    C_n  = & \sum_{Z \in \mc A_n^+(\PP_n^{xy}) \cap \mc  U_n^+(\PP_n^{xy})^c  } \log G(Z,\PP_n^{xy})
     - \sum_{Z \in \mc A_n^+(\PP_n^{x}) \cap \mc  U_n^+(\PP_n^{xy})^c  } \log G(Z,\PP_n^x) \\[0.5em]
         & \qquad \qquad \qquad    -\sum_{Z \in \mc A_n^+(\PP_n^{y}) \cap \mc  U_n^+(\PP_n^{xy})^c  } \log G(Z,\PP_n^y)
       + \sum_{Z \in \mc A_n^+(\PP_n) \cap \mc  U_n^+(\PP_n^{xy})^c  }\log G(Z,\PP_n).
\end{align*}
Here we think of $M_n$ as the main term, and $C_n$ as a cross-term. The contributions from the cross-term $C_n$ will, loosely speaking, 
be 0 in Case I, non-zero but vanishing when $n$ increases in Case II, and grow very slowly with $n$ in Case III.  First on the agenda is to establish some crude moment bounds on both of these terms.

\begin{lemma}[Crude $4^{\text{th}}$ moment bounds on $M_n$ and $C_n$]
    \label{lem:fourth_moment_main_term}
    Let $\eps > 0$ and $n \gg 1$. Under Assumption \ref{ass:B}, it holds for any event $E_n = E_n(\eps)$ that 
\begin{align}
        \E[M_n^4 \one_{E_n}] &\le  n^{4\eps} \vert  W_n \vert^{8} \sqrt{P(E_n)}, \label{eq:fourth_moment_main_term} \\
        \E[C_n^4 \one_{E_n}] & \le n^{4\eps} \sqrt{P(E_n)}. \label{eq:fourth_moment_cross_term}
\end{align}
    \end{lemma}
\begin{proof}
First, by the Cauchy--Schwarz inequality,
\begin{equation}
    \label{eq:m_cauchy}
\E[M_n^4 \one_{E_n}] \le \sqrt{\E[M_n^8]} \sqrt{P(E_n)}.
\end{equation}
Since $Z \in \mc U_n^+(\PP_n^{xy})$ which in particular implies $G(Z,\PP_n^x) \ge 1$ and $G(Z,\PP_n^y) \ge 1$, 
\begin{equation}
    \label{eq:m_8_sum}
\vert M_n \vert \le \Big\vert  \sum_{Z \in \mc A_n^+(\PP_n^{xy})} \log G(Z,\PP_n^{xy})\Big\vert   +\Big\vert  \sum_{Z \in \mc A_n^+(\PP_n)} \log G(Z,\PP_n) \Big\vert.
\end{equation}
Thus, by Lemma \ref{lem:log_moments_G},
\begin{equation}
\label{eq:m_8}
\E[M_n^8] \le n^{8\eps} \vert  W_n \vert^{16}.
\end{equation}
Combining \eqref{eq:m_cauchy} and \eqref{eq:m_8} yields \eqref{eq:fourth_moment_main_term} when performing a change of variable in $\eps$. Next, by the Cauchy--Schwarz inequality again, 
$
\E[C_n^4 \one_{E_n}] \le \sqrt{\E[C_n^8]} \sqrt{P(E_n)}.
$
As usual, it follows by Assumption \ref{ass:B}(ii) that it suffices to show the bound $\sqrt{\E[C_n^8]} \le n^{\eps}$ on the event $\widetilde F_n^c$.
 By the definition of the stabilization radii $R_n(x)$ and $ R_n(y)$,
$$
\mc U_n^+(\PP_n^{xy})^c \subseteq \mathbb{Q}(x,R_n) \cup \mathbb{Q}(y,R_n).
$$
Thus, using Lemmas \ref{lem:moments_of_Q_tilde} and \ref{lem:log_moments_G} yields \eqref{eq:fourth_moment_cross_term}, and hence completes the proof.
\end{proof}

For bounding the second-order difference operator, we need to consider whether or not the two points $x$ and $y$ are close or not. 
To that end, introduce the event that the non-stable cubes around $x$ and $y$ do not intersect, i.e.,
\begin{equation}
    \label{eq:disjoint_event}
H_n = H_n(x,y) = \big\{\mathbb{Q}(x,R_n) \cap \mathbb{Q}(y,R_n) = \es \big\}.
\end{equation}
First, we establish a series of bounds on the terms inside $M_n$ on the event $\widetilde F_n^c \cap H_n$,
i.e., when points cannot simultaneously lie in the cubes around both $x$ and $y$, and when the compound score in every point is sufficiently large.
\begin{lemma}[Compound score bounds; Part II]
    \label{lem:log_ratio_G_2}
    Let $m \ge 1$, $\eps > 0$. Under $\text{Assumption \ref{ass:B}}$, it holds for all $n \gg 1$,
\begin{align}
\sup_{x,y \in  \W_n}  \E \Big[ \Big\vert \sum_{Z \in \mc A_n^+(\PP_n^y) \cap \mc A_n^+(\PP_n) \cap \Q(x,R_n)} 
\log \Big(\frac{G(Z,\PP_n^y)}{G(Z,\PP_n)}\Big)\Big\vert^m \one_{\widetilde F_n^c \cap H_n} \Big] & \le \frac{n^{\eps}}{\vert S_n^k \vert^m}, \label{eq:compound_score_2_1} \\
\sup_{x,y \in  \W_n} \E \Big[ \Big\vert \sum_{Z \in \mc A_n^+(\PP_n^{xy}) \cap \mc A_n^+(\PP_n^y) \cap \Q(x,R_n)} 
\log \Big(\frac{G(Z,\PP_n^{xy})}{G(Z,\PP_n^y)}\Big)\Big\vert^m \one_{\widetilde F_n^c \cap H_n} \Big] & \le \frac{n^{\eps}}{\vert S_n^k \vert^m}, \label{eq:compound_score_2_2}
\end{align}
\end{lemma}
\begin{proof}
    First we note that we will largely reuse and slightly modify the arguments from the proof of Lemma \ref{lem:log_moments_G}: Let
$$
D_n = \sum_{Z \in \mc A_n^+(\PP_n^y) \cap \mc A_n^+(\PP_n) \cap \Q(x,R_n)} 
\log \Big(\frac{G(Z,\PP_n^y)}{G(Z,\PP_n)}\Big)
$$
and
$$
a(Z) = \frac{G(Z,\PP_n^y)- G(Z,\PP_n)}{G(Z,\PP_n)},
$$
and note that
\begin{align*}
\vert D_{n} \vert \one_{\widetilde F_n^c \cap H_n}
\le &
\Big \vert \sum_{Z \in \mc A_n^+(\PP_n^y) \cap \mc A_n^+(\PP_n) \cap \Q(x,R_n)} \log 
\big( 1+ a(Z)  \big)\one\{\vert a(Z) \vert \le 1/2\} \Big  \vert \one_{\widetilde F_n^c \cap H_n} \\
&  \qquad +  
\Big \vert \sum_{Z \in \mc A_n^+(\PP_n^y) \cap \mc A_n^+(\PP_n) \cap \Q(x,R_n)} \log 
\big( 1+ a(Z) \big) \one\{\vert a(Z) \vert > 1/2\} \Big \vert \one_{\widetilde F_n^c \cap H_n}.
\end{align*}

Note that by definition of $H_n$ and the definition of the non-stable cube,
 then for any $Z \in \Q(x,R_n)$, it must hold that $g(Z,V,\PP_n^y) = g(Z,V,\PP_n)$ for any $V \in \PP_n \cap \Q(y,R_n)^c$.
Hence, by Assumption \ref{ass:A}(i), i.e., $k$-locality, and the fact that $G(Z,\PP_n) > \beta_3 \vert S_n^k \vert$ on $\widetilde F_n^c \cap H_n$,
\begin{equation}
    \label{eq:log_moments_decomposition_again}
\begin{aligned}
  & \vert D_{n} \vert \one_{\widetilde F_n^c \cap H_n}
    \le  \frac{1}{\beta_3 \vert S_n^k \vert}
    \Big \vert \sum_{Z \in  \PP_n \cap \Q(x,R_n)} 
    g(Z,y, \PP_n^y) \Big\vert  \one_{\widetilde F_n^c \cap H_n} \\
    & \ \ + \frac{1}{\beta_3 \vert S_n^k \vert}
   \Big \vert \sum_{Z \in  \PP_n \cap \Q(x,R_n)} \sum_{V \in \PP_n \cap \Q(y,R_n)} \vert g(Z,V, \PP_n^y) 
    - g(Z,V, \PP_n) \vert\Big  \vert  \one_{\widetilde F_n^c \cap H_n}\\
    & \ \ +  
    \Big \vert \sum_{Z \in \PP_n \cap \Q(x,R_n)} \log 
    \big( 1+ \vert a(Z) \vert \big) \one\{\vert G(Z,\PP_n^y)- G(Z,\PP_n)\vert > \tfrac{\beta_3  \vert S_n^k \vert}{2} \vert\} \Big \vert  \one_{\widetilde F_n^c \cap H_n}.
    \end{aligned}
\end{equation}
By the Cauchy--Schwarz inequality, assumption \ref{ass:A3} and Lemma \ref{lem:moments_of_Q_tilde},
then first
\begin{equation}
    \label{eq:copy_step_1}
    \begin{aligned}
    \E& [\vert D_{n} \vert^m \one_{\widetilde F_n^c \cap H_n} ] \le 3^{m-1} \Big(\frac{ C_1(m) n^{m\eps}}{\beta_3^m \vert S_n^k \vert^m} 
    + \frac{C_1(m)^2 n^{m\eps}}{\beta_3^m \vert S_n^k \vert^m} \\
    & \ +  \E\Big[ \Big \vert \sum_{Z \in \PP_n \cap \Q(x,R_n)} G(Z,\PP_n^y) \one\{\vert G(Z,\PP_n^y)- G(Z,\PP_n)\vert > \tfrac{\beta_3  \vert S_n^k \vert}{2} \vert\} \Big \vert^m  \one_{\widetilde F_n^c \cap H_n}\Big] \Big)
    \end{aligned}
\end{equation}
By definition of the non-stable cube and the event $H_n$,
$$
\vert G(Z, \PP_n^y)- G(Z,\PP_n)\vert \le \overline g_{\sup}(\PP_n) \PP(\Q(y,R_n)),
$$
By the Cauchy--Schwarz inequality (twice), Lemma \ref{lem:moments_of_Q_tilde} and Assumption \ref{ass:A}(iii),
\begin{equation}
    \label{eq:log_moments_term_5_pre_new}
    \begin{aligned}
\E\Big[& \Big \vert \sum_{Z \in \PP_n \cap \Q(x,R_n)} G(Z,\PP_n^y) \one\{\vert G(Z,\PP_n^y)- G(Z,\PP_n)\vert > \tfrac{\beta_3  \vert S_n^k \vert}{2} \vert\} \Big \vert^m  \one_{\widetilde F_n^c \cap H_n}\Big] \\
& \qquad \qquad \qquad \quad \le C_1(m) n^{m\eps} \vert  W_n \vert^{m}
     P\Big(\overline g_{\sup}(\PP_n) \PP(\Q(y,R_n)) > \tfrac{\beta_3  \vert S_n^k \vert}{2}\Big)^{1/2}.
    \end{aligned}
\end{equation}
Continuing exactly as in the proof of Lemma \ref{lem:log_moments_G}: By $\{ rs > t \} \su \{ r  > \sqrt{t} \} \cup \{ s > \sqrt{t} \}$, Markov's inequality, Assumption \ref{ass:A}(iii) and Lemma \ref{lem:moments_of_Q_tilde},
\begin{equation}
    \label{eq:prob_new}
P\big(\overline g_{\sup}(\PP_n) \PP(\Q(y,R_n)) > \tfrac{\beta_3  \vert S_n^k \vert}{2}\big)^{1/2} \le 
2 \Big(\frac{2}{\beta_3 \vert S_n^k \vert}\Big)^{2md} n^{2md\eps }.
\end{equation}
Hence, combining \eqref{eq:copy_step_1}, \eqref{eq:log_moments_term_5_pre_new} and \eqref{eq:prob_new}, then for any $n \gg 1$,
\begin{equation}
    \label{eq:copy_step_2}
    \begin{aligned}
    \E& [\vert D_{n} \vert^m \one_{\widetilde F_n^c \cap H_n} ] \le 3^{m-1} \Big( \frac{ C_1(m) n^{m\eps}}{\beta_3^m \vert S_n^k \vert^m} 
    + \frac{C_1(m)^2 n^{m\eps}}{\beta_3^m \vert S_n^k \vert^m} + \frac{2 C_1(m) n^{m\eps}}{\beta_3^{2md}\vert S_n^k \vert^m} \Big),
    \end{aligned}
\end{equation}
which yields \eqref{eq:compound_score_2_1}, when performing a change of variable in $\eps$. Note that \eqref{eq:compound_score_2_2} 
may be proven with exactly the same arguments as above and hence this completes the proof. 
\end{proof}
Finally, we also prove that the cross term $C_n$ vanishes, when $x$ and $y$ are not close to one another. This will allow
us to control its contribution in Case II in \eqref{eq:three_cases}.
\begin{lemma}[Less crude $4^{\text{th}}$ moment bound on $C_n$]
    \label{lem:C_n_zero}
For any $\eps > 0$ and $n \gg 1$, 
    $$
    \E \big[ C_n^4 \one_{\widetilde F_n^c \cap H_n}\big] \le \frac{n^{\eps}}{\vert S_n^k \vert^4}.
    $$
\end{lemma}
\begin{proof}
First of all, note that by Assumption \ref{ass:B}(i), it suffices to consider contributions from $Z$ inside $\Q(x,R_n) \cup \Q(y,R_n)$ in the cross-term $C_n$. 
If $Z \in \Q(x, R_n)$, then on $H_n$, it must be that $Z \notin \Q(y, R_n)$. Then, by Assumption \ref{ass:B}(i) once more, it follows that
$ Z \in \mc A_n^+(\PP_n^{xy})$ if and only if $Z \in \mc A_n^+(\PP_n^{x})$, and likewise that $Z \in \mc A_n^+(\PP_n^{y})$ if and only if $Z \in \mc A_n^+(\PP_n)$.
If instead $Z \in \Q(y, R_n)$, then by identical arguments, it holds by Assumption \ref{ass:B}(i) again that $ Z \in \mc A_n^+(\PP_n^{xy})$ if and only if $Z \in \mc A_n^+(\PP_n^{y})$, 
and likewise that $Z \in \mc A_n^+(\PP_n^{x})$ if and only if $Z \in \mc A_n^+(\PP_n)$. Thus, by pairing the terms in $C_n$ two and two according to these observations,
\begin{align*}
 \E \big[ C_n^4 \one_{\widetilde F_n^c \cap H_n}\big] & \le 4^3 \E \Big[ \Big\vert \sum_{Z \in \mc A_n^+(\PP_n^y) \cap \mc A_n^+(\PP_n) \cap \Q(x,R_n)} 
\log \Big(\frac{G(Z,\PP_n^y)}{G(Z,\PP_n)}\Big)\Big\vert^4 \one_{\widetilde F_n^c \cap H_n} \Big] \\
& \qquad + 4^3 \E \Big[ \Big\vert \sum_{Z \in \mc A_n^+(\PP_n^{xy}) \cap \mc A_n^+(\PP_n^x) \cap \Q(x,R_n)} 
\log \Big(\frac{G(Z,\PP_n^y)}{G(Z,\PP_n)}\Big)\Big\vert^4 \one_{\widetilde F_n^c \cap H_n} \Big] \\
& \qquad + 4^3 \E \Big[ \Big\vert \sum_{Z \in \mc A_n^+(\PP_n^x) \cap \mc A_n^+(\PP_n) \cap \Q(y,R_n)} 
\log \Big(\frac{G(Z,\PP_n^y)}{G(Z,\PP_n)}\Big)\Big\vert^4 \one_{\widetilde F_n^c \cap H_n} \Big] \\
& \qquad + 4^3 \E \Big[ \Big\vert \sum_{Z \in \mc A_n^+(\PP_n^{xy}) \cap \mc A_n^+(\PP_n^y) \cap \Q(y,R_n)} 
\log \Big(\frac{G(Z,\PP_n^y)}{G(Z,\PP_n)}\Big)\Big\vert^4 \one_{\widetilde F_n^c \cap H_n} \Big].
\end{align*}
Invoking Lemma \ref{lem:log_ratio_G_2} completes the proof.
\end{proof}

The final task is establishing one last moment bound that can be used inside the main term $M_n$ in Case II. It turns out, we need to control the size
of slab $S_n^k(Z,1)$ and to achieve this, we rely on the following Poisson concentration inequality \cite[Lemma 1.2]{Penrose2003_RandomGeometricGraphs}.
\begin{lemma}[Poisson concentration inequality]
    \label{lem:pois_conc}
For any Poisson random variable $X$ with mean $\ell$, then
$
P(X > 8 \ell) \le \exp\big( - \tfrac{\log(8)}{4} \ell \big).
$
\end{lemma}
\begin{lemma}[Compound score bounds; Part III]
    \label{lem:log_ratio_G_3}
    Let $m \ge 1$, $\eps > 0$. Under $\text{Assumption \ref{ass:B}}$, it holds for all $n \gg 1$,
\begin{align}
\sup_{x,y \in  \W_n}  \E\Big[\Big\vert\sum_{Z \in \mc U_n^{\leftrightarrow}(\PP_n^{xy})}
\log \Big( \frac{G(Z,\PP_n^{xy})G(Z,\PP_n)}{G(Z,\PP_n^x)G(Z,\PP_n^y)} \Big) \Big\vert^m \one_{\widetilde F_n^c \cap H_n} \Big] & \le \frac{n^{ \eps}}{\vert S_n^k \vert^m}, \label{eq:compound_score_2_3}
\end{align}
where $\mc U_n^{\leftrightarrow}(\PP_n^{xy}) = \mc U_n^+(\PP_n^{xy}) \cap \Q(x,R_n)^c \cap \Q(y,R_n)^c$.
\end{lemma}
\begin{proof}
First, we introduce the event 
$
E_n = \big\{ \sup_{Z \in \PP_n} \PP (S_n^k(Z,1)) \le 8 \vert S_n^k \vert \big\}.
$
By the union bound and Lemma \ref{lem:pois_conc}, then for $n \gg 1$,
\begin{equation}
    \label{eq:pois_conc_bound}
P(E_n^c) \le \vert  W_n \vert \exp(- \tfrac{\log(8)}{4} \vert S_n^k \vert ) \le \exp(- \tfrac{\log(8)}{8} \vert S_n^k \vert ).
\end{equation}
Hence since $G(Z,\PP_n^{x}),G(Z,\PP_n^{y}) \ge 1$ for any $Z \in \mc U_n^{\leftrightarrow}(\PP_n^{xy})$, then by the Cauchy--Schwarz inequality (twice), 
Assumption \ref{ass:A}(iii) and \eqref{eq:pois_conc_bound},
\begin{equation}
    \label{eq:pois_conc_bound_2}
\E\Big[\Big\vert\sum_{Z \in \mc U_n^{\leftrightarrow}(\PP_n^{xy})}
\log \Big( \frac{G(Z,\PP_n^{xy})G(Z,\PP_n)}{G(Z,\PP_n^x)G(Z,\PP_n^y)} \Big) \Big\vert^m \one_{\widetilde F_n^c \cap H_n \cap E_n^c} \Big] \le \exp(- \tfrac{\log(8)}{32} \vert S_n^k \vert ).
\end{equation}
Thus, we may focus on the event $E_n$ and for convenience, let $\widetilde E_n = \widetilde F_n^c \cap H_n \cap E_n$. Rewrite
$$
\frac{G(Z,\PP_n^{xy})G(Z,\PP_n)}{G(Z,\PP_n^x)G(Z,\PP_n^y)} = 1 + \widetilde a(Z),
$$
where 
$$
\widetilde a(Z) = \frac{G(Z,\PP_n^{xy})G(Z,\PP_n) - G(Z,\PP_n^x)G(Z,\PP_n^y)}{G(Z,\PP_n^x)G(Z,\PP_n^y)}.
$$
The idea is now to roughly follow the approach in the proof Lemma \ref{lem:log_ratio_G_2}: First note that on the event $\widetilde F_n^c$, then $\vert \widetilde a(Z) \vert > 1/2$
implies that 
\begin{equation}
    \label{eq:expanded_tilde_a}
\vert G(Z,\PP_n^{xy}) \vert \vert G(Z,\PP_n^x) - G(Z,\PP_n) \vert + \vert G(Z,\PP_n^x) \vert  \vert G(Z,\PP_n^{xy}) -  G(Z,\PP_n^y) \vert > \tfrac{\beta_3 \vert S_n^k \vert^2}{2}.
\end{equation}
Using that $Z \in \mc U_n^{\leftrightarrow}(\PP_n^{xy})$ and the event $H_n$ now implies that $g(Z,V,\PP_n^x) - g(Z,V,\PP_n) \neq 0$ only for $V \in \PP_n \cap \Q(x,R_n)$
and similarly that $ g(Z,V,\PP_n^{xy}) - g(Z,V,\PP_n^{y})  \neq 0$ for only $V \in \PP_n \cap \Q(y,R_n)$. Thus, on the event $E_n$ and by \eqref{eq:expanded_tilde_a}, 
$\vert \widetilde a(Z) \vert > 1/2$ further implies that
$$
8 g_{\sup}(\PP_n)^2 \max \big\{ \PP(\Q(x,R_n)),\PP(\Q(y,R_n)) \big\}  \vert S_n^k \vert > \tfrac{\beta_3 \vert S_n^k \vert^2}{4}.
$$
By cancellation of $\vert S_n^k \vert$ on either side, it thus follows by the Cauchy--Schwarz inequality and the union bound that
\begin{align*}
\E & \Big[\Big\vert\sum_{Z \in \mc U_n^{\leftrightarrow}(\PP_n^{xy})}
\log ( 1+ \widetilde a(Z)) \one\{ \vert \widetilde a(Z) \vert > 1/2\} \Big\vert^m \one_{\widetilde E_n} \Big]  \le  \\
& \quad \E\Big[\Big\vert\sum_{Z \in \mc U_n^{\leftrightarrow}(\PP_n^{xy})} 
\log ( 1+ \widetilde a(Z)) \Big\vert^{2m} \one_{\widetilde E_n} \Big]^{1/2} \\
& \qquad  \times \Big(P\big(g_{\sup}(\PP_n)^2 \PP(\Q(x,R_n)) > \tfrac{\beta_3 \vert S_n^k \vert}{32}\big) + P\big(g_{\sup}(\PP_n)^2 \PP(\Q(y,R_n)) > \tfrac{\beta_3 \vert S_n^k \vert}{32}\big) \Big)^{1/2}.
\end{align*}
Now, verbatim to the proof of Lemma \ref{lem:log_ratio_G_2}, it follows by Markov's inequality, Assumption \ref{ass:A}(iii) and Lemma \ref{lem:moments_of_Q_tilde},
then the probabilities can be bounded any chosen negative power of the volume of the slab $\vert S_n^k \vert$, which can dominate the first factor, and hence
\begin{equation}
    \label{lem:bound_three_a_large}
\E\Big[\Big\vert\sum_{Z \in \mc U_n^{\leftrightarrow}(\PP_n^{xy})}
\log ( 1+ \widetilde a(Z)) \one\{ \vert \widetilde a(Z) \vert > 1/2\} \Big\vert^m \one_{\widetilde E_n} \Big] \le \frac{n^{\eps}}{\vert S_n^k \vert^{m}}.
\end{equation}
If instead $\vert \widetilde a(Z) \vert \le 1/2$, then we can use the inequality $\vert \log(1+x) \vert \le 2 \vert x \vert$, and applying the same bound as on
the left-hand side of \eqref{eq:expanded_tilde_a} and following the arguments above with only contributions for the difference inside $\Q(x,R_n)$ and
$\Q(y,R_n)$, it follows that similarly to the case in Lemma \ref{lem:log_ratio_G_2},
\begin{equation}
    \label{lem:bound_three_a_small}
\E\Big[\Big\vert\sum_{Z \in \mc U_n^{\leftrightarrow}(\PP_n^{xy})}
\log ( 1+ \widetilde a(Z)) \one\{ \vert \widetilde a(Z) \vert \le 1/2\} \Big\vert^m \one_{\widetilde E_n} \Big] \le \frac{n^{\eps}}{\vert S_n^k \vert^{m}}.
\end{equation}
Combining \eqref{eq:pois_conc_bound_2}, \eqref{lem:bound_three_a_large} and \eqref{lem:bound_three_a_small} completes the proof.
\end{proof}

\subsection{Moment bounds on $D_{xy}^2\Sigma^{\log}_n$ and proof of Lemma \ref{lem:en1_sim}(i)-(ii)}
\label{sec:5.4}
In this section, we prove the bounds on the first two error terms $\widetilde I_{n,1}$ and $\widetilde I_{n,2}$, and hence we need to additionally control the fourth
moment of the second-order difference operator $D_{xy}\Sigma^{\log}_n(\PP_n)$. 
We now establish fourth moment bounds on $D_{xy}^2\Sigma^{\log}_n(\PP_n)$ based on each of three cases of $y$ in relation to $x$ as described in Section 5.3.
Recall that in Case I, there are many possible $y$ values and hence we need a quickly decaying bound in $n$. Here,
it is sufficient with an exponential factor in $n^\eps$.
\begin{lemma}[$4^{\text{th}}$ moment of $D_{xy}^2$; Case I]
    \label{lem:fourth_moment_first_case_log}
    Let $\eps > 0$ and $n \gg 1$. Under Assumption \ref{ass:B},
   $$
        \sup_{x \in  \W_n} \sup_{y \in \W_n \setminus \S_n^k(x,n^\eps) }\E[(D_{xy}^2\Sigma^{\log}_n(\PP_n))^4] \le n^{4\eps} \vert  W_n\vert^8 \e^{-\beta_1(\eps) n^{\eps} / 8}.
    $$
    \end{lemma}
    \begin{proof}
    First, we look at the main term $M_n$ in \eqref{eq:decomposition_D2_log}. Define the event,
$$
E_n = E_n(x,y) = \big\{R_n(x) + R_n(y)  > \vert \pi_1(x-y) \vert - 3 \big\}.
$$
We then follow the same overall approach as in the proof of Lemma \ref{lem:fourth_moment_first_case}: We claim that on the event $E_n^c$,
\begin{equation}
    \label{eq:fraction_one}
\frac{G(Z,\PP_n^{xy})G(Z,\PP_n)}{G(Z,\PP_n^x)G(Z,\PP_n^y)} = 1,
\end{equation}
for any $Z \in \mc U_n^+(\PP_n^{xy})$ and hence $M_n \one_{E_n^c} = 0$. Indeed, to prove this, suppose first that $\vert \pi_1(Z-x) \vert \le \vert \pi_1(Z-y) \vert$,
i.e., that $Z$ is closer to $x$ than to $y$ in the first coordinate, then on $E_n^c$, $\vert \pi_1(V - Z)\vert > 1$ for any $V \in \Q(y,R_n)$.
Hence, by definition of the non-stable cube and Assumption \ref{ass:A}(i), i.e., $k$-locality,
\begin{align*}
G(Z,\PP_n^{xy}) &= \sum_{V \in \PP_n^{xy} \cap \Q(y, R_n) } g(Z,V,\PP_n^{xy}) + \sum_{V \in \PP_n^{xy} \cap \Q(y, R_n)^c } g(Z,V,\PP_n^{xy}) \\
&= \sum_{V \in \PP_n^{x} \cap \Q(y, R_n) } g(Z,V,\PP_n^{x}) + \sum_{V \in \PP_n^{x} \cap \Q(y, R_n)^c } g(Z,V,\PP_n^{x}) = G(Z,\PP_n^{x}),
\end{align*}
and an analogous computation also shows that $G(Z,\PP_n^{y}) =  G(Z,\PP_n)$, which yields \eqref{eq:fraction_one}.
If instead $\vert \pi_1(Z-x) \vert > \vert \pi_1(Z-y) \vert$,
i.e., that $Z$ is closer to $y$ than to $x$ in the first coordinate, then on $E_n^c$, $\vert \pi_1(V - Z)\vert > 1$ for any $V \in \Q(x,R_n)$.
Thus, in this case, then $G(Z,\PP_n^{xy}) = G(Z,\PP_n^x)$ and $G(Z,\PP_n^{x}) = G(Z,\PP_n)$, which once more yields \eqref{eq:fraction_one}.
Next, by Assumption \ref{ass:A}(ii) and since $y \in \W_n \setminus \S_n^k(x,n^\eps)$, and hence that $\vert \pi_1( x-y) \vert  > n^\eps$,
\begin{equation}
    \label{eq:prob_e}
P(E_n) \le \e^{-\beta_1(\eps) n^\eps / 4}.
\end{equation}
Thus, combining this with Lemma \ref{lem:fourth_moment_main_term},
\begin{equation}
    \label{eq:a_e}
\E[M_n^4 \one_{E_n}] \le  n^{4 \eps} \vert  W_n \vert^{8}\e^{-\beta_1(\eps) n^\eps / 8}.
\end{equation}
Next we look at the cross-term $C_n$ in \eqref{eq:decomposition_D2_log}. Analogously, we claim that $C_n \one_{E_n^c} = 0$. Indeed, suppose first that
$Z \in \mc A_n^+(\PP_n^{xy})$ and that $Z \in \mc A_n^+(\PP_n^x)^c$. Then, by Assumption \ref{ass:B}(i), it follows that $Z \in \Q(x,R_n)$.
Hence,
$$
G(Z,\PP_n^{xy}) = \sum_{V \in \PP_n^{xy} \cap \Q(y, R_n) } g(Z,V,\PP_n^{xy}) + \sum_{V \in \PP_n^{xy} \cap \Q(y, R_n)^c } g(Z,V,\PP_n^{x})  = 0,
$$
since in the first sum, $\vert \pi_1(Z-W) \vert > 1$ on $E_n^c$, and in the second sum we use that $G(Z,\PP_n^x) = 0$ due to $Z \in \mc A_n^+(\PP_n^x)^c$.
If instead $Z \in \mc A_n^+(\PP_n^y)^c$, then we argue the same way with $\Q(y,R_n)$ instead of $\Q(x,R_n)$ to conclude $G(Z,\PP_n^{xy}) = 0$.
Finally, if $Z \in \mc A_n^+(\PP_n)^c \cap \mc A_n^+(\PP_n^{x}) \cap \mc A_n^+(\PP_n^{y})$, then it follows that $Z$ must lie in both $\Q(x,R_n)$ and $\Q(y,R_n)$,
but this is impossible, since these boxes are disjoint on $E_n^c$. Thus, we conclude that
$$
\mc A_n^+(\PP_n^{xy}) \cap \mc  U_n^+(\PP_n^{xy})^c = \emptyset.
$$
By similar reasoning, we thus conclude that this is also true for the other three sums in $C_n$. Hence $C_n \one_{E_n^c} = 0$.
Combining equations (\ref{eq:a_e}) and (\ref{eq:prob_e}) as well as using Lemma \ref{lem:fourth_moment_main_term} completes the proof.
\end{proof}
Next, in Case II, the number of $y$ values are approximately the order of size of the vertical slab, and hence the moment bound on $D_{xy}^2$ need to cancel this
by including $\vert S_n^k \vert$ in the denominator of the bound.
\begin{lemma}[$4^{\text{th}}$ moment of $D_{xy}^2$; Case II]
    \label{lem:fourth_moment_second_case_two_log}
    Let $\eps > 0$, $n \gg 1$. Under Assumption \ref{ass:B},
   $$
        \sup_{x \in  \W_n} \sup_{y \in \S_n^k(x,n^\eps) \setminus \Q(x,n^\eps)}\E[(D_{xy}^2\Sigma^{\log}_n(\PP_n))^4] \le \frac{n^{4\eps}}{\vert S_n^k \vert^4}.
    $$
\end{lemma}   
\begin{proof}
    First, we note that on the event $F_n$ as defined in \eqref{eq:exponential_concentration} and its measurable extension $\widetilde F_n$, then by Assumption \ref{ass:B}(ii) as well as the Cauchy--Schwarz inequality,
\begin{equation}
    \label{eq:case2_fn}
    \E[(D_{xy}^2\Sigma^{\log}_n(\PP_n))^4 \one_{\widetilde F_n}] \le \e^{-\beta_3(\eps) \vert S_n^k \vert / 4}.
\end{equation}
Thus, it suffices to show the bound on $\widetilde F_n^c$. Additionally, consider the event
$$
E_n = E_n(x,y) = \big\{R_n(x) + R_n(y) > n^\eps \big\}.
$$  
First, we look at the cross-term $C_n$ in \eqref{eq:decomposition_D2_log}. 
Note that on $E_n^c$, it follows that $\Q(x,R_n) \cap \Q(y,R_n) = \emptyset$ since $y \in \S_n^k(x,n^\eps) \setminus \Q(x,n^\eps)$, and hence by Lemma \ref{lem:C_n_zero}, 
then
$
\E[C_n^4 \one_{\widetilde F_n^c \cap E_n^c}] \le \frac{n^{4\eps}}{\vert S_n^k \vert^4}.
$ 
Combining this with Lemma \ref{lem:fourth_moment_main_term} and Assumption \ref{ass:B}(i),
\begin{equation}
    \label{eq:case2_cn}
\E[C_n^4] \le \frac{ n^{4\eps}}{\vert S_n^k \vert^4} + n^{4\eps} \sqrt{P(E_n)} \le \frac{2 n^{4\eps}}{\vert S_n^k \vert^4}
\end{equation}
for every $\eps > 0$ and $n \gg 1$.

Next consider the main term $M_n$ in \eqref{eq:decomposition_D2_log}. First, by Lemma \ref{lem:fourth_moment_main_term}
and Assumption \ref{ass:B}(i),
\begin{equation}
    \label{eq:case2_mn}
\E[M_n^4 \one_{E_n}] \le  n^{4 \eps} \vert  W_n \vert^{8} \e^{-\beta_2(\eps) n^\eps / 4} \le \e^{-\beta_2(\eps) n^\eps / 8}
\end{equation}
for every $\eps > 0$ and $n \gg 1$.
Next, we split $M_n$ into three sums based on whether $Z$ is
close to $x$, close to $y$, or far from both $x$ and $y$, i.e.,
\begin{align*}
M_n  = & \sum_{Z \in \mc U_n^+(\PP_n^{xy}) \cap \Q(x,R_n) } 
    \log \Big( \frac{G(Z,\PP_n^{xy})}{G(Z,\PP_n^x)} \Big) + \log \Big( \frac{G(Z,\PP_n)}{G(Z,\PP_n^y)}  \Big)\\
    & \qquad \qquad  + \sum_{Z \in \mc U_n^+(\PP_n^{xy})\cap \Q(y,R_n) } 
    \log \Big( \frac{G(Z,\PP_n^{xy})}{G(Z,\PP_n^y)} \Big) + \log \Big( \frac{G(Z,\PP_n)}{G(Z,\PP_n^x)}  \Big)\\
    &  \qquad \qquad  + \sum_{Z \in \mc U_n^+(\PP_n^{xy}) \cap \Q(x,R_n)^c \cap \Q(y,R_n)^c}\log \Big( \frac{G(Z,\PP_n^{xy})G(Z,\PP_n)}{G(Z,\PP_n^x)G(Z,\PP_n^y)} \Big).
\end{align*}
Letting $\widetilde J_n = \widetilde F_n^c \cap E_n^c$,
\begin{align*}
\E[ \vert M_n \vert^4 \one_{\widetilde  J_n}]  \le & 5^3 \E\Big[ \Big\vert \sum_{Z \in \mc U_n^+(\PP_n^{xy}) \cap \Q(x,R_n) } \log \Big( \frac{G(Z,\PP_n^{xy})}{G(Z,\PP_n^x)} \Big)\Big\vert^4 \one_{\widetilde J_n} \Big] \\  
    & \quad + 5^3 \E\Big[ \Big\vert  \sum_{Z \in \mc U_n^+(\PP_n^{xy}) \cap \Q(x,R_n) }  \log \Big( \frac{G(Z,\PP_n)}{G(Z,\PP_n^y)}  \Big) \Big\vert^4 \one_{\widetilde  J_n} \Big] \\
    & \quad  +  5^3 \E\Big[ \Big\vert \sum_{Z \in \mc U_n^+(\PP_n^{xy})\cap \Q(y,R_n) }  \log \Big( \frac{G(Z,\PP_n^{xy})}{G(Z,\PP_n^y)} \Big) \Big\vert^4 \one_{\widetilde J_n} \Big] \\
    & \quad + 5^3 \E\Big[ \Big\vert \sum_{Z \in \mc U_n^+(\PP_n^{xy})\cap \Q(y,R_n) }  \log \Big( \frac{G(Z,\PP_n)}{G(Z,\PP_n^x)}  \Big) \Big\vert^4 \one_{\widetilde  J_n}  \Big] \\
    & \quad +  5^3 \E\Big[ \Big\vert \sum_{Z \in \mc U_n^+(\PP_n^{xy}) \cap \Q(x,R_n)^c \cap \Q(y,R_n)^c}\log \Big( \frac{G(Z,\PP_n^{xy})G(Z,\PP_n)}{G(Z,\PP_n^x)G(Z,\PP_n^y)} \Big) \Big\vert^4 \one_{\widetilde  J_n} \Big] .
\end{align*}
Once again, since the event $E_n^c$ implies that $\Q(x,R_n) \cap \Q(y,R_n) \neq \emptyset$, then by Lemma \ref{lem:log_ratio_G_2},
\begin{equation}
    \label{eq:case2_mn_final}
\E[ \vert M_n \vert^4 \one_{\widetilde  J_n}]  \le  5^4 \frac{n^{4\eps}}{\vert S_n^k \vert^4}.
\end{equation}
Combining equations (\ref{eq:case2_fn}), (\ref{eq:case2_cn}), (\ref{eq:case2_mn}) and (\ref{eq:case2_mn_final}) completes the proof.
\end{proof}

Finally, in Case III, the remaining $y$-values are only of order $n^\eps$ and hence the steps involved can 
be cruder compared to Cases I and II. In particular, it will be sufficient to obtain a bound on the fourth moment $D_{xy}^2$ also of order $n^\eps$.
\begin{lemma}[$4^{\text{th}}$ moment of $D_{xy}^2$; Case III]
    \label{lem:fourth_moment_second_case_three_log}
    Let $\eps > 0$, $n \gg 1$. Under Assumption \ref{ass:B},
   $$
        \sup_{x \in  \W_n} \sup_{y \in \Q(x,n^\eps)}\E[(D_{xy}^2\Sigma(\PP_n))^4]
         \le n^{4\eps}.
    $$
\end{lemma} 
\begin{proof}
First, we consider the cross term $C_n$ in \eqref{eq:decomposition_D2_log}. By Lemma \ref{lem:fourth_moment_main_term} with $E_n = \Omega$, i.e., $E_n$ as the entire event space,
$
\E[C_n^4] \le n^{4\eps}.
$
Thus, we can focus on the main term $M_n$ in \eqref{eq:decomposition_D2_log}. Consider again $F_n$ as defined in \eqref{eq:exponential_concentration} 
and its measurable extension $\widetilde F_n$. Note as in the proof of Lemma \ref{lem:fourth_moment_second_case_three_log},
it is sufficient to consider contributions of $M_n$ on the event $\widetilde F_n^c$.
Hence,
\begin{equation}
    \label{eq:case3_mn}
\begin{aligned}
\E[ \vert M_n \vert^4 \one_{\widetilde  F_n^c}]  \le & 2^3 \E\Big[ \Big\vert \sum_{Z \in \mc U_n^+(\PP_n^{xy}) } \log \Big( \frac{G(Z,\PP_n^{xy})}{G(Z,\PP_n^x)} \Big)\Big\vert^4 \one_{\widetilde F_n^c} \Big] \\  
    & \qquad + 2^3 \E\Big[ \Big\vert  \sum_{Z \in \mc U_n^+(\PP_n^{xy}) }  \log \Big( \frac{G(Z,\PP_n)}{G(Z,\PP_n^y)}  \Big) \Big\vert^4 \one_{\widetilde  F_n^c} \Big].
\end{aligned}
\end{equation}
Combining \eqref{eq:case3_mn} with Lemma \ref{lem:log_moments_G} completes the proof.

\end{proof}
We can now prove the bounds on the first two error terms $\widetilde I_{n,1}$ and $\widetilde I_{n,2}$.
\begin{proof}[Proof of Lemma \ref{lem:en1_sim}(i)]
Recall we want to prove that
\begin{align*}
    \widetilde I_{n,1}  &= \int_{\W_n^3}
     \E\big[D_x \Sigma^{\log}_n(\PP_n)^2
    D_y \Sigma^{\log}_n(\PP_n)^2\big]^{1/2} \\
    & \qquad \qquad \qquad \qquad \times \E\big[D_{x,z}^2 \Sigma^{\log}_n(\PP_n)^2 
    D_{y,z}^2 \Sigma^{\log}_n(\PP_n)^2\big]^{1/2} \ \text d(x,y,z)  \le n^{\eps} \vert  W_n \vert.
\end{align*}
First, by applying the Cauchy--Schwarz inequality to both expectations,
\begin{align*}
    \widetilde I_{n,1} \le &  \int_{ \W_n}\Big(\int_{ \W_n}
    \sqrt[4]{\E\big[D_y \Sigma^{\log}_n(\PP_n)^4\big]}\sqrt[4]{\E\big[D_{xy}^2 \Sigma^{\log}_n(\PP_n)^4\big]}\text dy\Big)^2 \text dx.
\end{align*}
Next, by decomposing the inner integral into Cases I - III as defined in \eqref{eq:three_cases},
\begin{align*}
	\widetilde I_{n,1}  = &  \int_{ \W_n}\Big(\int_{\W_n \setminus \S_n^k(x,n^\eps) }\sqrt[4]{\E\big[D_y \Sigma^{\log}_n(\PP_n)^4\big]}\sqrt[4]{\E\big[D_{xy}^2 \Sigma^{\log}_n(\PP_n)^4\big]}\text dy\Big)^2 \text dx\\
    &  +    \int_{ \W_n}\Big(\int_{\S_n^k(x,n^\eps) \setminus \Q(x,n^\eps)}\sqrt[4]{\E\big[D_y \Sigma^{\log}_n(\PP_n)^4\big]}\sqrt[4]{\E\big[D_{xy}^2 \Sigma^{\log}_n(\PP_n)^4\big]}\text dy\Big)^2 \text dx\\
    &   +   \int_{ \W_n}\Big(\int_{\Q(x,n^\eps)}\sqrt[4]{\E\big[D_y \Sigma^{\log}_n(\PP_n)^4\big]}\sqrt[4]{\E\big[D_{xy}^2 \Sigma^{\log}_n(\PP_n)^4\big]}\text dy\Big)^2 \text dx,
\end{align*}
then by Lemmas \ref{lem:fourth_moment_log} and \ref{lem:fourth_moment_first_case_log} - \ref{lem:fourth_moment_second_case_three_log},
\begin{align*}
	\widetilde I_{n,1} \le & \int_{ \W_n}\Big(\int_{\W_n \setminus \S_n^k(x,n^\eps) } n^{2\eps} \vert  W_n\vert^2 \e^{-\beta_1(\eps) n^{\eps} / 32} \text dy\Big)^2 \text dx\\
    &  \qquad +  \int_{ \W_n}\Big(\int_{\S_n^k(x,n^\eps) \setminus \Q(x,n^\eps)} \frac{n^{2\eps}}{\vert S_n^k \vert} \text dy\Big)^2 \text dx + \int_{ \W_n}\Big(\int_{\Q(x,n^\eps)}  n^{2\eps} \text dy\Big)^2 \text dx.
\end{align*}
As the first integral goes to zero at exponential speed and using $\vert S_n^k(0,n^\eps) \vert = n^{k \eps} \vert S_n^k \vert$,
then for all $n \gg 1$,
$$
	\widetilde I_{n,1} \le 2 \vert  W_n \vert  n^{(4+2k) \eps} + 2 \vert  W_n \vert n^{(4+2d)\eps},
$$
so a change of variables $\eps' = (4+2d)\eps$ and invoking Lemma \ref{lem:normalization} completes the proof. 
\end{proof}

\begin{proof}[Proof of Lemma \ref{lem:en1_sim}(ii)]
Recall we want to prove that
    $$
    \widetilde I_{n,2} = \int_{\W_n^3} \E\big[D_{x,z}^2 \Sigma(\PP_n)^2
    D_{y,z}^2 \Sigma(\PP_n)^2\big] \ \text d(x,y,z)
     \le n^{ \eps } \vert  W_n \vert.
    $$
Following the steps in the proof of Lemma \ref{lem:en1_sim}(i): First, by the Cauchy--Schwarz inequality,
$$
\widetilde I_{n,2} \le \int_{ W_n}\Big(\int_{ W_n}\sqrt{\E\big[D_{xy}^2 \Sigma(\PP_n)^4\big]} \text dy\Big)^2 \text dx.
$$
By the decomposition in \eqref{eq:three_cases} as well as Lemmas \ref{lem:fourth_moment_log} and \ref{lem:fourth_moment_first_case_log} - \ref{lem:fourth_moment_second_case_three_log}, 
\begin{align*}
	\widetilde I_{n,2} \le & \int_{ \W_n}\Big(\int_{\W_n \setminus \S_n^k(x,n^\eps) } n^{2\eps} \vert  W_n\vert^4 \e^{-\beta_1(\eps) n^{\eps} / 16} \text dy\Big)^2 \text dx\\
    &  \qquad +  \int_{ \W_n}\Big(\int_{\S_n^k(x,n^\eps) \setminus \Q(x,n^\eps)} \frac{n^{2\eps}}{\vert S_n^k \vert^2} \text dy\Big)^2 \text dx + \int_{ \W_n}\Big(\int_{\Q(x,n^\eps)}  n^{2\eps} \text dy\Big)^2 \text dx.
\end{align*}
Once more, as the first integral goes to zero at exponential speed,
then for all $n \gg 1$,
$$
\widetilde I_{n,2} \le 2 \frac{\vert  W_n \vert  n^{(4+2k) \eps}}{\vert S_n^k \vert} + 2 \vert  W_n \vert n^{(4+2d)\eps},
$$
so a change of variables $\eps' = (4+2d)\eps$ and invoking Lemma \ref{lem:normalization} completes the proof. 
\end{proof}

\section*{Acknowledgments}
The authors would like to thank Lianne de Jonge for helpful discussions regarding the crossing numbers. Also, C.
Hirsch was supported by a research grant (VIL69126) from Villum Fonden and H. Döring was supported by the Deutsche Forschungsgemeinschaft (DFG, German Research Foundation) under Project-ID 531542011.
Additionally, part of the research was carried during N. Lundbye's stay at Osnabrück University, 
and he would like to extend his gratitude to its Department of Mathematics for their hospitality.

Generative AI tools were used to assist in drafting and LaTeX preparation of parts of the manuscript. All mathematical arguments were developed and verified manually.

\pagebreak

\addcontentsline{toc}{section}{References}
\bibliographystyle{abbrv}
\bibliography{./lit}

\begin{thebibliography}{10}

\bibitem{MR2292589}
F.~Baccelli and C.~Bordenave.
\newblock The radial spanning tree of a {P}oisson point process.
\newblock {\em Ann. Appl. Probab.}, 17(1):305--359, 2007.

\bibitem{MR2035772}
A.~G. Bhatt and R.~Roy.
\newblock On a random directed spanning tree.
\newblock {\em Adv. in Appl. Probab.}, 36(1):19--42, 2004.

\bibitem{chinmoy}
C.~Bhattacharjee and I.~Molchanov.
\newblock Gaussian approximation for sums of region-stabilizing scores.
\newblock {\em Electron. J. Probab.}, 27:Paper No. 111, 27, 2022.

\bibitem{BoissonnatChazalYvinec2018_GeometricTopologicalInference}
J.-D. Boissonnat, F.~Chazal, and M.~Yvinec.
\newblock {\em Geometric and Topological Inference}, volume~57 of {\em
  Cambridge Texts in Applied Mathematics}.
\newblock Cambridge University Press, 2018.

\bibitem{bouquet2024barcode}
A.~Bouquet and A.~R. Vindas-Mel{\'e}ndez.
\newblock Combinatorial results on barcode lattices.
\newblock {\em Order}, 42:193--209, 2025.

\bibitem{BruckGarin2023_StratifyingSpaceOfBarcodes}
B.~Br{\"u}ck and A.~Garin.
\newblock Stratifying the space of barcodes using {C}oxeter complexes.
\newblock {\em J. Appl. Comput. Topol.}, 7:369--395, 2023.

\bibitem{chatterjee2017cltnew}
S.~Chatterjee and P.~Diaconis.
\newblock A central limit theorem for a new statistic on permutations.
\newblock {\em Indian J. Pure Appl. Math.}, 48(7):561--573, 2017.

\bibitem{chimani2018crossing}
M.~Chimani, H.~D{\"o}ring, and M.~Reitzner.
\newblock Crossing numbers and stress of random graphs.
\newblock In T.~Biedl and A.~Kerren, editors, {\em Graph Drawing and Network
  Visualization (GD 2018)}, volume 11282 of {\em Lecture Notes in Computer
  Science}, pages 255--268. Springer, 2018.

\bibitem{curry2024trees2}
J.~Curry, J.~DeSha, A.~Garin, K.~Hess, L.~Kanari, and B.~Mallery.
\newblock From trees to barcodes and back again {II}: Combinatorial and
  probabilistic aspects of a topological inverse problem.
\newblock {\em Comput. Geom.}, 116:102031, 2024.

\bibitem{brochette}
H.~Duminil-Copin, M.~R. Hil{\'a}rio, G.~Kozma, and V.~Sidoravicius.
\newblock Brochette percolation.
\newblock {\em Israel J. Math.}, 225(1):479--501, 2018.

\bibitem{doring}
H.~D{\"u}ring and L.~de~Jonge.
\newblock Limit theorems for the number of crossings and stress in projections
  of the random geometric graph.
\newblock {\em arXiv preprint arXiv:2408.03218}, 2024.

\bibitem{Esseen1942_LiapunoffError}
C.-G. Esseen.
\newblock On the {L}iapunoff limit of error in the theory of probability.
\newblock {\em Arkiv f{\"o}r Mat. Astron. Fys. Ser. A}, 28(9):1--19, 1942.

\bibitem{FerrariLandimThorisson2004}
P.~A. Ferrari, C.~Landim, and H.~Thorisson.
\newblock {P}oisson trees, succession lines and coalescing random walks.
\newblock {\em Ann. Inst. Henri Poincar{\'e} Probab. Stat.}, 40(2):141--152,
  2004.

\bibitem{fulman2004stein}
J.~Fulman.
\newblock Stein's method and non-reversible {M}arkov chains.
\newblock In {\em Stein's Method: Expository Lectures and Applications},
  volume~5 of {\em Lecture Notes Series}, pages 69--77. World Scientific, 2004.

\bibitem{GareyJohnson1983_CrossingNumberNPComplete}
M.~R. Garey and D.~S. Johnson.
\newblock Crossing number is {NP}-complete.
\newblock {\em SIAM J. Algebraic Discrete Methods}, 4(3):312--316, 1983.

\bibitem{GracarEtAl2021_percolation}
P.~Gracar, M.~Heydenreich, C.~M{\"o}nch, and P.~M{\"o}rters.
\newblock Recurrence versus transience for weight-dependent random connection
  models.
\newblock {\em Electron. J. Probab.}, 27:1--31, 2022.

\bibitem{hilario}
M.~R. Hil{\'a}rio, M.~S{\'a}, and R.~Sanchis.
\newblock Strict inequality for bond percolation on a dilute lattice with
  columnar disorder.
\newblock {\em Stochastic Process. Appl.}, 149:60--74, 2022.

\bibitem{hoffmann}
C.~Hoffman.
\newblock Phase transition in dependent percolation.
\newblock {\em Comm. Math. Phys.}, 254(1):1--22, 2005.

\bibitem{hug}
D.~Hug, G.~Last, and M.~Schulte.
\newblock Boolean models in hyperbolic space.
\newblock {\em arXiv preprint arXiv:2408.03890}, 2024.

\bibitem{AD}
B.~Jahnel, S.~K. Jhawar, and A.~D. Vu.
\newblock Continuum percolation in a nonstabilizing environment.
\newblock {\em Electron. J. Probab.}, 28:1--38, 2023.

\bibitem{jaramillo2023combinatorial}
E.~Jaramillo-Rodriguez.
\newblock Combinatorial methods for barcode analysis, 2023.

\bibitem{adelie}
L.~Kanari, A.~Garin, and K.~Hess.
\newblock From trees to barcodes and back again: theoretical and statistical
  perspectives.
\newblock {\em Algorithms (Basel)}, 13(12):Paper No. 335, 27, 2020.

\bibitem{vares}
H.~Kesten, V.~Sidoravicius, and M.~E. Vares.
\newblock Oriented percolation in a random environment.
\newblock {\em Electron. J. Probab.}, 27:Paper No. 82, 49, 2022.

\bibitem{mal_stab}
R.~Lachi\`eze-Rey, M.~Schulte, and J.~E. Yukich.
\newblock Normal approximation for stabilizing functionals.
\newblock {\em Ann. Appl. Probab.}, 29(2):931--993, 2019.

\bibitem{nestmann}
G.~Last, F.~Nestmann, and M.~Schulte.
\newblock The random connection model and functions of edge-marked {P}oisson
  processes: second order properties and normal approximation.
\newblock {\em Ann. Appl. Probab.}, 31(1):128--168, 2021.

\bibitem{mehler}
G.~Last, G.~Peccati, and M.~Schulte.
\newblock Normal approximation on {P}oisson spaces: {M}ehler's formula, second
  order {P}oincar{\'e} inequalities and stabilization.
\newblock {\em Probab. Theory Related Fields}, 165:667--723, 2016.

\bibitem{Penrose2003_RandomGeometricGraphs}
M.~Penrose.
\newblock {\em Random Geometric Graphs}, volume~5 of {\em Oxford Studies in
  Probability}.
\newblock Oxford University Press, 2003.

\bibitem{yukCLT}
M.~D. Penrose and J.~E. Yukich.
\newblock Central limit theorems for some graphs in computational geometry.
\newblock {\em Ann. Appl. Probab.}, 11(4):1005--1041, 2001.

\bibitem{Pinsky2017_ConnectionsBetweenPermutationCyclesAndTouchardPolynomials}
R.~G. Pinsky.
\newblock Some connections between permutation cycles and {T}ouchard
  polynomials and between permutations that fix a set and covers of multisets.
\newblock {\em Electron. Commun. Probab.}, 22, 2017.

\bibitem{Ross2011_FundamentalsOfSteinsMethod}
N.~Ross.
\newblock Fundamentals of {S}tein's method.
\newblock {\em Probab. Surv.}, 8:210--293, 2011.

\bibitem{trauth3}
H.~Sambale, C.~Th{\"a}le, and T.~Trauthwein.
\newblock Central limit theorems for the nearest neighbour embracing graph in
  {E}uclidean and hyperbolic space.
\newblock {\em Stochastic Process. Appl.}, 188, 2025.

\bibitem{y3}
M.~Schulte and J.~E. Yukich.
\newblock Rates of multivariate normal approximation for statistics in
  geometric probability.
\newblock {\em Ann. Appl. Probab.}, 33(1):507--548, 2023.

\bibitem{stanley1997enumerative}
R.~P. Stanley.
\newblock {\em Enumerative Combinatorics, Vol. {I}}.
\newblock Cambridge University Press, 1997.

\bibitem{trauth2}
T.~Trauthwein.
\newblock Multivariate second-order $p$-{P}oincar{\'e} inequalities.
\newblock {\em arXiv preprint arXiv:2409.02843}, 2024.

\bibitem{trauth}
T.~Trauthwein.
\newblock Quantitative {CLT}s on the {P}oisson space via {S}korohod estimates
  and $p$-{P}oincar{\'e} inequalities.
\newblock {\em Ann. Appl. Probab.}, 2025.
\newblock To appear.

\bibitem{Villani2003_TopicsInOptimalTransportation}
C.~Villani.
\newblock {\em Topics in Optimal Transportation}, volume~58 of {\em Graduate
  Studies in Mathematics}.
\newblock American Mathematical Society, 2003.

\bibitem{Wasserman2004_AllOfStatistics}
L.~Wasserman.
\newblock {\em All of Statistics: A Concise Course in Statistical Inference}.
\newblock Springer Texts in Statistics. Springer, 2004.

\bibitem{gibbsCLT}
A.~Xia and J.~E. Yukich.
\newblock Normal approximation for statistics of {G}ibbsian input in geometric
  probability.
\newblock {\em Adv. in Appl. Probab.}, 47(4):934--972, 2015.

\end{thebibliography}

\pagebreak 

\appendix
\section{Appendix: Remaining proofs from Section \ref{sec:3}}
\label{sec:A}
This appendix is dedicated to proving the technical, but basic lemmas that were stated and applied in Section 3. 
First, we prove the Poisson ball probability bound in Lemma \ref{lem:poisson_ball_bound} implies that $f$ has sub-polynomial moments, i.e., satisfies Assumption \ref{ass:A}(iii).
\begin{proof}[Proof of Lemma \ref{lem:poisson_ball_bound}]
    Recall, we want to prove that
    $
\E[\overline f_{\sup}(\PP_n)^m] \le n^{m \eps},$
    First, by Markov's inequality and the bound on $f$ in (\ref{eq:f_bound}), we have that
    \begin{equation}
        \label{eq:markov_f_bound}
    P(\overline f_{\sup}(\PP_n) > s) 
    \le  s^{-2} M_n,
    \end{equation}
    where
    $$
    M_n =  \E\Big[\Big(\sup_{\mathrm{x}\in \mc R } 
    \sup_{Z,V \in \PP_n^\mathrm{x}} \PP^\mathrm{x}(B(\dot Z, \ell R_Z)) \cdot \PP^\mathrm{x}(B(\dot V, \ell R_V))\Big)^2\Big].
    $$
    By Tonelli's theorem, 
    $$
    M_n = \sum_{Z,V = 1}^\infty \E\Big[\Big( \sup_{\mathrm{x}\in \mc R } 
    \sup_{Z,V \in \PP_n^\mathrm{x}} \PP^\mathrm{x}(B(\dot Z, \ell z)) \cdot \PP^\mathrm{x}(B(\dot V, \ell w)) \Big)^2  \Big]
    P\big(R_Z = z, R_V = v \big).
    $$
    Define for any $r \in \N$, the fixed points $x_{r,1}, \ldots, x_{r,N_{n,r}} \in  W_n$
    such that
    $
     W_n \subseteq \bigcup_{i=1}^{N_{n,r}} B(x_{r,i}, r),
    $
    and assume without loss of generality that there exists $C>0$ such that
    $
    N_{n,r} \le C \frac{\vert W_n \vert}{\pi r^d}.
    $
    By the triangle inequality, it follows that 
    $$
    \PP^\mathrm{x}(B(\dot Z,z)) \le \PP^\mathrm{x}(B(x_{z,i},2 \ell z)) \quad \text{ and } \quad \PP^\mathrm{x}(B(\dot W,w)) \le \PP^\mathrm{x}(B(x_{w,j},2\lambda w))
    $$
    for some $i \in \{1,\ldots,N_{n,z}\}$ and $j \in \{1,\ldots,N_{n,w}\}$ . Hence it follows that
    \begin{align*}
    M_n \le
    \sum_{Z,V = 1}^\infty  \E\Big[\Big( \max_{i = 1,\ldots, N_{n,z} }\max_{j = 1,\ldots, N_{n,w} }
    \big (\PP(B(x_{z,i},2 \ell z)) + 2 \big) \cdot & \big( \PP(B(x_{w,j},2 \ell w)) +2 \big) \Big)^2 \Big] \\
    & \qquad \times P\big(R_Z = z, R_W = w \big).
    \end{align*}
    By the union bound, the Cauchy--Schwarz inequality (twice) and stationarity of $\PP$,
    \begin{align*}
    M_n 
    \le 
	    \sum_{Z,V = 1}^\infty  &N_{n,z} N_{n,w} \sqrt{\E\big[(\PP(B(0,2 \ell z))+2)^4 \big]}
     \sqrt{\E\big[ (\PP(B(0,2 \ell w))+2 )^4 \big]} \\
    &\times \sqrt{P(R_Z = z)}\sqrt{P(R_W = w)}.
    \end{align*}
    By Lemma \ref{lem:poisson_moment_bound}, it follows that
    $$
    M_n 
    \le 
    \sum_{Z,V = 1}^\infty  C' \vert W_n \vert^2
    \sqrt{P(R_Z = z)}\sqrt{P(R_W = w)}
    $$
    for some $C' >0$. By the assumption on the tail of $R_Z$ and $R_W$ and the convergence radius of a geometric series, it follows that
    $
    M_n 
    \le 
    C'' \vert W_n \vert^2 \e^{-\gamma(\eps) n^\eps}.
    $
    for some $C'' >0$. Rewrite
    $$
        \E\big[\overline f_{\sup}(\PP_n)^{m}\big]
         = \int_0^{n^{m\eps}} P\big(\overline f_{\sup}(\PP_n)^{m} > r\big) \ \text dr + \int_{n^{m\eps}}^{\infty} P\big(\overline f_{\sup}(\PP_n)^{m} > r\big) \ \text dr.
    $$
    Together with \eqref{eq:markov_f_bound}  and Lemma \ref{lem:normalization}, we obtain that
        \begin{equation}
            \label{eq:sup_f_bound}
            \E\big[\overline f_{\sup}(\PP_n)^{m}\big]
            \le n^{m \eps} + C'' \vert W_n \vert^2 \e^{-\gamma(\eps) n^\eps}  \int_{n^{m\eps}}^{\infty} \frac{1}{r^2} \ \text dr
            \le n^{m \eps},
            \end{equation}
        for all $n \gg 1$, which completes the proof.
    \end{proof}

       Next, we turn to proving the variance bound in Lemma \ref{lem:var_bound}. To that end, we first prove a more general bound
    that we then can apply in our setting. Here
a modified version of the lower-bound found in \cite[Lemma 2.3]{gibbsCLT}.

    \begin{lemma}[General variance bound]
    \label{lem:var_lower_bound}
    Let $U_1$ and $U_2$ denote independent random variables with values in a measurable space $(\Lambda, \mathcal{E})$.
    For fixed $\mathcal{E}$-measurable sets $A$, $B$, $\widetilde B \su \Lambda$, let $h : \Lambda \times \Lambda \rightarrow [0,\infty)$
    denote a Borel measurable function
    and define
    \[
        \Delta = \inf_{(x, y, \widetilde y) \in A \times B \times \widetilde B} |h(x, y) - h(x, \widetilde y)|.
    \]
    Then,
    $
    \V[h(U_1, U_2) | \sigma(U_1)] \ge \frac{\Delta^2}{4} \Big(P(U_2 \in B) \land P(U_2 \in \widetilde B)\Big).
    $
\end{lemma}

    \begin{proof}[Proof of Lemma \ref{lem:var_lower_bound}]
        Let $E = \{U_2 \in B\} \cup \{U_2 \in \widetilde B\}$. Then, under the conditional measure $P(\cdot | U_1)$ and the law of total variance,
        \begin{align*}
        \V(h(U_1, U_2))  = & \mathbb{E}[\V(h(U_1, U_2)|E)] + \V(\mathbb{E}[h(U_1, U_2)|E]) \ge  \V(h(U_1, U_2)|E ) \cdot P(E).
        \end{align*}
        Hence by the definition of variance and the assumed independence,
        \begin{equation}
            \label{eq:def_var_input}
        \begin{aligned}
        \V(h(U_1, U_2) | U_1) \ge \mathbb{E}\big[(h(U_1, U_2) - \mathbb{E}\big [h(U_1, U_2)| U_1, E])^2 | U_1, E\big] \big(P(U_2 \in B) \wedge P(U_2 \in \widetilde B)\big).
        \end{aligned}
        \end{equation}
        On the event $E$, it follows by case-splitting, that
       $$
       \vert h(U_1, U_2)  - E[h(U_1, U_2) | U_1, E]\vert  \ge \frac{\Delta^2}{2}.
       $$
    Plugging this into (\ref{eq:def_var_input}) completes the proof.
\end{proof}

For bounding the variance of $\Sigma(\PP_n)$ and $\Sigma^{\log}_n(\PP_n)$, it will be useful to consider a martingale decomposition of these functionals in terms of the information contained in each of these boxes.
Recall that for a Borel set $A \subseteq \R^d$, let $\mathfrak{N}(A)$ denote the locally finite counting measures on $A \times \M$, and define for $1 \le j \le \alpha_r n^d$ the quantities
\begin{align*}
\mathfrak N_{n,j}^-  = \mathfrak{N}\Big(\bigcup_{i=1}^{j-1} Q_{n,i,r}\Big), \quad \text{ and } \quad
\mathfrak N_{n,j}^+  = \mathfrak{N}\Big(\bigcup_{i=j+1}^{\alpha_r n^d} Q_{n,i,r}\Big).
\end{align*}

\begin{proof}[Proof of Lemma \ref{lem:var_bound}]
First, we show that for some $C > 0$, $\V[\Sigma(\PP_n)] \ge C n^{3d-2k}$ for any $n \gg 1$.
Let $\Q_{n,i,r} = Q_{n,i,r} \times \M$ and define the filtration $\mathfrak{G} = (\mc G_{n,j})_{0\le j \le \alpha_r n^d} $ as $\mc G_{n,0} = \{\emptyset, \Omega\}$ and for $1 \le j \le \alpha_r n^d$,
$
\mc G_{n,j} = \sigma\big(\PP \cap \bigcup_{i=1}^j \Q_{n,i,r} \big).
$
Rewriting $\Sigma(\PP_n)$ as as a telescoping series of (orthogonal) square integrable martingales with respect to $\mathfrak{G}$, then
$$
\V[\Sigma(\PP_n)] = \sum_{j = 1}^{\alpha_r n^d} \V\Big[\mathbb{E}[\Sigma(\PP_n) | \mc G_{n,j}] - \mathbb{E}[\Sigma(\PP_n) | \mc G_{n,j-1}]\Big],
$$
and subsequently using the conditional law of total variance,
\begin{equation}
    \label{eq:var_decomposition}
\V[\Sigma(\PP_n)] = \sum_{j = 1}^{n^d} \E\Big[\V \Big[\E[\Sigma(\PP_n) \vert \mc G_{n,j}]\Big\vert \mc G_{n,j-1}\Big]\Big].
\end{equation}
Now define the function $h:\mathfrak N_{n,j}^- \times \mathfrak N(Q_{n,j,r}) \to [0,\infty)$ as
\begin{equation}
    \label{eq:h_function_def}
h(\omega_1, \omega_2) = \int_{\mathfrak N_{n,j}^+} \Sigma(\omega_1, \omega_2, \omega_3) P_{\PP(\cdot \cap \mathfrak N_{n,j}^+)}(\text{d}\omega_3),
\end{equation}
and let
$$
\Delta_n = \inf_{(\omega_1, \omega_2,\widetilde \omega_2) 
\in \mathfrak N_{n,j}^- \times \mc B_{n,j}^{(1)} \times \mc B_{n,j}^{(2)} }
\big\vert h(\omega_1, \omega_2) - h(\omega_1, \widetilde \omega_2) \big\vert.
$$
Thus, applying Lemma \ref{lem:var_lower_bound} with $h$ in \eqref{eq:h_function_def} and
\begin{align*}
U_{n,1} = \PP \cap \bigcup_{i=1}^{j-1} \Q_{n,i,r}, \quad
 U_{n,2} = \PP \cap \Q_{n,j,r}, \quad 
 U_{n,3} = \PP \cap \bigcup_{i=j+1}^{\alpha_r n^d} \Q_{n,i,r},
\end{align*}
it follows from (\ref{eq:var_decomposition}) that
$$
\V[\Sigma(\PP_n)] \ge \sum_{j=1}^{\alpha_r n^d} \frac{\E[\Delta_n^2]}{4}  \Big(P\big(\PP \cap \Q_{n,j,r} \in \mc B_{n,j}^{(1)}\big) \land P\big(\PP \cap \Q_{n,j,r} \in \mc B_{n,j}^{(2)}\big)\Big).
$$
By condition (V2) and Jensen's inequality, 
$$
\E[\Delta_n^2] \ge \E\Big[\# \big\{ i \in I_{n,j}^k \colon \PP \cap \Q_{n,i,r}\in \mc B_{n,i}^{(3)} \big\}\Big]^2
$$
Hence, by the Mecke formula and condition (V1), then there exists $\widetilde C > 0$ such that
$
\V[\Sigma(\PP_n)] \ge \widetilde C\alpha_r n^d  \vert I_{n,j}^k \vert^2.
$
Thus, using that $\vert I_{n,j}^k \vert$ is of order $n^{d-k}$,
we conclude there exists some $C > 0$ such that $\V[\Sigma(\PP_n)] \ge C n^{3d-2k}$
for any $n \gg 1$.

Next, we prove that for some $C > 0$,  $\V[\Sigma_n^{\log}(\PP_n)] \ge C n^d $ for any $n \gg 1$.
Consider the same filtration
$\mathfrak G$  and the same
variance decomposition as in \eqref{eq:var_decomposition},
\[
\V[\Sigma_n^{\log}(\PP_n)]
= \sum_{j = 1}^{\alpha_r n^d}
\E\Big[\V \Big[\E[\Sigma_n^{\log}(\PP_n) \vert \mc G_{n,j}]
\Big\vert \mc G_{n,j-1}\Big]\Big],
\]
and note that once more all summands are nonnegative.
Define $h^{\log}$ exactly as in \eqref{eq:h_function_def}, but with
the double-sum functional $\Sigma$ replaced by sum-log-sum functional $\Sigma_n^{\log}$, and define $\Delta_n^{\log}$ by
\[
\Delta_n^{\log}
= \inf_{(\omega_1, \omega_2,\widetilde \omega_2) 
\in \mathfrak N_{n,j}^- \times \mc B_{n,j}^{(1)} \times \mc B_{n,j}^{(2)} }
\big\vert h^{\log}(\omega_1, \omega_2)
      - h^{\log}(\omega_1, \widetilde \omega_2) \big\vert.
\]
Applying Lemma~\ref{lem:var_lower_bound} with $h^{\log}$ and the same
$(U_{n,1},U_{n,2},U_{n,3})$ as above yields
\begin{equation}
    \label{eq:var_log_decomposition}
\V[\Sigma_n^{\log}(\PP_n)]
\ge \sum_{j=1}^{\alpha_r n^d}
\frac{\E[(\Delta_n^{\log})^2]}{4}
\Big(
P(\PP \cap \Q_{n,j,r} \in \mc B_{n,j}^{(1)})
\land
P(\PP \cap \Q_{n,j,r} \in \mc B_{n,j}^{(2)})
\Big).
\end{equation}
On the event $\PP\cap\Q_{n,j,r}\in \mc B_{n,j}^{(1)}$, the quantity inside the
logarithm in $\Sigma_n^{\log}$ is the number of inversions created in the local
configuration. Now, this number is bounded below by
\[
 \#\big\{ i \in I_{n,j}^k \colon
\PP \cap \Q_{n,i,r}\in \mc B_{n,i}^{(3)} \big\}.
\]
On $\PP\cap\Q_{n,j,r}\in \mc B_{n,j}^{(2)}$, the same quantity inside the logarithm yields a
smaller number of inversions,
and hence
$$
\Delta_n^{\log}
\ge \log\bigg(1 + \#\big\{ i \in I_{n,j}^k \colon
\PP \cap \Q_{n,i,r}\in \mc B_{n,i}^{(3)} \big\}\bigg).
$$
Now, since the number of cubes configured as in $B_{n,\cdot }^{(3)}$ is stochastically dominated by a
binomial random variable with number trials $|I_{n,j}^k|$ and succes probability $q>0$ independent of
$n$ and $j$, then by Lemma \ref{lem:binom_conc},
$$
\P\bigg(\#\big\{ i \in I_{n,j}^k \colon
\PP \cap \Q_{n,i,r}\in \mc B_{n,i}^{(3)} \big\}\ge \tfrac{q|I_{n,j}^k|}{2}\bigg) \ge 1 - \exp\!\Big( - q |I_{n,j}^k|\Big(\tfrac{1}{2}+\tfrac{1}{2}\log(\tfrac{1}{2})\Big)\Big) > C_0
$$
for some $C_0>0$ and 
any $n \gg 1$. Thus by Jensen's inequality,
\begin{equation}
    \label{eq:var_log_lower_bound}
\E[(\Delta_n^{\log})^2]
\ge \E\bigg[\log \bigg(1+\#\big\{ i \in I_{n,j}^k \colon
\PP \cap \Q_{n,i,r}\in \mc B_{n,i}^{(3)} \big\}\bigg)\bigg]^2
\ge C_0 \log\Bigg(1+\frac{q |I_{n,j}^k|}{2} \Bigg)^2.
\end{equation}
Since $|I_{n,j}^k|$ is of order $n^{d-k}$, then combining \eqref{eq:var_log_decomposition}
and \eqref{eq:var_log_lower_bound} completes the proof.
\end{proof}

\bep[Proof of Lemma \ref{lem:shield}]
Recall we want to construct a set $\mc S_{n,j} \subseteq \mathfrak{N}(Q_{n,j,4} \setminus Q_{n,j,4}(1/2))$ such that 
if $\PP \cap(Q_{n,j,4} \setminus Q_{n,j,4}(1/2))\in \mc S_{n,j}$, then
$$
f(Z,V,\PP_n^x) = f(Z,V,\PP_n)
$$
for any $Z,V \in \PP \cap Q_{n,j,4}^c$ and $x \in Q_{n,j,4}(1/2)$. To that end, let $F_{n,j}^-$ denote the subset of $Q_{n,j,4}$, where when
$Q_{n,j,4}$ is centered at the origin, the first coordinate is in $[-4,-7/2]$, and similarly let $F_{n,j}^+$ denote the subset of $Q_{n,j,4}$,
where the first coordinate is in $[7/2,4]$. Moreover, $F_{n,j}^{\pm \uparrow}$ denote the subset of $F_{n,j}^\pm$, where (when $Q_{n,j,4}$ is shifted to the origin) the
$d$'th coordinate is in $[7/2,4]$. Then, we define the shield configuration set as
\begin{align*}
\mc S_{n,j} = \Big\{D \in \mathfrak{N}(Q_{n,j,4} \setminus Q_{n,j,4}(1/2)) \colon & D \cap (F_{n,j}^- \cup F_{n,j}^+)^c = \emptyset, \ \bigcup_{p \in D} B(p,1/4) \subseteq F_{n,j}^- \cup F_{n,j}^+, \\
& \ p_1^+ - p_1 \leq 1/2 \text{ for every } p \in D \cap (F_{n,j}^- \setminus F_{n,j}^{- \uparrow}), \\
& \ p_1 - p_1^- \leq 1/2 \text{ for every } p \in D \cap (F_{n,j}^+ \setminus F_{n,j}^{+ \uparrow}) \Big\},
\end{align*}
and let
\begin{equation}
    \label{eq:B0_def}
\mc B_{n,j}^{(0)} = \{ D \in \mathfrak{N}(Q_{n,j,4}) : D \cap(Q_{n,j,4} \setminus Q_{n,j,4}(1/2))\in \mc S_{n,j}\},
\end{equation}
see Figure \ref{fig:shield}. In other words, the shield $\mc S_{n,j}$ pads the left and right of $Q_{n,j,4}$ in the time direction densely with points
and leaves the rest of $Q_{n,j,4} \setminus Q_{n,j,4}(1/2)$ empty. Since neither the Poisson void probability of $(F_{n,j}^- \cup F_{n,j}^+)^c$ nor the volume of the balls $ B(0,1/4) $ depend on $n$ that
$
\inf_{n \ge 1} P(\PP \cap \Q_{n,j,4} \in \mc B_{n,j}^{(0)}) > 0.
$
Thus, it only remains to argue that the shield configurations indeed satisfies that insertions inside $Q_{n,j,4}(1/2)$ do not affect brances outside $Q_{n,j,4}$.
First the void condtion in the definition of $\mc S_{n,j}$ ensures that no branches can cross from outside $Q_{n,j,4}$ from the top or bottom (in spatial sense) 
into $Q_{n,j,4}(1/2)$ due to spatial constraint of 1 in the Poisson tree model.
From the condition that the shield pads the left and right of $Q_{n,j,4}$ densely with points, it follows that any branch that starts
to the left of $Q_{n,j,4}$ or inside $F_{n,j}^{-}$ (due to sucessor bound $p_1^+ - p_1 \leq 1/2$) cannot have a succesor inside $Q_{n,j,4}(1/2)$.
Similarly, any branch that starts to the right of $Q_{n,j,4}$ or inside $F_{n,j}^{+}$ (due to predecessor bound $p_1 - p_1^- \leq 1/2$) cannot have a ancestor inside $Q_{n,j,4}(1/2)$.
Thus, the only branches that can be affected by insertions inside $Q_{n,j,4}(1/2)$ are those that both start and end inside $Q_{n,j,4}$. However any changes to such branches
cannot affect the inversion from a branch starting outside $Q_{n,j,4}$ to a branch to another branch starting outside $Q_{n,j,4}$,
or in other words, $f(Z,V,\PP_n^x) = f(Z,V,\PP_n)$ for any $Z,V \in \PP \cap Q_{n,j,4}^c$ and $x \in Q_{n,j,4}(1/2)$ as claimed, which completes the proof.
\end{proof}

\begin{figure}[h!]
    \centering
    \includegraphics[width=0.5\textwidth]{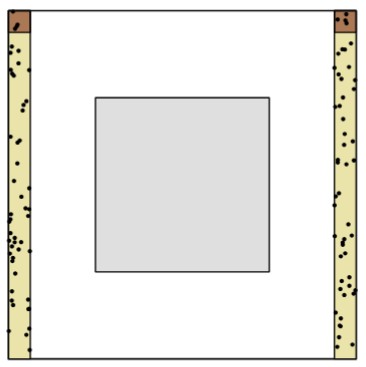}
    \caption{Illustration of the shield configuration in \eqref{eq:B0_def} when $d=2$.
    The yellow strip to the left to and right of the square represents the regions $F_{n,j}^-$ and $F_{n,j}^+$ padded with Poisson points,
    while the brown areas at the top of these strips represent the regions $F_{n,j}^{- \uparrow}$ and $F_{n,j}^{+ \uparrow}$.
    The white region does not contain any Poisson points and represents the void region in the shield configuration.
    The grey square in the center represents the inner cube $Q_{n,j,4}(1/2)$ where insertions do not affect branches outside the outer square $Q_{n,j,4}$.}
    \label{fig:shield}

\end{figure}

\end{spacing}
\end{document}